\documentclass[leqno,11pt,a4paper]{amsart}
\usepackage{bw_jc}
\usepackage{tikz-cd}
\usepackage{comment}
\usepackage{soul}
\usepackage{todonotes}
\usepackage{amsmath}
\usepackage{mathtools,amsfonts,stmaryrd}
\usepackage{color}
\usepackage{enumitem}

\newcommand{\s}{\mathfrak{s}}

\newcommand{\rsftgap}{\text{gap}}
\newcommand{\gap}{\text{gap}}

\newcommand{\btikz}[2]{\begin{figure}[h]\caption{#1}
\vspace{3pt}
\label{#2}\begin{center}\begin{tikzpicture}}
\newcommand{\etikz}{\end{tikzpicture}\end{center}\end{figure}}

\newcommand{\cM}{\mathcal{M}}

\newcommand{\D}{{D}}

\newcommand{\indCZ}{\operatorname{CZ}}
\newcommand{\indRS}{\operatorname{RS}}
\newcommand{\ind}{\operatorname{ind}}

\newcommand{\lcm}{\operatorname{lcm}}
\newcommand{\moduli}{\overline{\cM}(Y, J_f; \gamma_+,\gamma_-;z)}
\newcommand{\Umoduli}{\cM_{X \setminus E}}
\newcommand{\Emoduli}{\cM_E}

\newcommand{\id}{\on{Id}}
\newcommand{\T}{\mathbb{T}}
\graphicspath{{images/}}
\begin{document}
\title[Contact Homology and Higher Dimensional Closing Lemmas]{Contact Homology and Higher Dimensional Closing Lemmas}

\author{J.~Chaidez}
\address{Department of Mathematics\\Princeton University\\Princeton, NJ\\08544\\USA}
\email{jchaidez@princeton.edu}

\author{I. Datta}
\address{School of Mathematics\\Institute for Advanced Study\\Princeton, NJ\\08540\\USA}
\email{ipsi@ias.edu}

\author{R. Prasad}
\address{Department of Mathematics\\Princeton University\\Princeton, NJ\\08544\\USA}
\email{rrprasad@princeton.edu}

\author{S. Tanny}
\address{School of Mathematics\\Institute for Advanced Study\\Princeton, NJ\\08540\\USA}
\email{tanny.shira@gmail.com}
\maketitle
\begin{abstract} We develop methods for studying  the smooth closing lemma for Reeb flows in any dimension using contact homology. As an application, we prove a conjecture of Irie, stating that the strong closing lemma holds for Reeb flows on ellipsoids. Our methods also apply to other Reeb flows, and we illustrate this for a class of examples introduced by Albers-Geiges-Zehmisch. 
\end{abstract}

 
 \tableofcontents
\section{Introduction} \label{sec:introduction} In \cite{p1967}, Pugh proved a fundamental property of the periodic orbits of $C^1$ dynamical systems, called the \emph{$C^1$ closing lemma}. Morally speaking, Pugh's closing lemma states that nearly periodic points become periodic after a slight $C^1$-perturbation of the dynamical system. Precisely, it is stated as follows.

\begin{thm*} \label{thm:pugh_closing_lemma} \cite{p1967} Let $V$ be a $C^1$ vector-field on a closed manifold $X$ and let $x$ be a non-wandering point of $V$. Then there is a $C^1$ vector-field $V'$ that is $C^1$-close to $V$ such that $x$ is on a closed orbit of $V'$.
\end{thm*}

\noindent Note that a point $x$ is non-wandering if, for any open neighborhood $U$ of $x$ and every time $S$, there exists a $T \ge S$ such that $\phi^T(U) \cap U$ is non-empty, where $\phi$ is the flow of $V$. If $\phi$ preserves a volume form, then every point is non-wandering. 

\vspace{3pt}

A key question in smooth dynamical systems (posed, for instance, by Smale \cite{s1998}) is whether or not Theorem \ref{thm:pugh_closing_lemma} extends to the $C^\infty$ setting. However, in \cite{g1987} Gutierrez proved that any compact manifold $X$ containing an embedded punctured torus $\Sigma \subset X$ possesses a vector-field $V$ and a non-wandering point $p$ that does \emph{not} become periodic under any small $C^\infty$-perturbation of $V$. This follows the work of Herman \cite{h1979} in the Hamiltonian setting, where he proved that all sufficiently small smooth Hamiltonian perturbations of a Diophantine rotation $\phi$ of the two-torus $\mathbb{T}^2$ have no periodic orbits. These negative results suggested that, for general smooth diffeomorphisms and flows, there is no analogue of the closing lemma.  

\vspace{3pt}

Recently, dramatic progress has been made for area preserving diffeomorphisms of surfaces and Reeb flows of contact $3$-manifolds . In \cite{i2015}, Irie used spectral invariants coming from embedded contact homology (ECH) \cite{h2010} to prove the $C^\infty$ closing lemma for Reeb flows on closed contact $3$-manifolds. In fact, Irie's proof implied a \emph{strong closing property}.

\begin{definition}[\cite{irie2022strong}] \label{def:strong_closing_property} A manifold $Y$ with contact form $\alpha$ satisfies the \emph{strong closing property} if, for any non-zero smooth function
\[f:Y \to [0,\infty)\]
there is a $t \in [0,1]$ such that $(1 + tf)\alpha$ has a closed Reeb orbit passing through the support of $f$. \end{definition}

\begin{thm*} \cite{i2015} \label{thm:strong_closing_property} Every closed $3$-manifold with contact form $(Y,\alpha)$ has the strong closing property.
\end{thm*}

\noindent Strong versions of the closing lemma were later proven for Hamiltonian surface maps \cite{ai2016} and more generally, area preserving surface maps \cite{cpz2021,eh2021}.

\vspace{3pt}

ECH is a fundamentally low-dimensional theory, and so the methods in \cite{i2015} are not directly applicable to studying the dynamics of higher-dimensional symplectomorphisms or Reeb flows. On the otherhand, ECH is part of a family of Floer theories collectively called symplectic field theory (or SFT) \cite{egh2000}, and other flavors of SFT (e.g. contact homology) generalize naturally to any dimension. In a recent work \cite{irie2022strong}, Irie described an abstract framework for proving strong closing properties using invariants satisfying formal properties in the spirit of SFT. 

\vspace{3pt}

In this paper, we use contact homology to prove that the Reeb flow of any ellipsoid satisfies the strong closing property, as conjectured by Irie in \cite{irie2022strong}. This is a first step towards applying the machinery of SFT to prove closing properties for more general classes of Reeb flows in higher dimensions.

\begin{figure}[h]
\centering
\includegraphics[width=.8\textwidth]{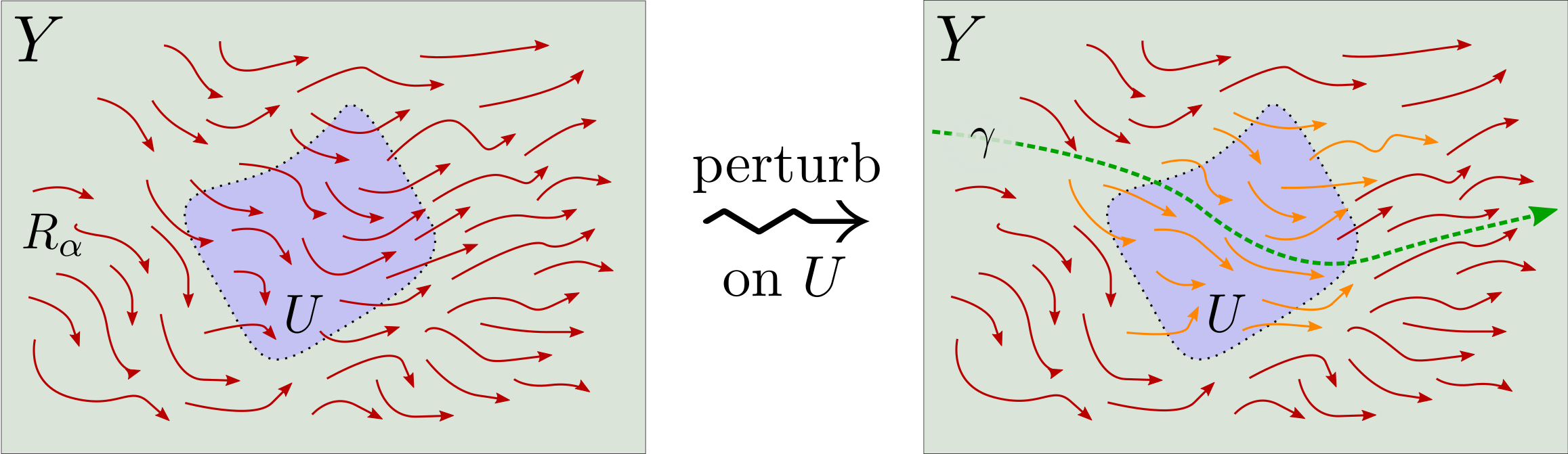}
\caption{Morally speaking, the strong closing property states that any positive perturbation of a Reeb flow supported in $U$ must produce a closed orbit $\gamma$ through $U$.}
\label{fig:closing_lemma}
\end{figure}

\subsection{Spectral Gaps} \label{subsec:spectral_gaps} The strong closing property for Reeb flows is a consequence of an abstract criterion on contact homology. In order to explain this criterion, let us briefly review the structure of contact homology (for a detailed discussion, see \S \ref{sec:contact_homology}). 

\vspace{3pt}

The contact homology of a closed contact manifold $(Y,\xi)$ with contact form $\alpha$ is a $\Z/2$-graded vector-space over $\Q$, denoted by
\[CH(Y,\xi) \qquad\text{with a filtration}\qquad CH^L(Y,\xi) \subset CH(Y,\xi) \qquad \text{determined by }\alpha.\]
If $\alpha$ is non-degenerate (i.e. if the linearized Poincar\'e return map of every closed Reeb orbit of $\alpha$ does not have $1$ as eigenvalue) then $CH(Y,\xi)$ can be computed as the homology of a dg-algebra freely generated by good Reeb orbits. The differential counts genus $0$ holomorphic curves in $\R \times Y$ with one puncture near $+\infty \times Y$ and any number of punctures near $-\infty \times Y$.

\vspace{3pt}

The contact homology algebra $CH(Y)$ comes with the additional structure of $U$-maps, which can be constructed as follows. An \emph{abstract constraint} $P$ of codimension $\on{codim}(P)$ is a graded map
\[P:CH(S^{2n-1},\xi_{\on{std}}) \to \Q[\on{codim}(P)].\]
Here $(S^{2n-1},\xi_{\on{std}})$ is the standard tight contact sphere. There is a filtered, graded map  associated to any abstract constraint $P$, denoted by
\[
U_P:CH(Y,\xi) \to CH(Y,\xi)[\on{codim}(P)].
\]
Intuitively, the \emph{$U$-map} $U_P$ counts holomorphic curves $C$ in $\R \times Y$ passing through a point $p \in \Sigma$ in a small codimension $2$ symplectic sub-manifold $\Sigma$ of $\R \times Y$, where the number of branches of $C$ through $p$ and order of tangency of $C$ at $\Sigma$ is determined by $P$. Rigorously, $U_P$ can be most easily constructed using the maps on contact homology induced by exact symplectic cobordisms (and this approach is related to the point constraint approach by Siegel \cite{sie_hig_19}).

\begin{figure}[h]
\centering
\includegraphics[width=.8\textwidth]{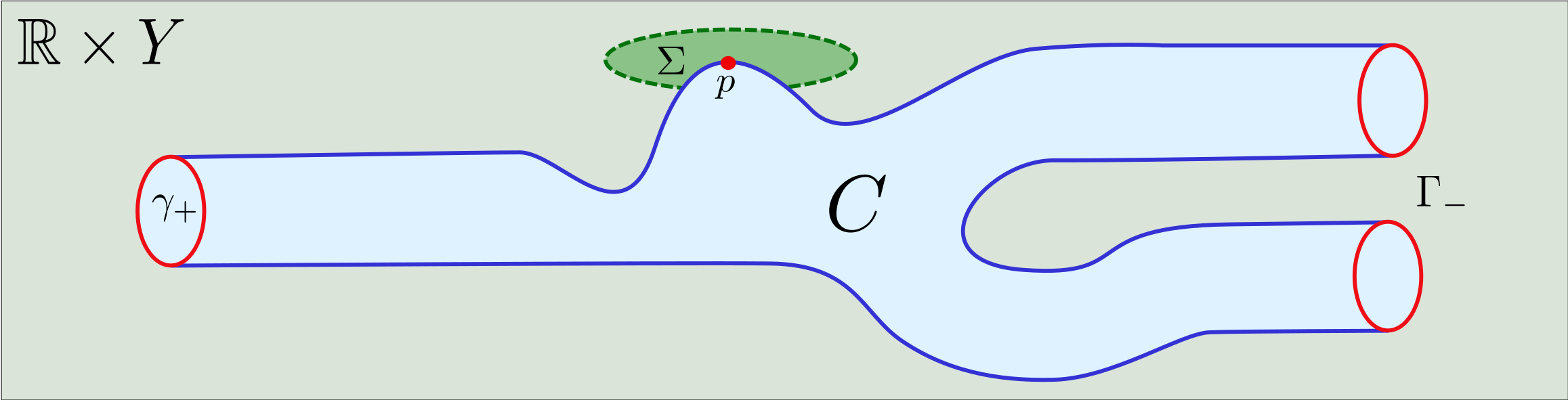}
\caption{A cartoon of a U-map curve in contact homology. In general, the tangency of the curve $C$ at the sub-manifold $\Sigma$ can be arbitrarily complex.}
\label{fig:u_map_curves}
\end{figure}

Floer homologies typically have associated \emph{spectral invariants} that track the minimal filtration at which a particular homology class appears. In symplectic geometry, these invariants have become pivotal tools in the study of quantitative and dynamical questions (cf. \cite{gh_sym_18,h2010,sie_hig_19,chs_hig_21}). There are spectral invariants associated to contact homology, denoted by
\[\mathfrak{s}_\sigma(Y,\alpha) \qquad\text{for each}\qquad \sigma \in CH(Y,\xi).\]
The $U$-map $U_P$ decreases this spectral invariant, in the sense that
\[\mathfrak{s}_{U_P\sigma}(Y,\alpha) \le \mathfrak{s}_\sigma(Y,\alpha) \qquad\text{ for any $\sigma$ and $P$}.\]
In particular, we can formulate an invariant that measures the minimal gap between the spectral invariants of a class $\sigma$ and $U_P\sigma$.

\begin{definition} \label{def:intro_ch_spectral_gap} The \emph{spectral gap} of a contact homology class $\sigma \in CH(Y,\xi)$ is given by
\[\gap_\sigma(Y,\alpha) := \inf_{P}\big\{\frac{\mathfrak{s}_\sigma(Y,\alpha) - \mathfrak{s}_{U_P\sigma}(Y,\alpha)}{\mathfrak{c}_P(B^{2n})}\big\} \in [0,\infty).\]
The normalizing constants $\mathfrak{c}_P(B^{2n})$ are certain symplectic capacities of the ball derived from the contact homology of its boundary. The \emph{contact homology spectral gap} of a closed contact manifold $Y$ with contact form $\alpha$ is given by
\[\gap(Y,\alpha) := \inf_\sigma\big\{\gap_\sigma(Y,\alpha)\big\}.\]\end{definition}

\noindent The criterion for the strong closing property using the spectral gap can now be stated as follows.

\begin{thm*} \label{thm:spectral_gap_to_closing_lemma_intro} Let $(Y,\xi)$ be a closed contact manifold with contact form $\alpha$, and suppose that
\[\gap(Y,\alpha) = 0.\]
Then $(Y,\alpha)$ satisfies the strong closing property.
\end{thm*}

\noindent This criterion is formulated in abstract terms in \cite{irie2022strong}. We will provide our own concrete discussion of this condition, along with properties of the spectral gap, in \S \ref{sec:contact_homology}.

\subsection{Spectral Gap Of Ellipsoids} \label{subsec:spectral_gap_of_ellipsoids} Given the spectral gap framework discussed above, we are naturally lead to the following question.

\begin{question*} \label{qu:spectral_gap} Let $(Y,\xi)$ be a closed contact manifold with non-trivial contact homology. Does the contact homology spectral gap vanish for any contact form $\alpha$?
\end{question*}

\noindent Even in simple cases, computing the $U$-map involves a difficult analysis of $J$-holomorphic curves, and so Question \ref{qu:spectral_gap} is extremely difficult. As a first step, Irie conjectured an affirmative answer to Question \ref{qu:spectral_gap} in the following family of examples of contact manifolds.

\vspace{3pt}

\begin{example}[Ellipsoids] An \emph{ellipsoid boundary} $(\partial E,\lambda|_{\partial E})$ is a contact manifold $\partial E$ given as the boundary of an ellipsoid  $E$ in $\C^n$ of the form
\[E = \{z \in \C^n \; : \; \langle z,Az\rangle \le 1\} \qquad\text{where} \qquad\text{$A$ is symmetric and positive definite.}\]
The contact form $\lambda|_{\partial E}$ is the restriction of the standard Liouville form $\lambda$ on $\C^n$.
\[\lambda = \frac{1}{2} \cdot \sum_{j=1}^n x_j dy_j - y_j dx_j.\]
The Reeb vector-field $R$ is given by $R(z) = 2JAz$ where $J:\R^{2n} \to \R^{2n}$ is multiplication by $i$. \end{example}

\begin{figure}[h]
\centering
\includegraphics[width=.9\textwidth]{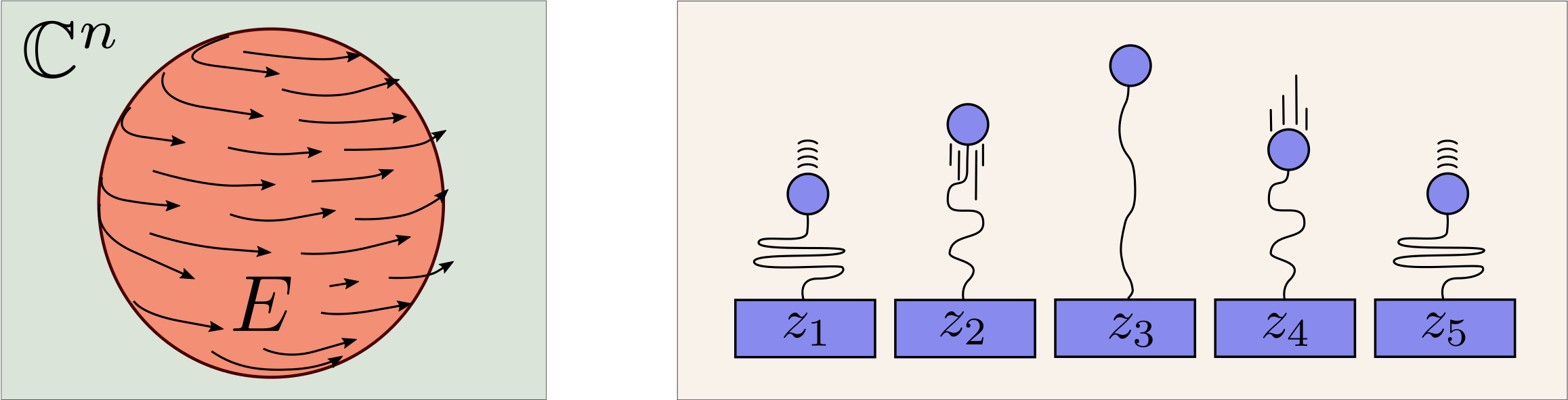}
\caption{Two different visualizations of the Reeb flow of an ellipsoid $E \subset \C^n$. On the left, a flow on the boundary of the higher-dimensional domain $E$. On the right, as the dynamics of $n$ independent harmonic oscillators (e.g. springs).}
\label{fig:harmonic_oscillator}
\end{figure}

\begin{remark}[Harmonic Oscillator] From the perspective of classical mechanics, the Reeb flow of an ellipsoid is the Hamiltonian dynamics of a harmonic oscillator on a fixed energy surface. 

\vspace{3pt}

Indeed, up to a linear symplectomorphism, every ellipsoid is equivalent to one of the form
\[
E(a) = E(a_1,\dots,a_n) := \{(z_1,\dots,z_n) \; : \; \pi \cdot \sum_i \frac{|x_i|^2 + |y_i|^2}{a_i} \le 1\} \qquad \text{where}\qquad 0 < a_1 \le \dots \le a_n.
\]
The Reeb flow on the boundary of $E(a)$ is simply the Hamiltonian flow of the Hamiltonian
\[
H:\R^{2n} \to \R \qquad\text{given by}\qquad H(x,y) = \pi \cdot \sum_i \frac{|x_i|^2 + |y_i|^2}{a_i}.
\]
This is precisely the Hamiltonian for $n$ independent harmonic oscillators with periods $a_1,\dots,a_n$. \end{remark}

Ellipsoid boundaries are some of the most well-studied contact manifolds (cf. \cite{chs_hig_21,sie_hig_19,gh_sym_18}) and have provided a useful testing ground for many conjectures. Irie's conjecture from \cite{irie2022strong} can be stated as follows.

\begin{conjecture*}[{\cite[{Conjecture 5.1}]{irie2022strong}}] \label{con:irie_ellipsoid_1} The boundary $(\partial E,\lambda|_{\partial E})$ of an ellipsoid $E \subset \C^n$ has the strong closing property.
\end{conjecture*} 

\subsection{Strong Closing Property For Ellipsoids} \label{subsec:main_results} The purpose of this paper is to prove Irie's conjecture via the following spectral gap result.

\begin{thm*} \label{thm:spectral_gap_intro} The boundary $(\partial E, \lambda|_{\partial E})$ of an ellipsoid $E \subset \C^n$ has vanishing spectral gap, and thus satisfies the strong closing property.
\end{thm*}

\begin{remark} After the posting of this paper, a different proof of the strong closing property for integrable systems was given by Xue \cite{x2022} using KAM theory. Our method of proof is different and applies to more general, non-integrable systems. See \S \ref{subsec:strong_closing_for_non_toric}. 
\end{remark}

\noindent Our approach to Theorem \ref{thm:spectral_gap_intro} has two parts: the periodic (or integer) case and the general case.

\subsubsection{Periodic Case} An \emph{integer ellipsoid} $E$ is an ellipsoid that is linearly symplectomorphic to a standard ellipsoid $E(a_1,\dots,a_n)$ where $a_1,\dots,a_n$ are all integers,
\[(a_1,\dots,a_n) \in \Z^n.\]
The Reeb flow of an ellipsoid with integer $a_i$ is periodic, i.e., the flow $\varphi_T$ is the identity for some $T$. The period is given by the least common multiple of $a_1,\dots,a_n$,
\[T = \lcm(a_1,\dots,a_n).\]

The strong closing property is automatically satisfied for these flows, since every point goes through a periodic orbit of period bounded by $T$. We first prove that this property is reflected in contact homology as follows.

\begin{thm*} \label{thm:periodic_spectral_gap_intro} Let $E = E(a) \subset \C^n$ be an ellipsoid with $(a_1,\dots,a_n) \in \Z^n$. Then, there is a contact homology class $\sigma \in CH(\partial E)$ and a $U$-map $U = U_P$ such that
\[
\mathfrak{s}_{U\sigma}(\partial E,\lambda|_{\partial E}) = \mathfrak{s}_{\sigma}(\partial E,\lambda|_{\partial E}) = \on{lcm}(a_1,\dots,a_n).
\]
\end{thm*}

\noindent
Theorem \ref{thm:periodic_spectral_gap_intro} is proven by a direct holomorphic curve calculation in contact homology. Let us briefly sketch the proof, as carried out in \S \ref{sec:moduli_space}.

\vspace{3pt}

\emph{Proof Sketch.} Note that $\partial E$ does not have a non-degenerate contact form. Instead, the contact form is Morse-Bott, and the set of closed Reeb orbits of a given period $T$ forms a sub-manifold
\[N_T \subset \partial E \qquad\text{with quotient}\qquad S_T := N_T/\R \qquad \text{by the Reeb flow}.\]
Any Morse-Bott contact form admits an arbitrarily small non-degenerate perturbation so that every closed Reeb orbit $\gamma$ of period less than some fixed $L > 0$ corresponds to a pair
\[(T,p) \qquad\text{where $p$ is a critical point of a Morse function }f:S_T \to \R.\]
The period of $\gamma$ is also approximately $T$. Note that $S_T$ is an orbifold in general, and so morally we must work with orbifold Morse functions (in the appropriate sense). Moreover, gradient flow lines $\eta$ between critical points on a fixed Morse-Bott family $S_T$ lift to holomorphic cylinders $u_\eta$ in the symplectization of $\partial E$ between the correspoinding orbits. 

\vspace{3pt}

When $T$ is the period of the Reeb flow, i.e., the least common multiple of $a_1,\dots,a_n$, the sub-manifold $N_T \subset \partial E(a)$ is simply the ellipsoid boundary $\partial E(a)$ itself and $S_T$ is a closed orbifold of dimension $2n-2$. If $p_{+}$ and $p_{-}$ are the unique maximum and minimum of a Morse function $f$ on $S_T$, then for any point $z \in S_T$, there is a unique gradient flow line
\[
\eta:\R \to S_T \qquad\text{from $p_{+}$ to $p_{-}$ passing through $z$.}
\]
This flow line lifts to a cylinder $u_\eta$ from the orbit $\gamma_{+}$ of $(T,p_{+})$ to the orbit $\gamma_{-}$ of $(T,p_{-})$ passing through a point whose projection to $S_T$ is $z$. 

\vspace{3pt}

On the other hand, there is a $U$-map $U_{P_0}$ that counts holomorphic curves satisfying a point constraint. Using intersection theory from \cite{siefring3d,moreno2019holomorphic} and Wendl's automatic transversality \cite{w2010}, we prove that the cylinder $u_\eta$ is unique and transversely cut out. Therefore,
\[U_{\on{P_0}}(\gamma_{+}) \qquad\text{has a non-zero $\gamma_{-}$ coefficient}.\]
The orbits $\gamma_{+}$ and $\gamma_{-}$ are both closed and non-exact in contact homology, and thus, the spectral invariant of $\sigma = [\gamma_{+}]$ and $U_{P_0}(\sigma)$ have the same action $T$, proving Theorem \ref{thm:periodic_spectral_gap_intro}. \qed

\begin{remark} In principle, one could use Morse-Bott formulations of contact homology (cf. \cite{b2002}) to compute the $U$-map in Theorem \ref{thm:periodic_spectral_gap_intro} directly using gradient flow lines (and holomorphic cascades more generally).  For the sake of completeness we provide here a direct analysis of the relevant moduli space and do not rely on the Morse-Bott formulation of contact homology.
\end{remark}

\subsubsection{General Case} The second step of our proof transfers the vanishing spectral gap of integer ellipsoids to general ellipsoids via the following approximation property. 

\begin{prop*}\label{prop:limit_of_gap0_intro} Let $(Y,\xi)$ be a closed contact manifold with a sequence of
\[\text{contact forms }\alpha_i, \qquad\text{homology classes }\sigma_i \in CH(Y,\xi), \qquad \text{and}\qquad \text{real numbers }\epsilon_i > 0. \]
Suppose that, as $i \to \infty$, these sequences satisfy
\[\alpha_i \le \alpha \le (1 + \epsilon_i) \cdot \alpha_i \qquad\text{and}\qquad \epsilon_i \cdot \s_{\sigma_i}(Y,\alpha_i) \to 0.\]
Then, the contact homology spectral gaps satisfy
\[\gap(Y,\alpha)\leq \liminf_{i\rightarrow\infty} \gap_{\sigma_i}(Y,\alpha_i).\] 
\end{prop*}

Proposition \ref{prop:limit_of_gap0_intro} can be proven in an entirely formal way from Definition \ref{def:intro_ch_spectral_gap}. On the other hand, elementary results in Diophantine approximation can be used to prove the following result.

\begin{prop*} \label{prop:approximating_ellipsoids_intro} Let $E$ be any ellipsoid. Then there exists a linear symplectomorphism $\phi:\C^n \to \C^n$ and a sequence of rational ellipsoids $E_i$ of period $T_i$ such that
\[
(1 - \epsilon_i) \cdot E_i \subset \phi(E) \subset (1 + \epsilon_i) \cdot E_i \qquad\text{and}\qquad \epsilon_i \cdot T_i \to 0.
\]
\end{prop*}
\noindent Rational ellipsoids are rescalings of integer ellipsoids and so by Theorem \ref{thm:periodic_spectral_gap_intro}, there exist classes 
\[\sigma_i \in CH(\partial E_i,\lambda|_{\partial E_i}) \qquad \text{with}\qquad \mathfrak{s}_{\sigma_i}(\partial E_i,\lambda|_{\partial E_i}) = T_i.\]
Theorem \ref{thm:spectral_gap_intro} is therefore an immediate consequence of Propositions \ref{prop:limit_of_gap0_intro} and \ref{prop:approximating_ellipsoids_intro}.

\subsection{Strong Closing Property For A Non-Integrable Flow}\label{subsec:strong_closing_for_non_toric} The methods of this paper are quite general, and can be applied to prove the strong closing lemma for flows that are very different from the ellipsoids. To illustrate this, in \S \ref{sec:non_toric_example} we will prove the strong closing property for a family of non-periodic contact forms on a non-integrable contact $5$-manifold $Y$. This family arises as a particular case of a construction by Albers-Geiges-Zehmisch \cite{agz2018}.

\begin{definition}\label{def:integrable_intro}
     We say that a Reeb flow on a contact manifold $Y$ of dimension $2n-1$ is \emph{integrable} if it generates an $S^1$ action that extends to a Hamiltonian $\T^{2n-1}$ action on $Y$.
\end{definition}
\vspace{3pt}

Consider the symplectic manifold $X = \C P^1 \times \C P^1$ equipped with the product symplectic form $\omega$. Let $A_1,A_2$ denote the cohomology classes Poincare dual to $[\C P^1 \times \on{pt}]$ and $[\on{pt} \times \C P^1]$ respectively. Consider the prequantization bundle
\[S \to X \quad\text{with}\quad c_1(S) = A_1 + A_2\]
As a prequantization space, $S$ admits a contact form $\alpha$ uniquely determined by $d\alpha = \pi^*\omega$. Moreover, $S$ admits a Hamiltonian circle action generated by the vector-field $X_H$ of the Hamiltonian
\[
H:X \to \R \quad\text{given by}\quad H(x,y) = \frac{\pi}{2}\Big(\frac{|x|^2}{(1 + |x|^2)} + \frac{|y|^2}{(1 + |y|^2)}\Big) + 1
\]
This action lifts to a circle action on $S$ generated a Reeb vector-field $R_2$ (for a contact form $\alpha_2$) that commutes with the Reeb vector-field $R_1$ of $\alpha$. This yields a contact $\T^2$-action
\[
\T^2 \curvearrowright (S,\xi)
\]
Moreover, any vector $a = (a_1,a_2) \in \mathfrak{t}^2 \simeq \R^2$ with $a_i > 0$ corresponds to a contact form $\alpha_a$ and a vector-field $R_a$ given by $a_1 \cdot R_1 + a_2 \cdot R_2$.

\vspace{3pt}

The space $S$ described above is toric (since the $S^1$-action generated by $H$ extends to a Hamiltonian $\T^2$-action). However, there is free $\Z_4$-action of $S$ lifting the $\Z_4$-action on $X$ generated by the following $4$-periodic map.
\[
\C P^1 \times \C P^1 \to \C P^1 \times \C P^1 \quad\text{with}\quad f(x,y) = (-y,x)
\]
The action on $S$ commutes with the $\T^2$-action and we can show that
\begin{lemma} \label{lem:non_toricity_intro} The quotient $Y = S/\Z_4$ does not admit a Hamiltonian $\T^3$-action extending the $\T^2$-action. 
\end{lemma}

The contact form $\alpha_a$ descends to $Y$, and is periodic if and only if $a$ is (proportional to) a rational vector. By an analogous argument to \S \ref{subsec:main_results}, we prove that

\begin{thm*} \label{thm:non_toric_ex_intro} The contact manifolds $(Y,\alpha_a)$ satisfy the strong closing property for each $a \in (0,\infty)^2$.
\end{thm*}

\noindent In light of Lemma \ref{lem:non_toricity_intro}, we suspect that these Reeb flows cannot be approximated by integrable flows. If this is the case,  methods from integrable systems (e.g. the KAM theory methods of \cite{x2022}) cannot be applied to these flows.

\subsection{Generalizing Irie's Conjecture} We conclude this introduction by discussing several conjectures that are motivated by the results of this paper. These conjectures will be the subject of future work.

\vspace{3pt}

A contact form $\alpha$ on a contact manifold $(Y,\xi)$ is \emph{periodic} if the Reeb flow $\varphi:\R \times Y \to Y$ satisfies
\[\varphi_T = \id \qquad\text{for some time}\qquad T > 0.\]
The closed Reeb orbits of a periodic contact form on a closed manifold $Y$ form closed orbifolds
\[S_T := N_T/\R \qquad\text{where}\qquad N_T := \{y \in Y \; : \; \varphi_T(y) = y\}.\]
Every connected components $S$ of $S_T$  has an associated grading shift $\on{gr}_S$, given by the formula
\[\on{gr}_S := \indRS(S) - \frac{1}{2} \cdot \on{dim}(S) \mod 2.\]
Here $\on{RS}$ is the Robbin-Salamon index (see \S \ref{subsubsec:linearized_flow}) of the linearized flow along any Reeb orbit in $S$. To each component $S$, we also associate a filtered, graded vector-space 
\[
V(S) := H_{\bullet + \on{gr}_S}(S;\Q).
\]
We equip $V(S)$ with the homology grading shifted by $\on{gr}_S$ and the trivial filtration where $V^L(S)$ is $0$ if $L$ is less than the period of the orbits in $S$ and $V(S)$ otherwise.

\vspace{3pt}

Our first conjecture provides a simple formula for contact homology in the periodic setting.

\begin{conjecture*} \label{con:CH_of_periodic} Let $(Y,\xi)$ be a closed contact manifold with a periodic contact form $\alpha$. Then
\[
CH(Y,\xi) \simeq \on{Sym}\Big(\bigoplus_S V(S) \Big) \qquad\text{as filtered, graded vector-spaces}.
\]
\end{conjecture*}

\noindent We expect that Conjecture \ref{con:CH_of_periodic} can be proven without Morse-Bott theory using Pardon's formulation of contact homology \cite{p2015} using an $S^1$-localization argument, analogous to the proof of the Arnold conjecture given in \cite{p2016}. 

\vspace{3pt}

Our next conjecture generalizes Theorem \ref{thm:periodic_spectral_gap_intro} to general periodic contact forms.

\begin{conjecture*} \label{con:spectral_gap_periodic} Let $(Y,\xi)$ be a closed contact manifold with a periodic contact form $\alpha$ of period $T$. Then, there exists a class $\sigma \in CH(Y,\xi)$ and a $U$-map $U_P$ such that
\[
\mathfrak{s}_{U_P\sigma}(Y,\alpha) = \mathfrak{s}_{\sigma}(Y,\alpha) = T.
\]
\end{conjecture*} 

\noindent We expect that Conjecture \ref{con:spectral_gap_periodic} admits a similar proof to Theorem \ref{thm:periodic_spectral_gap_intro}. Namely, one should lift a gradient flow line on the orbifold of closed orbits $S = S_T$ between the maximum and minimum of an (orbifold) Morse function on $S$ to a $J$-holomorphic cylinder counted by the $U$-map and then verify uniqueness and transversality using intersection theory and automatic transversality.

\begin{remark} A key technical difficulty is our use of a flag of Reeb invariant contact sub-manifolds of $\partial E$ in Theorem \ref{thm:periodic_spectral_gap_intro}. A possible  analogue of this flag can be constructed using the orbifold Donaldon divisor construction \cite{gmz2022} on the orbifold quotient $Y/\R$. \end{remark}

We conclude by noting that the spectral gap provides a criterion for periodicity; a similar statement for ECH spectral invariants was proved by Cristofaro-Gardiner--Mazzuchelli \cite{cm2020}. The proof of this criterion appears in \S \ref{sec:contact_homology}. It follows from a simple modification of the formal arguments used to prove Theorem \ref{thm:spectral_gap_intro}.

\begin{thm*}[Periodicity Criterion] \label{thm:periodicity_criterion} Let $(Y,\xi)$ be a closed contact manifold with contact form $\alpha$ where
\begin{equation}\label{eq:zero_sigma_gap}
    \on{gap}_\sigma(Y,\alpha) = 0 \qquad\text{for some}\qquad \sigma \in CH(Y,\xi).
\end{equation}
Then the Reeb flow of $\alpha$ is periodic. \end{thm*}

\noindent Note that Conjecture \ref{con:spectral_gap_periodic} states that (\ref{eq:zero_sigma_gap}) is also a necessary condition for periodicity. Therefore, a slightly weaker version of Conjecture \ref{con:spectral_gap_periodic} can be reformulated as follows.

\newtheorem*{conj_reformulated}{Conjecture \ref{con:spectral_gap_periodic}'}
\begin{conj_reformulated} If $(Y,\xi)$ is a closed contact manifold with contact form $\alpha$, then the following are equivalent.
\begin{itemize}
    \item[(a)] The Reeb flow of the contact form $\alpha$ is periodic.
    \vspace{2pt}
    \item[(b)] There is a class $\sigma \in CH(Y,\xi)$ such that $\on{gap}_\sigma(Y,\alpha) = 0$.
\end{itemize}
\end{conj_reformulated}

We say that a contact form $\alpha$ on $(Y,\xi)$ is \emph{near periodic} if there is a sequence of periodic contact forms $\alpha_i$ of period $T_i$ and $\epsilon_i > 0$ such that 
\[\alpha_i \le \alpha \le (1 + \epsilon_i) \cdot \alpha_i \qquad\text{and}\qquad \epsilon_i \cdot T_i \to 0\]
The vanishing of the spectral gap of periodic contact forms can be transferred to near-periodic contact contact forms via the approximation result in Proposition \ref{prop:limit_of_gap0_intro}, as in the ellipsoid case. We expect many new examples of non-periodic Reeb flows that satisfy Irie's strong closing property to arise in this way.

\subsection*{Outline}
The paper is organized as follows. In Section~\ref{sec:contact_homology} we review necessary preliminaries from contact homology. Section~\ref{sec:spectral_gaps} contains the exposition of abstract constraints in contact homology and the spectral gap. It also contains the proofs of Theorem~\ref{thm:spectral_gap_to_closing_lemma_intro}, Proposition~\ref{prop:limit_of_gap0_intro} and Theorem~\ref{thm:periodicity_criterion}. Section~\ref{sec:moduli_space} contains an analysis of the moduli space counted by the $U_{P_1}$-map. In Section~\ref{sec:proof_of_main_thm} we use this analysis to prove Theorem~\ref{thm:periodic_spectral_gap_intro}. We also provide a proof of Proposition~\ref{prop:approximating_ellipsoids_intro} and thus conclude the proof of the closing property for all ellipsoids, as stated in Theorem~\ref{thm:spectral_gap_intro}. Section~\ref{sec:non_toric_example} contains a proof of Theorem~\ref{thm:non_toric_ex_intro}, the strong closing property for the non-integrable flows constructed in Section~\ref{subsec:strong_closing_for_non_toric}. 

\subsection*{Acknowledgements} We would like to thank Helmut Hofer for discussions that initiated this project. We also thank Lior Alon, Oliver Edtmair, Michael Hutchings, Agustin Moreno and Kyler Siegel for helpful conversations. JC was supported by the National Science Foundation under Award No. 2103165. ID was supported by the National Science Foundation under Grant No. DMS-1926686. RP was supported by the National Science Foundation under Award No. DGE-1656466. ST was supported by the AMIAS Membership at the Institute for Advanced Study and the Zuckerman Israeli Postdoctoral Scholarship.

\section{Contact Homology} \label{sec:contact_homology} In this section, we review the formalism of contact homology, which is a simple variant of  symplectic field theory originally introduced by Eliashberg-Givental-Hofer \cite{egh2000}.

\begin{remark} We will work with the transversality framework developed by Pardon \cite{p2015}, although all of the results discussed here should be independent of the specific transversality scheme. \end{remark}

\subsection{Reeb Orbits} We start by discussing some preliminaries about Reeb dynamics. Throughout this section, we fix a contact manifold $(Y,\xi)$ with a contact form $\alpha$.
We let $R$ denote the Reeb vector-field of $\alpha$ and $\varphi:\R \times Y \to Y$ denote the Reeb flow.

\subsubsection{Reeb Orbits} A closed or periodic \emph{Reeb orbit} $\gamma$ is a closed trajectory of the Reeb vector-field $R$, that is,
\[\gamma:\R/T\Z \to Y \qquad \text{satisfying}\qquad \frac{d\gamma}{dt} = R \circ \gamma.\]
Here $T$ is called the the \emph{period} or \emph{action} of $\gamma$, and for any Reeb trajectory one has
\[T = \mathcal{A}(\gamma) \qquad\text{where}\qquad \mathcal{A}(\gamma) := \int_\gamma \alpha. \]
Two Reeb orbits $\gamma$ and $\eta$ are \emph{equivalent} if they are related by translation in $t$, that is, if
\[\gamma(t + t_0) = \eta(t) \qquad\text{for some $t_0 \in \R$ and all $t \in \R$.}\]
Any closed trajectory $\gamma$ factors into a covering map $\phi$ and a \emph{simple} (i.e., injective) closed orbit $\eta$.
\[
\R/T\Z \xrightarrow{\phi} \R/T'\Z \xrightarrow{\eta} Y.
\]
The \emph{covering multiplicity} $\kappa_\gamma$ of $\gamma$ is the degree of the covering map $\phi$, namely, 
\begin{equation}
    \kappa_\gamma := \on{deg}(\phi).
\end{equation}

We will consider tuples of Reeb orbits, possibly with repetition, of the form
\[\Gamma = (\gamma_1,\dots,\gamma_N).\]
If $\Gamma$ consists of $m$ distinct orbits $\eta_1,\dots,\eta_m$ occurring with multiplicity $\mu_i$ in the sequence and having covering multiplicity $\kappa_i$, respectively, then we let 
\[
\mathcal{A}(\Gamma) = \sum_i \mathcal{A}(\gamma_i), \qquad \mu_\Gamma = \mu_1 \mu_2 \cdots \mu_m, \qquad\text{and}\qquad \kappa_\Gamma = \kappa_1 \kappa_2 \cdots \kappa_m.
\]

\subsubsection{Non-Degeneracy} Let $N_T$ denote the set of fixed points of the time $T$ Reeb flow.
\[
N_T := \{y \in Y \; : \; \varphi_T(y) = y\} \subset Y.
\]
We say that $N_T$ is $\emph{Morse-Bott}$ if it is a closed sub-manifold of $Y$ with tangent bundle given by
\[
TN_T = \on{ker}(d\varphi_T - \id)|_{N_T}.
\]
A \emph{Morse-Bott family} $S \subset N_T/S^1$ of Reeb orbits is a connected component of the quotient
$N_T/S^1$ for some period $T$.
Here $S^1 \simeq \R/T\Z$ acts on $N_T$ by the Reeb flow. Note that any Morse-Bott family is automatically an effective orbifold. As a special case, a Reeb orbit $\gamma$ is \emph{non-degenerate} if
\[\on{ker}(d\varphi_T - \id)|_\gamma = T\gamma = \on{span}(R).\]
That is, if $S = \gamma/S^1$ is a $0$-dimensional Morse-Bott family.

\vspace{3pt}

A contact form $\alpha$ is said to be \emph{non-degenerate below action $L$} if every closed Reeb orbit $\gamma$ with $\mathcal{A}(\gamma) \le L$ is non-degenerate. The form $\alpha$ is called \emph{non-degenerate} if every closed Reeb orbit is non-degenerate. Finally, $\alpha$ is \emph{Morse-Bott} if every closed Reeb orbit is in a Morse-Bott family.

\subsubsection{Linearized Flow And Indices} \label{subsubsec:linearized_flow} A \emph{trivialization} $\tau:\gamma^*\xi \simeq \C^{n-1}$ of $\xi$ along a Reeb orbit $\gamma$ is a trivialization of $\gamma^*\xi$ as a symplectic vector-bundle, that is, $\tau$ is a 1-parameter family of symplectic diffeomorphism $\tau_{\gamma(t)}: \xi_{\gamma(t)}\rightarrow \C^{n-1}$. The \emph{linearized flow} $\Phi_{\tau,\gamma}$ associated to $\gamma$ and $\tau$ is the path of symplectic matrices
\[
\Phi_{\tau,\gamma}:[0,T] \to \on{Sp}(2n-2) \quad\text{given by}\quad \Phi_{\tau,\gamma}(t) = \tau_{\gamma(t)} \circ d\varphi^R_t|_\xi \circ \tau_{\gamma(0)}^{-1}.
\]
This path depends on the trivialization, but if $\sigma$ and $\tau$ are isotopic trivializations, then the paths $\Phi_{\sigma,\gamma}$ and $\Phi_{\tau,\gamma}$ are isotopic via paths $\Phi^s:[0,T] \to \on{Sp}(2n-2)$ for $t \in [0,1]$ such that
\[
\on{rank}(\ker(\Phi^s(T) - \id)) \qquad\text{ is constant in } s.
\]

The path $\Phi_{\tau,\gamma}$ allows us to associate a Robbin-Salamon index (introduced by Robbin-Salamon in \cite{robbin1993maslov}) to any orbit with a trivialization. 

\begin{prop} \cite{gutt2014generalized} \label{prop:CZ} For each $n\geq 1$, there exists an integer valued \emph{Robbin-Salamon index} of paths of symplectic matrices
\[\on{RS}:C^0([0,1],\on{Sp}(2n)) \to \Z,\]
that is characterized by the following axioms.

\begin{itemize}
    \item[(a)] (Homotopy) If $\Phi^s:[0,1] \to \on{Sp}(2n)$ for $s \in [0,1]$ is a family of paths such that $\Phi^s(0)$ and $\Phi^s(1)$ are independent of $s$, then
    \[\indRS(\Phi^0) = \indRS(\Phi^1).\]
    \item[(b)] (Additive) $\indRS$ is additive under concatenation and direct sum, that is,
    \[\indRS(\Phi * \Psi) = \indRS(\Phi) + \indRS(\Psi) \qquad \text{and}\qquad \indRS(\Phi \oplus \Psi) = \indRS(\Phi) + \indRS(\Psi).\]
    \item[(c)] (Vanishing) If $\Phi:[0,1] \to \on{Sp}(2n)$ is such that $\on{rank}(\Phi(t) - \id)$ is constant in $t$, then $\indRS(\Phi) = 0$.
\end{itemize}
\end{prop}

\noindent The Robbin-Salamon index generalizes the Conley-Zehnder index, in the sense that $\indRS(\Phi) = \CZ(\Phi)$ when $\on{ker}(\Phi(1) - \id)$ is $0$-dimensional and $\Phi(0)$ is the identity. 

\vspace{3pt}

For sufficiently nice paths, the Robbin-Salamon index can be explicitly computed using crossings of the Maslov cycle. To be precise, for a given path $\Phi$ and any $t \in [0,1]$, let $\Gamma_t$ be the symmetric bilinear form on $\on{ker}(\Phi(t) - \id)$ given by
\[
\Gamma_t(v,w) = \omega_0 \left(\frac{d\Phi}{dt}(t)\Phi^{-1}(t)v,w \right).
\]
Let $\on{sign}(\Gamma_t)$ be $-1$ raised to the power of the signature of $\Gamma_t$. A \emph{crossing} is a time $t \in [0,T]$ such that $\on{det}(\Phi(t) - \id) = 0$. A crossing $t$ is \emph{non-degenerate} if $\Gamma_t$ is non-degenerate. When every crossing time of $\Phi$ is non-degenerate, the Robbin-Salamon index is given by
\begin{equation} \label{eq:RS_def}
\indRS(\Phi) = \frac{1}{2}\on{sign}(\Gamma_0) + \sum_{0 < t < T} \on{sign}(\Gamma_t) + \frac{1}{2}\on{sign}(\Gamma_T)\end{equation}
where the sum is over all crossings of $\Phi$. 

\vspace{3pt}

Let $\gamma$ be an orbit lying in a Morse-Bott family of Reeb orbits $S$, and let $\tau$ be a trivialization of the contact structure $\xi$ on a neighborhood of $\gamma$. The \emph{Robbin-Salamon index} of a $\gamma$ with respect to  $\tau$ is given by
\[\indRS_\tau(\gamma) = \indRS(\Phi_{\tau,\gamma})\]
where $\Phi_{\tau,\gamma}$ is the linearized flow  along $\gamma$, restricted to $\xi$ and trivialized by $\tau$. We will sometimes denote $\indRS_\tau(\gamma)$ by $\indRS_\tau(S)$ in order to distinguish the index of a degenerate orbit from the index of a  non-degenerate perturbation of it. We remark that the pairity of the RS index is independent of the choice of trivialization\footnote{This follows from the loop property of the CZ index (see e.g. \cite{gutt2014generalized}) and the concatenation property of the RS index}. If $\gamma$ is non-degenerate, then we define the \emph{Conley-Zehnder index} of $\gamma$ with respect to $\tau$ as
\[
\CZ_\tau(\gamma) := \indRS(\Phi_{\tau,\gamma}).
\]
Finally, the \emph{SFT grading} $|\gamma|_\tau$ is given by
\[
|\gamma|_\tau = (n-3) + \on{CZ}_\tau(\gamma) \mod 2
\]
where $\on{dim}(Y) = 2n - 1$. 

\subsection{Holomorphic Buildings} We next establish basic notation for $J$-holomorphic curves and,  more generally, $J$-holomorphic buildings.

\subsubsection{Symplectic Cobordisms And Liouville Domains} \label{subsec:symplectic_cobordisms} A \emph{symplectic cobordism} $X:Y_+ \to Y_-$ between closed contact manifolds $(Y_+,\alpha_+)$ and $(Y_-,\alpha_-)$ is a compact symplectic manifold $(X,\omega)$ with boundary such that
\[\partial X = Y_+ \cup (-Y_-) \qquad\text{and}\qquad \omega|_{TY_\pm} = d\alpha_\pm.\]
The symplectic cobordism $X$ is \emph{exact} if $\omega = d\lambda$ where $\lambda|_{TY_\pm} = \alpha_\pm$. A \emph{deformation} of exact cobordisms is simply a smooth $1$-parameter family of exact cobordisms $(X,\lambda_t):Y_+ \to Y_-$ parametrized by $t \in [0,1]$. Two exact cobordisms are \emph{deformation equivalent} if they are (up to isomorphism) connected by a deformation.

\vspace{3pt}

Given symplectic cobordisms $X:Y_1 \to Y_2$ and $X':Y_2 \to Y_3$, there is a well-defined \emph{composition} given by
\[X \circ X' := X \cup_{Y_1} X'.\]
This cobordism inherits a symplectic cobordism structure (induced by a standard collar neighborhood of the boundary) and is exact if $X$ and $X'$ are exact.

\vspace{3pt}

Any symplectic cobordism $X$ can be completed to a non-compact, cylindrical manifold called the \emph{completion} of $X$, denoted by
\[\hat{X} = (-\infty,0] \times Y_- \cup_{Y_-} X \cup_{ Y_+} [0,\infty) \times Y_+.\]
As a special case, if $(Y,\alpha)$ is a contact manifold with contact form, we have a trivial exact cobordism
\[
[a,b]_r \times Y:(Y,e^b\alpha) \to (Y,e^a\alpha) \qquad\text{with Liouville form}\qquad \lambda = e^r\alpha.
\]
The completion of the trivial cobordism is called the \emph{symplectization} and is denoted by $\hat{Y}$. 

\vspace{3pt}

A \emph{Liouville domain} $(W,\lambda)$ is an exact symplectic cobordism from $\partial W$ to the emptyset. If $X$ is an exact symplectic manifold with Liouville form $\theta$, then a symplectic embedding $\iota:W \to X$ from a Lioiville domain is called \emph{exact} if
\[\iota^*\theta = \lambda\]
More generally, $\iota$ is \emph{weakly exact} if
\[[\lambda|_{\partial W} - \iota^*\theta|_{\partial W}] = 0 \in H^1(\partial W;\R).\]
For any weakly exact embedding $\iota:W \to X$, the Liouville form $\theta$ of $X$ is homotopic through Liouville forms in a neighborhood of $W$ so that $\iota$ is exact. After such a deformation, $X \setminus W$ is a symplectic cobordism
\[X:Y_+ \to Y_- \cup \partial W\]
and we may write $X = (X \setminus W) \circ (W \cup [0,1] \times Y_-)$ (up to deformation).

\subsubsection{Homology Classes} Given a symplectic cobordism $X:Y_+ \to Y_-$ and sequences of Reeb orbits $\Gamma_\pm$ in $Y_\pm$, let
\[\Xi_\pm \subset Y_\pm\]
denote the $1$-manifold in $Y_\pm$ given as the union of the underlying simple orbits of $\Gamma_\pm$. We may identify $\Xi_+ \cup \Xi_-$ as a sub-manifold of $\partial X$. We denote the following subset of the relative homology by
\[
S(X;\Gamma_+,\Gamma_-) := \{A \in H_2(X,\Xi_+ \cup \Xi_-) \; : \; \partial A = [\Gamma_+] - [\Gamma_-] \in H_1(\partial X)\}.
\]
Given a homology class $A \in S(X;\Gamma_+,\Gamma_-)$ and a trivialization $\tau:\xi|_{\Gamma_+ \cup \Gamma_-} \simeq \R^{2n-2}$ of $\xi$ over the collections of Reeb orbits $\Gamma_+$ and $\Gamma_-$, there is a well-defined \emph{relative Chern number} (Definition 5.1 in \cite{wendlsymplectic})
\[c_1(A,\tau) \in \Z.\]
Moreover, given a choice of genus $g$, there is a well-defined \emph{Fredholm index} given by
\[\ind(A,g) := (n-3) \cdot (2 - 2g - |\Gamma_+| - |\Gamma_-|) + c_1(A,\tau) + \sum_{\gamma_+ \in \Gamma_+} \CZ_\tau(\gamma_+) - \sum_{\gamma_- \in \Gamma_-} \CZ_\tau(\gamma_-).\]

Note that if $X:Y_0 \to Y_1$ and $X':Y_1 \to Y_2$ are symplectic cobordisms, then there is a map of homology classes in $X$ and $X'$ to homology classes in the composition
\[
S(X;\Gamma_0,\Gamma_1) \times S(X';\Gamma_1,\Gamma_2) \to S(X \circ X';\Gamma_0,\Gamma_2) \qquad \text{denoted by}\qquad (A,A') \mapsto A + A'.
\]
This map is associative. Moreover, the Chern class and Fredholm index are both additive with respect to this operation.

\subsubsection{Complex Structures} A \emph{compatible almost complex structure} on a contact manifold $(Y,\xi)$ is a bundle endomorphism $J:\xi \to \xi$ such that
\[J^2 = -\id \qquad\text{and}\qquad d\alpha(-, J-)|_\xi \text{ is a metric for any contact form }\alpha.\]
Any such almost complex structure extends to an almost complex structure $\hat{J}$ on $\hat{Y}$ by
\[\hat{J}|_\xi = J \qquad\text{and}\qquad \hat{J}(\partial_r) = R \qquad\text{where}\qquad \hat{Y} = \R_r \times Y. \]
Likewise, for $J_\pm$ compatible almost complex structures on $Y_\pm$, a \emph{compatible almost complex structure} $\hat{J}$ on the completion $\hat{X}$ of a symplectic cobordism $X:Y_+ \to Y_-$ is a bundle endomorphism $\hat{J}:T\hat{X} \to T\hat{X}$ such that
\[\hat{J}^2 = -\id,\qquad \omega(-, \hat{J}-)|_X \text{ is a metric on }X, \qquad \hat{J}|_{(-\infty,0] \times Y_-} = \hat{J}_-, \qquad\text{and }  \hat{J}|_{[0,\infty) \times Y_+} = \hat{J}_+.\]

\subsubsection{Holomorphic Maps} Fix a symplectic cobordism $X:Y_+ \to Y_-$ and a compatible almost complex structure $\hat{J}$ on $\hat{X}$. Consider a closed Riemann surface
\[(\bar{\Sigma},j) \qquad \text{with a finite set of punctures}\qquad P \subset \bar{\Sigma}.\]
A \emph{$J$-holomorphic map} from the punctured surface $\Sigma := \bar{\Sigma} \setminus P$ to the symplectization $(\hat{X},\hat{J})$ is a map
\[u:\Sigma \to \hat{X} \qquad\text{satisfying}\qquad J \circ du = du \circ j.\]
The \emph{$\lambda$-energy} $\mathcal{E}_\lambda(u)$ and $\emph{area}$ $\mathcal{A}(u)$ of a $\hat{J}$-holomorphic map are defined as follows.
\begin{align*}
\mathcal{E}_\lambda(u) & = \sup_{\phi_-} \int_{u^{-1}((-\infty,0) \times Y)} u^*(\phi_- ds \wedge \alpha_-)) + \sup_{\phi_+} \int_{u^{-1}((0,\infty) \times Y)} u^*(\phi_+ ds \wedge \alpha_+)),\\
\mathcal{A}(u) & = \int_{u^{-1}((-\infty,0) \times Y)} u^*(d\alpha_-) + \int_{u^{-1}(X)} u^*\omega + \int_{u^{-1}((0,\infty) \times Y)} u^*(d\alpha_+).
\end{align*}
Here the supremums are over compactly supported functions $\phi_-:(-\infty,0] \to [0,\infty)$  and $\phi_+:[0,\infty) \to [0,\infty)$ with integral $1$. The \emph{energy} $\mathcal{E}(u)$ is simply the sum
\[\mathcal{E}(u) = \mathcal{E}_\lambda(u) + \mathcal{A}(u).\]
A holomorphic map is said to have \emph{finite energy} if $\mathcal{E}(u)$ is well-defined. Any finite energy, proper holomorphic map $u$ is asymptotic to sequences of closed Reeb orbits $\Gamma_\pm$ as $r\to \pm\infty$,
$u \to \Gamma_\pm$ as $\pm \infty$.
To be precise, for each puncture $p \in P$, there is a neighborhood $U$ of $p$ and a holomorphic chart $\phi:[0,\infty)_s \times (S^1)_t \simeq U \setminus p \subset \Sigma$ such that
\[u \circ \phi([0,\infty) \times S^1) \subset [0,\infty) \times Y_+ \qquad\text{or}\qquad u \circ \phi([0,\infty) \times S^1)) \subset (-\infty,0] \times Y_-,\]
and
\[\pi_\R \circ u\circ\phi(s,-) \xrightarrow[s\rightarrow\pm\infty]{} \pm \infty 
,\qquad \pi_Y \circ u \circ \phi(s,-) \xrightarrow[s\rightarrow\pm\infty]{C^0} \gamma(\pm T-) \qquad\text{as a map}\quad S^1 \to Y_\pm.\]
Here $\gamma$ is a closed Reeb orbit of $Y_\pm$, $\pi_\R:\hat{X} \to \R$ is the projection to the $\R$-factor, and $\pi_Y:\hat{X} \to Y_\pm$ is the projection to the $Y_\pm$-factor where both projections are defined on the cylindrical ends of $\hat{X}$. Hence, Stokes theorem implies that
\[
\mathcal{A}(u) = \mathcal{A}(\Gamma_+) - \mathcal{A}(\Gamma_-).
\]
Since the area of holomorphic maps is always non-negative, this implies that $\mathcal{A}(\Gamma_+) \ge \mathcal{A}(\Gamma_-)$. Finally, any holomorphic curve $u$ from $\Gamma_+$ to $\Gamma_-$ represents a class
\[[u] \in S(X;\Gamma_+,\Gamma_-)\]
acquired as the fundamental class of the composition $\pi \circ u$ where $\pi:\hat{X} \to X$ is the continuous map sending $X$ to itself and the ends $(-\infty,0) \times Y_-$ and $(0,\infty) \times Y_+$ to $Y_-$ and $Y_+$, respectively.

\vspace{3pt}

\subsubsection{Asymptotic Markers and Matchings} Consider a cobordism $X:Y_+ \to Y_-$ and equip each simple Reeb orbit $\eta$ in $Y_+$ and $Y_-$ with a basepoint $b_\eta \in \eta \subset Y_\pm$. 

\vspace{3pt}

Given a finite energy holomorphic map $u:\Sigma \to \hat{X}$ asymptotic to a Reeb orbit $\gamma$ at a puncture $p$, let $S_p$ denote the unit circle bundle at the puncture $p$
\[S_p := (T_p\bar{\Sigma} \setminus 0)/\R_+.\]
The complex structure on $\bar{\Sigma}$ induces an $S^1$-action on $S_p$ and there is a natural map of the form
\[\pi_{u,p}:S_p \to \eta ,\qquad\text{where $\eta$ is the simple orbit of $\gamma$}.\]
An \emph{asymptotic marker} $m_\gamma$ at a puncture $p \in P$ is a choice of element
\[m_\gamma \in S_p, \qquad\text{such that}\qquad \pi_{u,p}(m_\gamma) = b_\eta.\]
Given two holomorphic curves $u$ and $v$ asymptotic to $\gamma$ at punctures $p$ and $q$, respectively, a \emph{matching isomorphism} $\mu$ between their ends is a map
\[\mu:S_p \to S_q, \qquad\text{such that}\qquad \pi_{v, q}\circ\mu(\pi_{u,p}^{-1}(b_\eta)) = b_\eta.\]
A holomorphic map $u$ is said to be \emph{equipped with asymptotic markers} if each of its punctures is. Note that a biholomorphism $\phi: \Sigma \to \Sigma'$ induces a map on the set of asymptotic markers and matching isomorphisms along the punctures. We denote these maps by $\phi_*$. 

\vspace{3pt}

\subsubsection{Holomorphic Curves} Given an integer $g$, a homology class $A \in S(X;\Gamma_+,\Gamma_-)$, and tuples of Reeb orbits $\Gamma_\pm$ in $Y_\pm$, we have an associated \emph{moduli space} of finite energy $J$-holomorphic curves
\[
\mathcal{M}_{g,A}(X,J;\Gamma_+,\Gamma_-) \qquad\text{or}\qquad \mathcal{M}_{g,A}(\hat{X},\hat{J};\Gamma_+,\Gamma_-).
\]
The points in this moduli space are $J$-holomorphic maps $u:\Sigma \to \hat{X}$ equipped with asymptotic markers that satisfy
\[g(\bar{\Sigma}) = g, \qquad [u] = A \quad\text{and}\quad u \to \Gamma_\pm \text{ as } r\to \pm \infty,\]
modulo the relation that $u$ is equivalent to $u'$ if there is a biholomorphism $\phi:\Sigma' \simeq \Sigma$ such that
\[u' = u \circ \phi\]
and the induced $\phi^*$'s respect the asymptotic markers. We refer to such an equivalence class $u$ as a \emph{$J$-holomorphic curve} (with asymptotic markers). 

\vspace{3pt}

In the case where $\hat{X} = \hat{Y}$ is the symplectization of a contact manifold $Y$ and $\hat{J}$ is the symplectization of a compatible almost complex structure on $\xi$, the moduli space admits a natural $\R$-action given by $\R$-translation in $\R \times Y$. Then we adopt the notation
\[
\mathcal{M}_{g,A}(Y,J;\Gamma_+,\Gamma_-) := \mathcal{M}_{g,A}(\R \times Y,\hat{J};\Gamma_+,\Gamma_-)/\R.
\]
\begin{example}A \emph{trivial cylinder} in a symplectization $(\hat{Y},\hat{J})$ is any map $u:\R \times S^1 \to \hat{Y}$ of the form $u(s,t) = (s,\gamma(T \cdot t))$ where $\gamma$ is a closed Reeb orbit of period $T$. 
\end{example}

\subsubsection{Holomorphic Buildings} A \emph{$J$-holomorphic building} $\bar{u}$ in a contact manifold $Y$ from $\Gamma_+$ to $\Gamma_-$ is a finite sequence of orbit tuples
\[\Gamma_1,\dots,\Gamma_m \qquad\text{with}\qquad \Gamma_+ = \Gamma_1 \quad\text{and}\quad\Gamma_- = \Gamma_m,\]
and a sequence of finite energy $J$-holomorphic maps in $\hat{Y}$, called the \emph{levels} of $\bar{u}$, denoted by
\[u_i:\Sigma_i \to \hat{Y} \qquad\text{with}\qquad u_i \to \Gamma_{i+1} \text{ at }-\infty \quad\text{and}\quad u_i \to \Gamma_{i} \text{ at }+\infty,\]
where each level $u_i$ is non-trivial, i.e., not a union of trivial cylinders. Moreover, the levels asymptotic to $\Gamma_i$ are equipped with asymptotic markers for all $i\in\{1,\dots,m\}$, and the (pairs of) punctures of the levels asymptotic to orbits in $\Gamma_i$ for $i = 2,\dots,m-1$ are equipped with matching isomorphisms.

\vspace{3pt}

Two buildings $\bar{u}$ and $\bar{v}$ are equivalent if, up to $\R$-translations, there is a biholomorphism of domains on each level respecting the holomorphic map, asymptotic markers and matching isomorphisms. 

\vspace{3pt}

Any building $\bar{u}$ has a well-defined genus $g = g(\bar{u})$ and homology class $[\bar{u}] = A$ determined by gluing the curves along the matching punctures. In particular, the homology class is given by
\[A = \sum_i A_i \in S(Y;\Gamma_+,\Gamma_-).\]
The moduli space of equivalence classes of $J$-holomorphic buildings in $Y$ from $\Gamma_+$ to $\Gamma_-$ of genus $g$ and homology class $A$ is denoted by
\[\overline{\mathcal{M}}_{g,A}(Y,J;\Gamma_+,\Gamma_-).\]

More generally, a $J$-holomorphic building $\bar{u}$ in a symplectic cobordism $X:Y_+ \to Y_-$ from $\Gamma_+$ to $\Gamma_-$ is a sequence of Reeb orbit tuples
\[\Gamma^+_1,\dots,\Gamma^+_a,\Gamma^-_1,\dots,\Gamma^-_b \qquad\text{with}\qquad \Gamma_+ = \Gamma^+_1\text{ and }\Gamma_- = \Gamma^-_b\]
and a sequence of finite energy $J$-holomorphic maps of the form
\begin{align*}
u^\pm_i&:\Sigma^\pm_i \to \hat{Y}_\pm, \qquad\text{with}\qquad u^\pm_i\to \Gamma^\pm_i \text{ at }+ \infty \quad\text{and}\quad u^\pm_i \to \Gamma^\pm_{i+1} \text{ at }- \infty,\\
u_X&:\Sigma_X \to \hat{X}, \qquad\text{with}\qquad u_X\to \Gamma^+_a \text{ at }+\infty \quad\text{and}\quad u_X \to \Gamma^-_1 \text{ at }-\infty,
\end{align*}
equipped with the same asymptotic markers and matching isomorphisms as in the symplectization case. Equivalence of a pair of buildings is defined as in the symplectization case, but we only quotient by the $\R$-direction in the symplectization levels. The moduli space of $J$-holomorphic buildings in $X$ from $\Gamma_+$ to $\Gamma_-$ of genus $g$ and homology class $A$ is denoted by
\[\overline{\mathcal{M}}_{g,A}(X,J;\Gamma_+,\Gamma_-).\]
The moduli spaces $\overline{\mathcal{M}}_{g,A}(Y,J;\Gamma_+,\Gamma_-)$ and $\overline{\mathcal{M}}_{g,A}(X,J;\Gamma_+,\Gamma_-)$ admit a Gromov topology described by \cite[\S 9.1]{sftCompactness}. Moreover, both spaces are compact \cite[\S 10.1]{sftCompactness}.

\begin{figure}
\centering
\includegraphics[width=\textwidth]{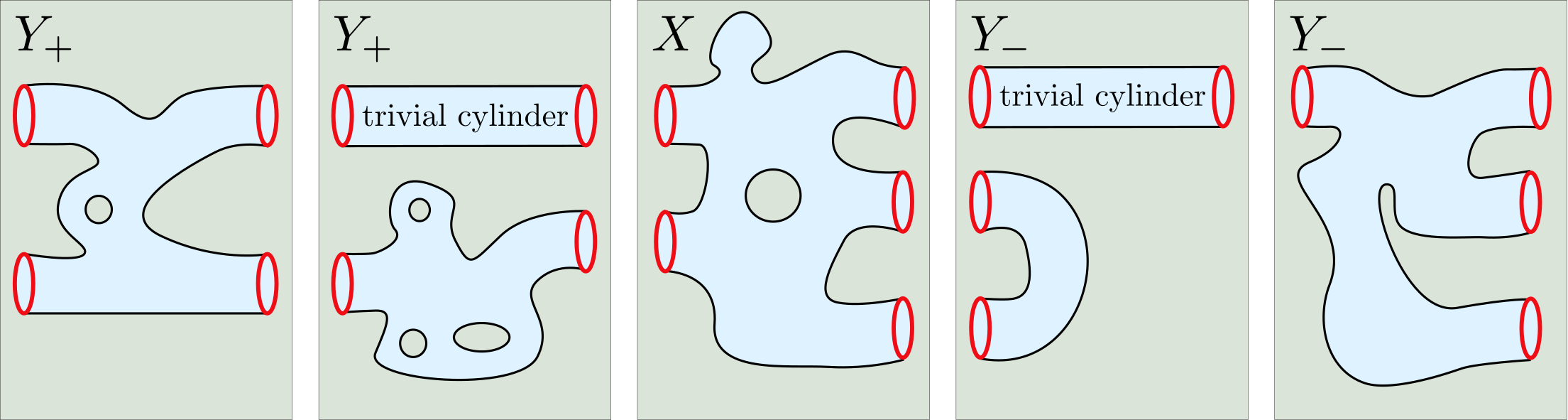}
\caption{A simplified picture of a possible holomorphic building. Note that all of the holomorphic buildings and curves of interest in this paper will be genus $0$.}
\label{fig:sft_building}
\end{figure}

\subsubsection{Marked Moduli Spaces} $J$-holomorphic curves and buildings can be decorated to include marked points. More precisely, we can formulate a moduli space
\[\mathcal{M}_{g,A,m}(X,J;\Gamma_+,\Gamma_-).\]
The points in this moduli space are genus $g$ $J$-holomorphic maps $u:\Sigma \to \hat{X}$ in homology class $A$ (with asymptotic markers) and an ordered tuple of $m$ \emph{marked points} $s_1,\dots,s_m$ in $\Sigma$, modulo reparametrizations that respect the marked points. There is an \emph{evaluation map}
\[\on{ev}:\mathcal{M}_{g,A,m}(X,J;\Gamma_+,\Gamma_-) \to \hat{X}^m\]
that takes an equivalence class $[u,s_1,\dots,s_m]$ to the point $(u(s_1),\dots,u(s_m)) \in \hat{X}^m$. One may also form a moduli space of buildings
\[\overline{\mathcal{M}}_{g,A,m}(X,J;\Gamma_+,\Gamma_-)\]
of J-holomorphic buildings where each level is a $J$-holomorphic curve with marked points and the total number of marked points over all the levels is $m$. This moduli space is compact with respect to the topology in \cite{sftCompactness}.

\subsubsection{Parametric Moduli Spaces} \label{subsubsec:parametric_moduli_spaces} Let $P$ be a compact manifold with boundary and let $J_P$ be a $P$-parameter family of compactible complex structures on $X$, consisting of a compatible complex structure $J_p$ for each $p \in P$. 
\vspace{3pt}

There is a \emph{parametric moduli space} of $J_p$-holomorphic curves ranging over all $p$, namely,
\[\mathcal{M}_{g,A,m}(X,J_P;\Gamma_+,\Gamma_-) := \bigcup_{p \in P} \{p\} \times  \mathcal{M}_{g,A,m}(X,J_P;\Gamma_+,\Gamma_-)\]
That is, a point in this moduli space is a pair consisting of a point $p \in P$ and a $J_p$-holomorphic curve (with marked points). Likewise, there is a compactified moduli space of buildings
\[\overline{\mathcal{M}}_{g,A,m}(X,J_P;\Gamma_+,\Gamma_-) := \bigcup_{p \in P} \{p\} \times  \overline{\mathcal{M}}_{g,A,m}(X,J_P;\Gamma_+,\Gamma_-).\]
These parametric moduli spaces inherit evaluation maps $\on{ev}$ and an additional continuous projection map
\[\pi:\overline{\mathcal{M}}_{g,A,m}(X,J_P;\Gamma_+,\Gamma_-) \to P, \qquad (p,u) \mapsto p\]
given by projection to the $P$-factor. We will consider parametric moduli spaces of buildings with marked points in \S \ref{subsec:Umap_moduli_space_vs_point_moduli_space}. 

\subsubsection{Generic Tranversality} \label{subsubsec:generic_transversality} Let $P$ be a compact manifold with boundary $\partial P$ and let $J_P:=\{J_p\}_{p\in P}$ be a $P$-family of compatible almost complex structures on $X$. We now briefly review some generic transversality results that are standard in the literature on SFT. 

\vspace{3pt}

Recall that a point $(p,[u,s_1,\dots,s_m])$ in the moduli space
\[\mathcal{M}_{g,A,m}(X,J_P;\Gamma_+,\Gamma_-)\]
is called \emph{parametrically regular} or \emph{parametrically transverse} if the parametric linearized operator $D_{u,p}$, incorporating both variations in the map $u$ and variations in the parameter space $P$, is surjective. 

\begin{prop} (cf. \cite[Thm. 7.1 and Rmk. 7.4]{wendlsymplectic}) \label{prop:moduli_regularity} The set of parametrically regular points
\[
\mathcal{M}_{\on{reg}} \subset \mathcal{M}_{g,A,m}(X,J_P;\Gamma_+,\Gamma_-)
\]
is an open set and a smooth orbifold of dimension
\[\on{dim}(\mathcal{M}_{\on{reg}}) = \on{vdim}(\mathcal{M}_{g,A,m}(X,J_P;\Gamma_+,\Gamma_-)) := \on{ind}(A,g) + \on{dim}(P) + 2m.\]
The local isotropy group at an orbifold point $(s,[u,s_1,\dots,s_m]) \in \mathcal{M}_{\on{reg}}$ is given by
\[\on{Aut}(p,u,s_1,\dots,_m) := \{\phi:\Sigma \to \Sigma \; : \; j \circ d\phi = d\phi \circ j,  u \circ \phi = u, u(s_i) = s_i\}.\]
Finally, the evaluation map $\on{ev}$ and projection map $\pi$ are both smooth on $\mathcal{M}_{\on{reg}}$. 
\end{prop}

\noindent As a special case, an unparametrized $J$-holomorphic curve $[u,s_1,\dots,s_m])$ is simply called \emph{regular} if it is parametrically regular with respect to the $0$-parameter family $J$.

\vspace{3pt}

Given a compact, closed submanifold $Z \subset \hat{X}^m$, a parametrically regular $(p,[u_1,\dots,u_m])$ is \emph{parametrically $Z$-regular} if the evaluation map
\[\on{ev}: \mathcal{M}_{g,A,m}(X,J_P;\Gamma_+,\Gamma_-) \to \hat{X}^m\]
from the parametric moduli space is transverse to $Z$ at $u$.

\begin{prop} \label{prop:generic_transversality} There is a comeager set $\mathcal{J}_{\on{reg}}(X,P)$ of $P$-families of compatible almost complex structures $J_P$ such that the space of somewhere injective curves
\[\mathcal{M}^{\on{i}}_{g,A,m}(X,J_P;\Gamma_+,\Gamma_-) := \big\{u \in \mathcal{M}_{g,A,m}(X,J_P;\Gamma_+,\Gamma) \; : \; u \text{ is somewhere injective}\big\}\]
consists of (parametrically) $Z$-regular curves.
\end{prop}

The parametric regularity part of this result (without accounting for the evaluation map) is proven in \cite[Thm. 7.1-7.2, Rmk. 7.4]{wendlsymplectic}. The transversality of the evaluation map is proven in \cite[\S 4.6]{wendlsymplectic} for the case of closed curves in a closed symplectic manifold $X$. The approach used in \cite[\S 4.6]{wendlsymplectic} is a standard one, using the Sard-Smale theorem, and can be adapted to the symplectization case with minimal modifications. Mainly, we need to work in the appropriate analytic set up, for example, by working with weighted Sobolev spaces instead of Sobolev spaces.

\subsubsection{Buildings in Contact Homology} \label{subsubsec:buildings_in_contact_homology} We primarily consider holomorphic buildings arising in contact homology. These are genus $0$ buildings with a single positive end in an exact cobordism $(X,\lambda)$ or in (the symplectization of) a contact manifold $Y$. These are curves in the moduli spaces
\begin{equation} \label{eqn:exact_genus_0_cobordism}
\overline{\mathcal{M}}_{0,A}(X,J;\gamma,\Gamma) \qquad\text{and}\qquad \overline{\mathcal{M}}_{0,A}(Y,J;\gamma,\Gamma).
\end{equation}

\begin{lemma} Let $\bar{u}$ be a $J$-holomorphic building in one of the moduli spaces (\ref{eqn:exact_genus_0_cobordism}). Then, each level $u_i$ is a disjoint union of curves
\[v:\Sigma \to \hat{Y}_\pm \qquad \text{or}\qquad v:\Sigma \to \hat{X}\]
where $\Sigma$ is connected and of genus $0$, and $v$ has exactly one positive puncture.
\end{lemma}

\begin{proof} Let $\bar{u} = (u_1,\dots,u_m)$ be a holomorphic building in (the symplectization of) $Y$. The case of buildings in a cobordism is similar. Let $\bar{u}_k$ denote the building
\[\bar{u}_k = (u_1,\dots,u_k) \qquad\text{for}\qquad 1 \le k \le m.\]
We prove by induction on $k$ that each building $\bar{u}_k$ is connected and that each component curve $v$ satisfies the conclusion of the lemma. 

\vspace{3pt}

For the base case, note that every component $v$ of every level $u_i$ is genus $0$, since otherwise the entire building would have positive genus. Moreover, since $\hat{X}$ and $\hat{Y}$ are both exact, every non-constant finite energy holomorphic curve $v$ must have at least one positive puncture. Thus, the top level $\bar{u}_1 = u_1$ is a connected genus $0$ curve with one positive puncture.

\vspace{3pt}

For the induction case, assume that $\bar{u}_k$ satisfies the induction hypothesis. By the above reasoning, each component $v$ of the level $u_{k+1}$ is genus $0$ with at least one positive puncture. This positive puncture must connect to a negative puncture of $\bar{u}_k$ so that $\bar{u}_{k+1}$ is connected. If a component $v$ has more than one positive puncture, then attaching $v$ to the connected building $\bar{u}_k$ contributes genus to $\bar{u}_{k+1}$. Therefore, $v$ has exactly one puncture. \end{proof}

\begin{remark} In \cite{p2015}, a slightly different compactification of $\mathcal{M} = \mathcal{M}_{g,A}(X,J;\Gamma_+,\Gamma_-)$ is used (and similarly for $\mathcal{M}_{g,A}(Y,J;\Gamma_+,\Gamma_-)$). In Pardon's compactification, if $u_i:\Sigma_i \to \hat{Y}_\pm$ is a symplectization level of a building $\bar{u}$ and $\Sigma_i$ breaks into disconnected components $\Sigma_{i,1},\dots,\Sigma_{i,k}$, then each component $u_i|_{\Sigma_{i,j}}$ is separately regarded as a holomorphic curve modulo translation and any trivial cylinder components are eliminated.

\vspace{3pt}

The differences between the BEHWZ compactification of \cite{sftCompactness} and the Pardon compactification \cite{p2015} will not be important for this paper. In particular, we will treat them as equivalent in Constructions \ref{con:contact_dga_algebra} and \ref{con:contact_dga_map} below.

\vspace{3pt}

However, we do note that any BEHWZ building corresponds to a unique Pardon building by only remembering the constituent maps of the building on each connected component of the levels (and eliminating the trivial components of every level). Conversely, any Pardon building can be lifted to a BEHWZ building by adding trivial levels and grouping the connected components of the building appropriately. In particular
\[\mathcal{M} = \overline{\mathcal{M}} \text{ in the BEHWZ topology} \qquad\iff\qquad \mathcal{M} = \overline{\mathcal{M}} \text{ in the Pardon topology}.\]
\end{remark}

\subsection{Basic Formalism} We can now discuss the basic construction of the contact dg-algebra of a contact manifold and the cobordism map of an exact cobordism. This construction was introduced in \cite{egh2000}. Here we discuss the specific foundational setup of Pardon \cite{p2015}.

\begin{construction} \label{con:contact_dga_algebra} The \emph{contact dga} of a closed contact manifold $Y$ with a non-degenerate contact form $\alpha$, denoted by
\[(A(Y,\alpha),\partial_{J,\theta}) \qquad \text{or more simply}\qquad A(Y)\] 
is the filtered dg-algebra formulated as follows. Associate a generator $x_\gamma$ to each good Reeb orbit $\gamma$ (see \cite[Definition $2.49$]{p2015} for a definition). Each generator $x_\gamma$ is given a standard SFT grading and action filtration
\begin{equation}|x_\gamma| = -\on{CZ}(\gamma) - n + 3 \mod 2\quad\text{and}\quad\mathcal{A}(x_\gamma) = \int_\gamma \alpha, \qquad  \text{respectively.} \end{equation}
The algebra $A(Y,\alpha)$ is the graded-symmetric algebra freely generated by these generators,
\begin{equation}
A(Y,\alpha) := \on{Sym}_\bullet\big[x_\gamma \; : \; \gamma\text{ is a good orbit}\big].
\end{equation}
The SFT-grading $|\cdot|$ and the action filtration $\mathcal{A}$ on $x_{\Gamma} = x_{\gamma_1}\dots x_{\gamma_k}$  are given by
\[
|x_\Gamma| := \sum_i |\gamma_i| \qquad \text{and}\qquad \mathcal{A}(x_\Gamma) := \sum_i \mathcal{A}(\gamma_i).
\]
We let $A^L(Y,\alpha) \subset A(Y,\alpha)$ denote the graded subspace given by
\[A^L(Y,\alpha) := \Q\langle x_\Gamma \; : \; \mathcal{A}(x_\Gamma) \le L\rangle\]
The differential on $A(Y,\alpha)$ is the unique derivation such that, for any good orbit $\gamma$, we have
\[\partial_{J,\theta}(x_\gamma) := \sum_{A,\Gamma} \frac{\#_\theta \overline{\mathcal{M}}_{0,A}(Y,J;\gamma,\Gamma)}{\mu_\Gamma \cdot \kappa_\Gamma} \cdot x_\Gamma.\]
Here, $\#_\theta\overline{\mathcal{M}}_{0,A}(Y,J;\gamma,\Gamma)$ is a (virtual) point count of index $1$ holomorphic buildings in the symplectization of $Y$ with one positive puncture at $\gamma$ and negative punctures at $\Gamma$ (see \cite{p2015}). This count depends on a choice of the VFC data $\theta$. The sum is over all ordered lists $\Gamma$ of good orbits and all homology classes $A \in S(Y;\gamma,\Gamma)$ such that $\on{ind}(A,0) = 1$.
\end{construction} 

\begin{construction} \label{con:contact_dga_map} The \emph{cobordism map} of an exact cobordism $X:(Y_+,\alpha_+) \to (Y_-,\alpha_-)$, denoted
\[\Phi_{X,\lambda,J,\theta}:A(Y_+,\alpha_+) \to A(Y_-,\alpha_-), \qquad\text{or more simply}\quad \Phi_X,\]
is the unique filtered dg-algebra map such that, for any good closed orbit $\gamma$ of $Y_+$, we have
\[\Phi_{X,\lambda,J,\theta}(x_\gamma) := \sum_{A,\Gamma} \frac{\#_\theta \overline{\mathcal{M}}_{0,A}(X,J;\gamma,\Gamma)}{\mu_\Gamma \cdot \kappa_\Gamma} \cdot x_\Gamma.\]
Here $\#\overline{\mathcal{M}}_{0,A}(X,J;\gamma,\Gamma)$ is a (virtual) point count of index $0$ holomorphic buildings in the completion of $X$ with 1 positive puncture at $\gamma$ and negative punctures at $\Gamma$ (see \cite{p2015}), and the sum is over all ordered lists $\Gamma$ of good orbits in $Y_-$ and all homology classes $A \in S(X;\gamma,\Gamma)$ such that $\on{ind}(A,0) = 0$.
\end{construction}

\begin{remark} \label{rmk:CH_moduli_spaces_transverse} In both Constructions \ref{con:contact_dga_algebra} and \ref{con:contact_dga_map}, the virtual point count is equal to an actual (oriented) point count when the relevant moduli space $\mathcal{M}$ is compact (i.e. $\overline{\mathcal{M}} = \mathcal{M}$) and each holomorphic curve is transversely cut out (i.e. the relevant linearized $\overline{\partial}$ operator is surjective). See \cite[Thm. 1.1(iv)]{p2015}.
\end{remark}

The main results of Pardon's construction \cite{p2015} can now be summarized as follows.

\begin{thm} \cite{p2015} \label{thm:contact_homology_well_defined} Let $(Y,\alpha)$ and $X:(Y_+,\alpha_+) \to (Y_-,\alpha_-)$ be as in Construction \ref{con:contact_dga_algebra}-\ref{con:contact_dga_map}. Then the following hold.
\begin{itemize}
\item[(a)] The map $\partial_{J,\theta}:A(Y) \to A(Y)$ is a filtered differential. That is,
\[\partial_{J,\theta}^2 = 0 \qquad\text{and}\qquad \mathcal{A}(\partial_{J,\theta}(x)) \le \mathcal{A}(x).\]
\item[(b)] The map $\Phi_{X,\lambda,J,\theta}:A(Y_+) \to A(Y_-)$ is a filtered chain map. That is,
\[\Phi_{X,\lambda,J,\theta} \circ \partial_{J_+,\theta_+} = \partial_{J_-,\theta_-} \circ \Phi_{X,\lambda,J,\theta} \qquad \text{and}\qquad \mathcal{A}(\Phi_{X,\lambda,J,\theta}(x)) \le \mathcal{A}(x).\]
 Furthermore, $\Phi_{X,\lambda,J,\theta}$ is independent of $J$ and $\theta$ up to filtered chain homotopy.
 \vspace{3pt}
\item[(c)] The composition of cobordism maps is filtered homotopic to the cobordism map of the composition,
\[
\Phi_{X \circ X'} \simeq \Phi_X \circ \Phi_{X'}. 
\]
\item[(d)] If $(X,\lambda)$ is the trivial cobordism $X = [a,b] \times Y, \lambda = e^t\alpha$, $J$ is a translation invariant almost complex structure induced by a compatible complex structure on $\xi$, and $\theta$ is any VFC data, then
\[\Phi_{X,\lambda,J,\theta}:H(A(Y),\partial_{J,\theta}) \to H(A(Y),\partial_{J,\theta})\]
is the identity map on the level of unfiltered graded dg-algebras.
\end{itemize}
\end{thm}

By Theorem \ref{thm:contact_homology_well_defined}, we can now define contact homology as an invariant of contact manifolds.

\begin{definition} \label{def:contact_homology} The \emph{contact homology} $CH(Y)$ of a closed contact manifold $(Y,\xi)$ is given by
\[CH(Y) := H(A(Y),\partial_{J,\theta}), \qquad\text{for any choice of }\alpha,J,\theta.\]
The map $\Phi_X:CH(Y_+) \to CH(Y_-)$ induced by an exact cobordism $X:Y_+ \to Y_-$ is similarly defined with respect to any choice of $J,\theta$. Any choice of contact form $\alpha$ induces a filtration of $CH(Y)$ by sub-spaces
\[CH^L(Y) \quad\text{or}\quad CH^L(Y,\alpha) \qquad\text{for any}\qquad L \in [0,\infty)\]
Here $CH^L(Y,\alpha)$ is the image of the map
\[H(A^L(Y)) \to H(A(Y)) = CH(Y)\]
if the contact form $\alpha$ is non-degenerate. In general, the $CH^L(Y)$ is defined as the colimit
\[
CH^L(Y,\alpha) = \on{colim}_{\beta} CH^L(Y,\beta)
\]
where the colimit is taken over all non-degenerate contact forms $\beta = f\alpha$ with $f > 1$ pointwise. The cobordism maps $\Phi_X$ are filtered with respect to this filtration. 
\end{definition}

\begin{remark} \label{rmk:CH_for_L_nondegenerate} More generally, given a contact form $\alpha$ on $Y$ that is non-degenerate below action $T$ and any $L < T$, we can still define the graded algebra
\[
A^L(Y) := \Q \langle x_\Gamma \; : \; \mathcal{A}(\Gamma) \le L\rangle.
\]
We may equip this algebra with a differential $\partial_{J,\theta}$ given a choice of compatible complex structure and VFC data. Likewise, if $X:(Y_+,\alpha_+) \to (Y_-,\alpha_-)$ are non-degenerate below action $T$, then we have a cobordism dg-algebra map
\[\Phi_X:A^L(Y_+) \to A^L(Y_-)\]
for any $L \le T$, well-defined up to filtered chain homotopy. In particular, $CH^L(Y,\alpha)$ is the image of the map
\[H(A^L(Y,\alpha)) \to CH(Y)\]
Moreover, for any contact forms $\alpha_+$ and $\alpha_-$ that are non-degenerate up to action $T$, the following diagram commutes.
\begin{equation} \label{eqn:CHL_diagram}
\begin{array}{ccc}
        H(A^L(Y_+)) & \xrightarrow{\Phi_X} &  H(A^L(Y_-))\\
        & & \\
         \Phi\Big\downarrow\hspace{0.4cm} & & \Big\downarrow \Phi\\
         & & \\
         CH^L(Y_+) & \xrightarrow[\Phi_X]{} &  CH^L(Y_-)
    \end{array}.\end{equation}
\end{remark}

\subsection{The Tight Sphere} We now calculate the contact homology algebra of the standard sphere, which is a key example in later constructions. 

\vspace{3pt}

Consider $\C^n$ equipped with the standard Liouville form and associated Liouville vector-field
\[
\lambda_{\on{std}} = \frac{1}{2} \sum_i x_i dy_i - y_i dx_i, \text{and} \qquad Z_{\on{std}} = \frac{1}{2} \sum_i x_i \partial_{x_i} + y_i \partial_{y_i}.
\]

\begin{definition} A \emph{star-shaped domain} $W \subset \C^n$ is an embedded Liouville sub-domain of $\C^n$. 
Equivalently, $W$ is a codimension zero submanifold with smooth boundary that is transverse to the Liouville vector field $Z$, $Z\pitchfork\partial W$. \end{definition}

\noindent Every pair of star-shaped domains $W$ and $W'$ are equivalent through a canonical deformation, i.e., by deforming the boundary along the radial direction. In particular, the contact boundaries are all contactomorphic to the standard tight sphere.

\begin{definition} The \emph{standard tight sphere} $(S^{2n-1},\xi_{\on{std}})$ is the unit sphere $S^{2n-1} \subset \C^n$ equipped with the contact structure $\xi_{\on{std}} = \on{ker}(\lambda_{\on{std}}|_{S^{2n-1}})$. \end{definition} 

\noindent If $W \subset W'$ is an inclusion of star-shaped domains, then the exact symplectic cobordism
\[
X:(\partial W',\lambda|_{\partial W'}) \to (\partial W,\lambda|_{\partial W}) \qquad\text{given by}\qquad X := W' \setminus W
\]
is isomorphic (as an exact cobordism) to a cylindrical cobordism $X:(S^{2n-1},\alpha) \to (S^{2n-1},\beta)$ for a pair of contact forms $\alpha$ and $\beta$ on $(S^{2n-1},\xi)$.

\begin{example} \label{ex:irrational_ellipsoid_orbits} Choose a sequence of rationally independent, positive real numbers
\[0 < a_1 \le a_2 \le \dots \le a_n.\]
The \emph{standard ellipsoid} $E = E(a_1,\dots,a_n) \subset \C^n$ is the star-shaped domain given by
\[E(a_1,\dots,a_n) = \{z \in \C^n \; : \; \sum_i \frac{\pi}{a_i} \cdot |z_i|^2 \le 1\}.\]

The Reeb dynamics on $\partial E$ is very explicit and easy to determine (cf. \cite[\S 2.1]{gh_sym_18}). Specifically, there are exactly $n$ non-degenerate, simple, closed Reeb orbits given by
\[
\gamma_i = \partial E \cap \C_i \qquad\text{for}\qquad i = 1,\dots,n.
\]
Here $\C_i \subset \C^n$ is the $i$th complex axis in $\C^n$. Every Reeb orbit $\gamma$ is an iterate of one of these orbits. The action and Conley-Zehnder index of an orbit $\gamma = \gamma^j_i$ is given by
\[
\mathcal{A}(\gamma) = j \cdot a_i \qquad\text{and}\qquad \CZ(\gamma) = n - 1 + 2 \cdot \#\{\text{orbits $\eta$ of }\partial E \; : \; \mathcal{A}(\eta) \le \mathcal{A}(\gamma)\}.
\]
Note that the Conley-Zehnder index is well-defined without reference to a trivialization, since $c_1(S^3,\xi_{\on{std}}) = 0$ and $\pi_1(S^3) = 0$.
\end{example}

\begin{lemma} \label{lem:CH_of_sphere} The contact homology algebra $CH(S^{2n-1},\xi_{\on{std}})$ is isomorphic to a graded-symmetric algebra freely generated by generators $x_k$ of grading $-2n + 2 - 2(k-1)$ for each $k \ge 1$, namely,
\[CH(S^{2n-1},\xi_{\on{std}}) \simeq \on{Sym}_\bullet\big[x_k \; ; \; k \ge 1\big].\]\end{lemma}

\begin{proof} Consider $S^{2n-1}$ equipped with the contact form induced as the boundary of an irrational ellipsoid $E(a_1,\dots,a_n)$. The (cohomological) SFT grading of a closed Reeb orbit $\gamma$ is given by 
\[|\gamma| = -2n + 2 - 2 \cdot \#\{\text{orbits $\eta$ of }\partial E \; : \; \mathcal{A}(\eta) \le \mathcal{A}(\gamma)\}.\]
In particular, the SFT grading is even for all generators of $A(\partial E)$ and the differential is trivial.
\end{proof}
\begin{remark}
Note that the isomorphism in Lemma~\ref{lem:CH_of_sphere} is not claimed to be canonical.
\end{remark}

In \S \ref{sec:proof_of_main_thm}, we will require the following property of cobordism maps given by inclusion.

\begin{lemma}\label{lem:Phi_inv_wl_decreasing}
    Let $E(a)\subset E(b)$ be two ellipsoids and consider the map $\Phi:A(\partial E(a))\rightarrow A(\partial E(b))$ corresponding to the cobordism $E(b)\setminus E(a)$. Then, $\Phi$ is an isomorphism on the chain level and its inverse is word-length non-decreasing.
\end{lemma}
\begin{proof}
The exact cobordism $E(b)\setminus E(a)$ is isomorphic to a cylindrical cobordism, and so the map
\[\Phi:A(\partial E(a))\rightarrow A(\partial E(b))\]
is a quasi-isomorphism inducing the natural isomorphism on homology. Since the differentials of $A(\partial E(a))$ and $A(\partial E(b))$ are trivial, $\Phi$ is in fact an isomorphism of dg-algebras. Moreover, by Theorem~\ref{thm:contact_homology_well_defined}(c)-(d), the inverse $\Phi^{-1}$ is the cobordism map induced by $E(b) \setminus E(c \cdot a)$ for any $c > 0$ sufficiently small. 

\vspace{3pt}

It remains to show that cobordism maps between ellipsoids are word-length non-decreasing. Since cobordism maps are algebra maps, this is equivalent to having a non-constant generator being mapped to the constants. In the case of the ellipsoids this is impossible, since the constants have grading zero, the non-constant generators lie in positive degrees and the cylindrical cobordism maps preserve the $\Z$-grading.
\end{proof}

\section{Spectral Gaps} \label{sec:spectral_gaps} In this section, we discuss contact homology spectral gaps and related structures, including abstract constraints, constrained cobordism maps and spectral invariants.

\subsection{Abstract Constraints} \label{subsec:abstract_constraints} An abstract constraint provides a purely homological tool for tracking the ways in which a holomorphic curve can be tangent to (or asymptotic to) a set of points in a symplectic cobordism. Rigorously, we have the following definition.

\begin{definition} An \emph{abstract constraint} $P$ \emph{in contact homology} with $m$ points, dimension $n$ and codimension $\on{codim}(P) = k$ is a degree $k$ map
\[P:\bigotimes_{i=1}^m CH(S^{2n-1},\xi_{\on{std}}) \to \Q[k],\]
or equivalently, a cohomology class $P \in CH(\cup_1^m S^{2n-1})^\vee$ of grading $\on{codim}(P) = k$.
\end{definition} 

\begin{example}[Empty Constraint] The \emph{empty constraint} $P_\emptyset$ is the codimension $0$ map given by
\[
P_\emptyset(1) = 1 \qquad\text{and}\qquad P_\emptyset(x) = 0 \quad\text{if}\quad |x| > 0.
\]
Alternatively, $P_\emptyset$ is the unique $\Z$-graded algebra map $CH(S^{2n-1},\xi_{\on{std}}) \to \Q$.
\end{example}

\begin{example}[Tangency Constraints]\label{exa:tangency_constraints} It follows from Lemma \ref{lem:CH_of_sphere} that
\[
\on{ker}(P_\emptyset)/\on{ker}(P_\emptyset)^2 \simeq \bigoplus_{k=0}^\infty \Q[2n -2 + 2k].
\]
In particular, there is (up to multiplication by a non-zero constant) a unique surjective map
\[
\Pi_k:\on{ker}(P_\emptyset)/\on{ker}(P_\emptyset)^2 \to \Q[2n - 2 + 2k] \qquad\text{for each }k \ge 0.
\]
Therefore, we have an abstract constraint (well-defined up to multiplication by a constant)
\[
P_k:CH(S^{2n-1},\xi_{\on{std}}) = \Q \oplus \on{ker}(P_\emptyset) \to \on{ker}(P_\emptyset) \to \on{ker}(P_\emptyset)/\on{ker}(P_\emptyset)^2 \xrightarrow{\Pi_k} \Q[2n - 2 + 2k].
\]
 \end{example}

 \begin{remark} As observed by Siegel \cite[\S 5.5]{sie_hig_19}, the constraint $P_k$ coincides with the map acquired by counting genus $0$ holomorphic curves  in a star-shaped domain, with one positive puncture, passing through a point $p$ and tangent to a local divisor $D$ through $p$ to order $k$. 
 
 \vspace{3pt}
 In \S \ref{sec:proof_of_main_thm}, we will use a version of this fact in the proof of our main results. We also give an alternate argument for a weak version of this correspondence (for certain curves counted in constrained cobordism maps and counts of cylinders passing through a point) in \S \ref{subsec:Umap_moduli_space_vs_point_moduli_space}.\end{remark}

 \begin{example}[Dual Constraints] \label{ex:dual_constraint} Let $E$ be an irrational ellipsoid and let $\Xi$ be an orbit tuple on $\partial E$. There is a dual abstract constraint
\[P_\Xi:CH(\partial E) \simeq CH(S^{2n-1},\xi_{\on{std}}) \to \Q[|\Xi|]\]
determined by $\Xi$. This is defined in the usual way, with
\[P_\Xi(x_\Gamma) = 1 \text{ if }\Gamma = \Xi \qquad\text{and}\qquad P_\Xi(x_\Gamma) = 0 \qquad \text{otherwise}\]
As a special case, the abstract constraint $P_k$ in Example \ref{exa:tangency_constraints} coincides with $P_{\gamma}$ where $\gamma$ is the unique closed orbit of $E$ with $|\gamma| = 2n - 2 + 2k$. 
 \end{example}

\subsection{Constrained Cobordism Maps} \label{subsec:constrained_cobordism_maps} By using abstract constraints, we can formulate a generalization of the cobordism maps in contact homology, which morally counts curves satisfying a number of tangency constraints.

\begin{definition} \label{def:constrained_cobordism_map} Let $X:Y_+ \to Y_-$ be a connected exact cobordism and let $P$ be an abstract constraint. The \emph{$P$-constrained cobordism map} 
\[\Phi_{X,P}:A(Y_+) \to A(Y_-)[\on{codim}(P)]\]
is the filtered chain map of degree $\on{codim}(P)$ constructed by the following procedure. Choose a star-shaped domain $W_i$ for $i = 1,\dots,m$ and an embedding
\[
\iota:W = W_1 \cup \dots \cup W_m \to \on{int}(X).
\]
Since $W_i$ are star-shaped domains, $\iota$ is automatically weakly exact (see \S \ref{subsec:symplectic_cobordisms}). Thus we may assume (after deformation) that $\iota$ is exact and $X \setminus W$ is an exact cobordism. Also, choose a cochain $\pi_{\partial W,P}$ representing $P$, i.e., a chain map
\[
\pi_{\partial W,P}:A(\partial W) \simeq A(\cup_1^m S^{2n-1}) \to \Q[\on{codim}(P)]
\]
which is in the cohomology class $P$. Then, we define $\Phi_{X,P}$ to be the composition
\[
A(Y_+) \xrightarrow{\Phi_{X \setminus W}} A(Y_- \cup \partial W) = A(Y_-) \otimes A(\partial W) \xrightarrow{\id \otimes \pi_{\partial W,P}} A(Y_-) \otimes \Q[\on{codim}(P)].
\]
\end{definition}

\begin{definition} \label{def:U_map} Let $P$ be an abstract tangency constraint and $Y$ be a connected, closed contact manifold. The \emph{U-map}
\[U_P:A(Y) \to A(Y)[\on{codim}(P)].\]
is the $P$-constrained cobordism map $\Phi_{X,P}$ where $X = [0,1] \times Y$ is the trivial cobordism.\end{definition}

\begin{figure}[h]
\centering
\includegraphics[width=.9\textwidth]{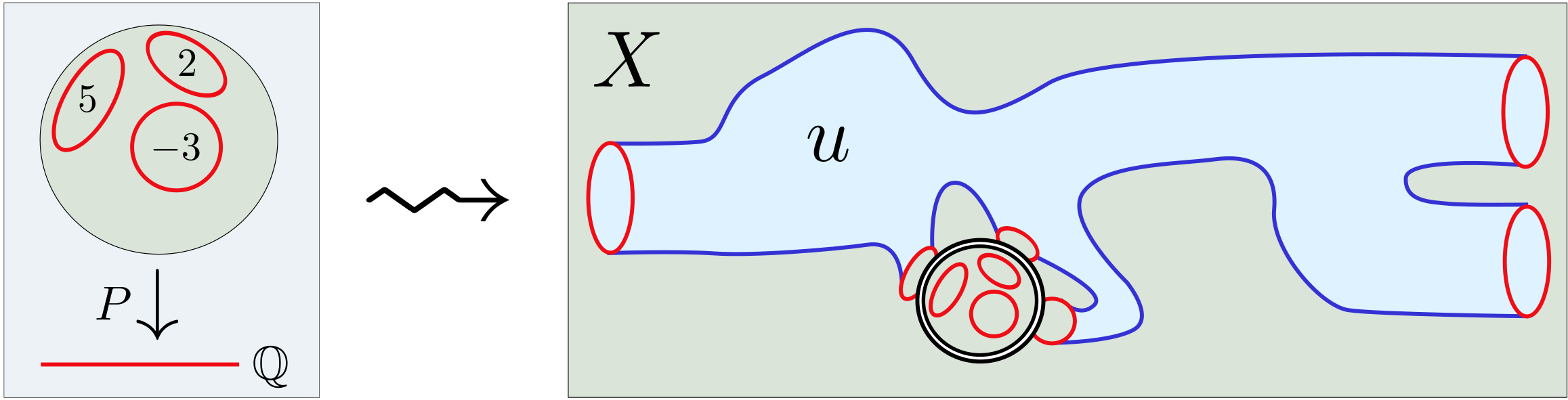}
\label{fig:abstract_constraint}
\caption{A cartoon of (a) an abstract constraint $P$, visualized as a weighting of the orbits on the sphere and (b) the curves counted in constrained cobordism maps, asymptotic to the orbits with non-zero weights under $P$.}
\end{figure}

\begin{remark} The constrained cobordism maps $\Phi_{X,P}$ are \emph{not} dg-algebra maps in general. However, some constrained cobordism maps are compatible with the algebra structure in other ways (see Lemma \ref{lem:Leibniz_rule}). \end{remark}

\begin{lemma}\label{lem:constrained_cob_maps_well_def} The cobordism maps $\Phi_{X,P}$ are well-defined up to filtered chain homotopy.
\end{lemma}

\begin{proof} Choose a disjoint union $W = W_1 \cup \dots \cup W_m$ of star-shaped domains $U_i$, an embedding $W \subset X$, Floer data $(J,\theta)$ on $X\setminus W$ and a cochain representative $\pi_{\partial W,P}$ of $P$. We adopt the notation
\[
\Phi_{X,P,W,J,\theta} := (\id_{Y_-} \otimes \pi_{\partial W,P}) \circ \Phi_{X\setminus W,J,\theta}.
\]

\noindent Since chain homotopy is a closed relation under composition, $\Phi_{X,P,W,J,\theta}$ is independent of the choice of $\pi_{\partial W,P}$ up to filtered chain homotopy. Note that we are viewing $\pi_{\partial W,P}$ as a filtered map by equipping $\Q[\on{codim}(P)]$ with the trivial filtration.

\vspace{3pt}

To show independence of $W$ and $(J,\theta)$ up to filtered homotopy, we start by considering two special cases and then move on to address the general case.
\vspace{3pt}

{\bf Case 1.} Let $\iota:[0,1] \times W \to X$ be a family of symplectic embeddings. This family $\iota$ induces a family of exact symplectic cobordisms
\[
X \setminus \iota_t(W):Y_+ \to Y_- \cup \partial W
\]
By Theorem \ref{thm:contact_homology_well_defined}(b), the induced cobordism maps of $X \setminus \iota_0(W)$ and $X \setminus \iota_1(W)$ are homotopic (for any choices of Floer data). Since filtered homotopy is a closed relation under composition with filtered chain maps, we see that
\[
\Phi_{X,P,\iota_0(W),J,\theta} = (\id_{Y_-} \otimes \pi_{\partial W,P}) \circ \Phi_{X\setminus \iota_0(W),J,\theta} \simeq (\id_{Y_-} \otimes \pi_{\partial W,P}) \circ \Phi_{X\setminus \iota_1(W),J',\theta'} = \Phi_{X,P,\iota_1(W),J',\theta'}
\]

{\bf Case 2.} Let $V = V_1 \cup \dots \cup V_m$ be a collection of star-shaped domains with inclusions $V_i \subset W_i$ that are strictly exact, i.e., that intertwine the Liouville forms. Let $W \setminus V:\partial W \to \partial V$ be the difference cobordism. Consider the cobordism map
\[
\Phi_{W \setminus V,I,\varphi}:A(\partial W) \to A(\partial V)
\]
for some choice of Floer data $(I,\varphi)$ on $W\setminus V$. We may choose the cochains $\pi_{\partial W,P}$ and $\pi_{\partial V,P}$ so that
\[
\pi_{\partial W,P} \simeq \pi_{\partial V, P} \circ \Phi_{W,I,\varphi}.
\]
On the other hand, by Theorem \ref{thm:contact_homology_well_defined}(c), we have 
\[
\Phi_{X\setminus V,J,\theta} \simeq \Phi_{(Y_- \times [0,1]) \cup W\setminus V} \circ \Phi_{X \setminus W,J',\theta'} = (\id_{Y_-} \otimes \Phi_{W\setminus V,I,\varphi}) \circ \Phi_{X \setminus W,J',\theta'}
\] 
for appropriate choices of Floer data $(J,\theta)$ on $X \setminus V$ and $(J',\theta')$ on $X \setminus W$. Therefore,
\begin{align*}
    \Phi_{X,P,V,J',\theta'} &=  (\id_{Y_-} \otimes\pi_{\partial V,P}) \circ \Phi_{X\setminus V,J',\theta'}\\
    &\simeq (\id_{Y_-} \otimes \pi_{\partial V,P}) \circ (\id_{Y_-} \otimes \Phi_{W\setminus V,I,\varphi}) \circ \Phi_{X \setminus W,J',\theta'}\\
    & \simeq (\id_{Y_-} \otimes \pi_{\partial W,P}) \circ \Phi_{X\setminus W,J,\theta} = \Phi_{X,P,W,J,\theta}.
\end{align*}

{\bf General Case.} Let $W = W_1 \cup \dots \cup W_m$ and $V = V_1 \cup \dots \cup V_m$ be any two choices of $m$ disjoint star-shaped domains in $X$. We may choose star-shaped domains $E_i$ (e.g. sufficiently small ellipsods) that include into $E_i \subset W_i$ and $E_i \subset V_i$, and such that the embeddings
\[
E_i \to W_i \to X \qquad\text{and}\qquad E_i \to V_i \to X
\]
are homotopic through a homotopy of embeddings $\iota:[0,1] \times E_i \to X$. Thus, by Cases 1 and 2, $\Phi_{X,P,W,J,\theta}$ and $\Phi_{X,P,V,J',\theta'}$ are filtered chain homotopic (for any Floer data). \end{proof}

As mentioned above, an abstract constraint $P$ is an assignment of numerical weights to the Reeb orbits on the boundary of a star-shaped domain $W$. Constrained cobordism maps are acquired by deleting $W$ from a cobordism $X$ and counting curves with ends on $\partial W$ with non-zero weights. 

\vspace{3pt}

Let us make this intuition precise in a specific case. Fix a connected exact cobordism $X:Y_+ \to Y_-$ between contact manifolds with non-degenerate contact forms. Let $E \subset X$ be an embedded irrational ellipsoid and let $\Xi$ be an orbit tuple in $\partial E$. Consider the exact cobordism
\[X \setminus E:Y_+ \to Y_- \cup \partial E\]
Finally, fix an orbit $\gamma$ of $Y_+$ and an orbit tuple $\Gamma$ of $Y_-$. The following lemma is immediate from Construction \ref{con:contact_dga_map}, Remark \ref{rmk:CH_moduli_spaces_transverse} and Definition \ref{def:constrained_cobordism_map}. 

\begin{lemma} \label{lem:constrained_map_curve_count} Let $P_\Xi$ be the dual constraint to $\Xi$ (see Example \ref{ex:dual_constraint}). Let $J$ be a compatible almost complex structure on $J$ such that
\[\mathcal{M}_{0,A}(X \setminus E;\gamma,\Gamma \cup \Xi)\]
is regular and compact in the BEHWZ topology for each homology class $A \in S(X \setminus E;\gamma,\Gamma \cup \Xi)$. Then the $x_\Gamma$-coefficient of $\Phi_{X,P_\Xi}(x_\gamma)$ is given by
\[
\langle x_\Gamma,\Phi_{X,P_\Xi}(x_\gamma)\rangle := \frac{1}{\mu_\Gamma \cdot \kappa_\Gamma} \cdot \sum_{A} \# \mathcal{M}_{0,A}(X \setminus E,J;\gamma,\Gamma \cup \Xi)
\]
Here the sum is over all classes with $\on{ind}(0,A) = 0$ and $\#$ denotes an oriented point count. \end{lemma}

Constrained cobordism maps satisfy a number of useful (and expected) axioms presented in the following lemma.

\begin{lemma} \label{lem:constrained_cobordism_map_composition} The constrained cobordism maps $\Phi_{X,P}$ satisfy the following properties.
\begin{itemize}
\item[A.] (Functoriality) If $X:Y \to Y'$ and $X':Y' \to Y''$ are two exact cobordisms, and $P$ and $Q$ are two tangency constraints, then
\[\Phi_{X \circ X', P \otimes Q} \simeq \Phi_{X,P} \circ \Phi_{X',Q} \qquad\text{and}\qquad U_Q \circ \Phi_{X,P} = \Phi_{X,P \otimes Q} = \Phi_{X,P} \circ U_Q.\]
\item[B.] (Additivity) If $X:Y \to Y'$ is an exact cobordism, and $P$ and $Q$ are two tangency constraints of the same dimension, then
\[\Phi_{X, P + Q} \simeq \Phi_{X,P} + \Phi_{X,Q}.\]
\item[C.] (Empty Constraint) Let $P_\emptyset$ be the empty constraint. Then
\[
\Phi_{X,P_\emptyset} = \Phi_X \qquad\text{and}\qquad U_{P_\emptyset} = \id.
\]
\end{itemize}
\end{lemma}

\begin{proof} It suffices to prove these properties for $\Phi_{X,P}$, as the $U$-maps are a special case.

\vspace{3pt}

 {\bf Axiom A.} Choose collections of disjoint star-shaped domains $W = W_1 \cup \dots \cup W_m \subset X$ and $V = V_1 \cup \dots \cup V_n \subset X'$. Up to deformation of exact cobordisms, we may write
\[
(X \circ X')\setminus (W \cup V) = \big((X \setminus W) \cup ([0,1] \times {\partial}V)\big) \circ (X' \setminus V).
\] 
Therefore, by Theorem \ref{thm:contact_homology_well_defined}(c) we have (up to filtered chain homotopy) the equivalence 
\[\Phi_{(X \circ X')\setminus (W \cup V)} \simeq \Phi_{(X \setminus W) \cup ({\partial}V \times [0,1])} \circ \Phi_{X' \setminus V} \simeq (\Phi_{X \setminus W} \otimes \id_{\partial V}) \circ \Phi_{X'\setminus V}.\]
We may choose the cochains representing $P,Q$ and $P \otimes Q$ to satisfy
\begin{equation} \label{eqn:constrained_composition_2}
\pi_{\partial(W \cup V), P \otimes Q} = \pi_{\partial W,P} \otimes \pi_{\partial V,Q}.
\end{equation}
This implies the desired composition property, by the following calculation.
\begin{align*}
    \Phi_{X \circ X', P \cup Q} &= (\id_{Y''} \otimes \pi_{\partial(W \cup V), P \cup Q}) \circ \Phi_{(X \circ X')\setminus (W \cup V)}\\
    &\simeq \big((\id_{Y''} \otimes \pi_{\partial W, P}) \otimes \pi_{\partial V, Q}\big) \circ (\Phi_{X \setminus W} \otimes \id_{\partial V}) \circ \Phi_{X'\setminus V}\\
    &= (\id_{Y''} \otimes \pi_{\partial W,P}) \circ \Phi_{X \setminus W} \circ (\id_{Y'} \otimes \pi_{\partial V.Q}) \circ \Phi_{X' \setminus V} = \Phi_{X,P} \circ \Phi_{X',Q}. 
\end{align*}

\vspace{3pt}

{\bf Axiom B.} Choose a collection of disjoint star-shaped domains $W = W_1 \cup \dots \cup W_m \subset X$. We may choose the cochain representatives of $P,Q,$ and $P + Q$ so that
\[\pi_{\partial W,P+Q} = \pi_{\partial W,P} + \pi_{\partial W,Q}.\]
Thus, we can calculate that
\begin{align*}
    \Phi_{X, P + Q} &= (\on{Id_{Y'}} \otimes (\pi_{\partial W,P} + \pi_{\partial W,Q})) \circ \Phi_{X\setminus W} \\
    &= (\id_{Y'} \otimes \pi_{\partial W,P}) \circ \Phi_{X\setminus W} + (\id_{Y'} \otimes \pi_{\partial W,Q}) \circ \Phi_{X\setminus W} = \Phi_{X,P} + \Phi_{X,Q}.
\end{align*}

\vspace{3pt}

{\bf Axiom C.} Any star-shaped domain $W$ determines a unital, $\Z$-graded dg-algebra map \[\Phi_W:CH(\partial W) \to CH(\emptyset) = \Q.\]
The cohomology class of this map is unique, and equal to $P_\emptyset$. Therefore,
\[
\Phi_{X,P_\emptyset} = (\id \otimes \Phi_W) \circ \Phi_{X\setminus W} = \Phi_X.
\qedhere \]
\end{proof}

The constrained cobordism maps $\Phi_{X,P_k}$ with respect to the tangency constraints $P_k$ defined in Example~\ref{exa:tangency_constraints} satisfy, in addition, a chain level Leibniz rule.

\begin{lemma} \label{lem:Leibniz_rule} The constrained cobordism map $\Phi_{X,P_k}$ of the tangency constraints $P_k$ satsifies
    \begin{equation}
        \Phi_{X,P_k}(xy) = \Phi_{X,P_k}(x)\cdot \Phi_X(y) + \Phi_X(x) \cdot \Phi_{X,P_k}(y). 
    \end{equation}
As a special case, the $U$-maps $U_{P_k}$ satisfy the Leibniz rule.
    \begin{equation}
        U_{P_k}(xy) = U_{P_k}(x)\cdot y + x \cdot U_{P_k}(y). 
    \end{equation}

\end{lemma}

\begin{proof} Let $W \subset X$ be an embedded irrational ellipsoid. Let $y_j$ denote the generators of $CH(\partial W)$. We first note that the map $\Phi_{X \setminus W}$ in Definition \ref{def:constrained_cobordism_map} satisfies
\[
\Phi_{X \setminus W}(z) = \Phi_X(z) \otimes 1 + \sum_{k =0 }^\infty \Phi_{X,P_k}(z) \otimes y_{k+1} + r
\]
where $r$ is a sum of factors $x \otimes y$ where $y$ is a monomial of word length $\ge 2$. This follows from Definition \ref{def:constrained_cobordism_map}, the fact that $P_k$ is the dual constraint to the $(k+1)$-th generator $y_{k+1}$ (see Example \ref{ex:dual_constraint}) and the fact that $P_\emptyset$ is dual to $1 \in CH(\partial W)$. Since $\Phi_{X \setminus W}$ is an algebra map, we thus have
\begin{equation} \label{eqn:PhiXW_leibniz_formula}
\Phi_{X \setminus W}(z z') = \Phi_X(zz') \otimes 1 + \sum_{k=1}^\infty (\Phi_{X,P_k}(z)\Phi_X(z') + \Phi_X(z)\Phi_{X,P_k}(z')) \otimes y_{k+1} + r'
\end{equation}
where $r'$ is a remainder term of the same form as $r$.  On the otherhand, the map $\Pi = \on{Id} \otimes \pi_{\partial W,P_k}$ from Definition \ref{def:constrained_cobordism_map} is given by
\begin{equation} \label{eqn:Pi_leibniz_formula}\Pi(x \otimes y_{k_1}\dots y_{k_m}) = \left\{
\begin{array}{cc}
x & \text{ if }m = 1, k_1 = k+1\\
0 & \text{ otherwise}\end{array}\right.\end{equation}
By Definition \ref{def:constrained_cobordism_map}, we have $\Phi_{X,P_k} = \Pi \circ \Phi_{X,W}$. Thus it follows from (\ref{eqn:PhiXW_leibniz_formula}) and (\ref{eqn:Pi_leibniz_formula}) that
\[
\Phi_{X,P_k}(zz') = \Pi\big((\Phi_{X,P_k}(z)\Phi_X(z') + \Phi_X(z)\Phi_{X,P_k}(z')) \otimes y_{k+1}\big) = \Phi_{X,P_k}(z)\Phi_X(z') + \Phi_X(z)\Phi_{X,P_k}(z')
\] which is the desired Leibniz rule.\end{proof}

\subsection{Spectral Invariants} \label{subsec:spectral_invariants} We now recall the definitions and properties of spectral invariants and capacities in the setting of contact homology.

\begin{definition} \label{def:CH_spectral_invariants} The \emph{contact homology spectral invariant} $\mathfrak{s}_\sigma(Y,\alpha)$ of a closed contact manifold $Y$ with contact form $\alpha$ and a class $\sigma \in CH(Y)$ is given by
\begin{equation} \label{eq:CH_spectral_invariants}
\mathfrak{s}_\sigma(Y,\alpha) := \mathcal{A}(\sigma) = \on{min}\{L \; : \; \sigma \in 
CH^L(Y) \subset CH(Y)\} \in (0,\infty).
\end{equation}
\end{definition}

\begin{definition} \label{def:CH_capacity} The \emph{contact homology capacity} $\mathfrak{c}_P(W)$ of a Liouville domain $(W,\lambda)$ and an abstract constraint $P$ is given by
\[
\mathfrak{c}_P(W) := \on{inf}\{\mathfrak{s}_\sigma(\partial W,\lambda|_{\partial W}) \; : \; \Phi_{W,P}(\sigma) \neq 0\}.
\]
Note that here we are viewing $W$ as an exact cobordism from $\partial W$ to $\emptyset$.
\end{definition}

\begin{remark} If $(Y,\alpha)$ is a closed contact manifold where $\alpha$ is non-degenerate, (\ref{eq:CH_spectral_invariants}) is equivalent to the minimum action of a cycle in the dg-algebra $A(Y)$ representing $\sigma$,
\[
\mathfrak{s}_\sigma(Y,\alpha) = \on{min}\{\mathcal{A}(x) \; : \; x \in A(Y) \text{ satisfying }\partial x = 0 \text{ and }[x] = \sigma\}.
\] \end{remark}

\begin{thm} \label{thm:rsftSpectralAxioms} The contact homology spectral invariants $\mathfrak{s}_{\sigma}$ satisfy the following properties.
\begin{itemize}
	\item[A.] (Conformality) If $(Y,\alpha)$ is a contact manifold with contact form $\alpha$ and $a > 0$ is a constant, then
    \[\mathfrak{s}_\sigma(Y,a \cdot \alpha) = a \cdot \mathfrak{s}_\sigma(Y,\alpha). \] 
    \item[B.] (Cobordism Map) If $X:(Y_+,\alpha_+) \to (Y_-,\alpha_-)$ is an exact symplectic cobordism, $P$ is an abstract constraint and $W \subset X$ is a (weakly) exactly embedded Liouville domain, then 
    \[\mathfrak{s}_{\Phi_{X,P}(\sigma)}(Y_-,\alpha_-) + \mathfrak{c}_{P}(W) \le \mathfrak{s}_{\sigma}(Y_+,\alpha_+).\]
    \item[C.] (U-Map) If $(Y,\alpha)$ is a contact manifold with contact form $\alpha$ and $P$ is an abstract constraint, then
    \[\mathfrak{s}_{U_P(\sigma)}(Y,\alpha) \le \mathfrak{s}_{\sigma}(Y,\alpha).\]
    \item[D.] (Monotonicity) Let $f:Y \to [0,\infty)$ be a smooth non-negative function on $Y$. Then,
    \[\mathfrak{s}_\sigma(Y,\alpha) \le \mathfrak{s}_\sigma(Y,e^f \cdot \alpha).\]
    Moreover, $\mathfrak{s}_\sigma(Y,-)$ is continuous in the $C^0$-topology on contact forms.
    \vspace{3pt}
    \item[E.] (Reeb Orbits) For each class $\sigma \in CH(Y)$, there is a Reeb orbit tuple $\Gamma$ such that
    \[\mathfrak{s}_{\sigma}(Y,\alpha) = \mathcal{A}(\Gamma).\]
    Furthermore, if $\alpha$ is non-degenerate and $\sigma$ has grading $|\sigma|$, then
    \[|\Gamma| = |\sigma|.\]
\end{itemize}
\end{thm}

\begin{proof} We demonstrate each of these axioms individually. 

\vspace{3pt}

{\bf Axiom A.} This follows immediately from the definition and the fact that the contact homology groups of $CH(Y)$ with respect to $\alpha$ and $a \cdot \alpha$ are canonically identified with action filtrations differing by the scaling factor $a$.

\vspace{3pt}

{\bf Axiom B.} By the functoriality property of constrained cobordism maps stated in Lemma \ref{lem:constrained_cobordism_map_composition}, $\Phi_{X,P}$ can be written as the composition
\[CH(Y_+) \xrightarrow{\Phi_{X \setminus W}} CH(Y_-) \otimes CH(\partial W) \xrightarrow{\id \otimes \Phi_{W,P}} CH(Y_-) \otimes CH(\emptyset)[\on{codim}(P)] = CH(Y_-)[\on{codim}(P)].\]
Choose basis $x_i$ of $CH(Y_-)$ of pure action filtration, that is, $x_i$ is a linear combination of generators which have the same action. Then, for some set of elements $y_i \in CH(\partial W)$, all but finitely many of which vanish), we may write
\[
\Phi_{X \setminus W}(\sigma) = \sum_i x_i \otimes y_i \qquad \text{and}\qquad \Phi_{X,P}(\sigma) = \sum_i \Phi_{W,P}(y_i) \cdot x_i.
\]
Since $x_i$ is a basis of pure filtration, we know that
\begin{equation}\label{eqn:action_of_PhiXW}\mathcal{A}(\Phi_{X\setminus W}(\sigma)) = \mathcal{A}\big(\sum_i x_i \otimes y_i\big) = \on{max}\{\mathcal{A}(x_i \otimes y_i) \; : \; y_i \neq 0\}.\end{equation}
\begin{equation}\label{eqn:action_of_PhiXP} \mathcal{A}(\Phi_{X,P}(\sigma)) = \mathcal{A}\left(\sum_i \Phi_{W,P}(y_i) \cdot x_i\right) = \on{max}\{\mathcal{A}(x_i) \; : \; \Phi_{W,P}(y_i) \neq 0\}.\end{equation}
Let $m$ be the index such that $\mathcal{A}(x_m) = \mathcal{A}(\Phi_{X,P}(\sigma))$ and $\Phi_{W,P}(y_m) \neq 0$. By Definitions \ref{def:CH_spectral_invariants} and \ref{def:CH_capacity}, we have
\begin{equation} \label{eqn:constrained_monotonicity_1}
\s_{\Phi_{X,P}(\sigma)}(Y_-,\alpha_-) = \mathcal{A}(\Phi_{X,P}(\sigma)) = \mathcal{A}(x_m) \qquad\text{and}\qquad c_P(W) \le \mathcal{A}(y_m).
\end{equation}
On the other hand, by (\ref{eqn:action_of_PhiXW}) and the monotonicity of the action filtration under cobordism maps, we know that
\begin{equation} \label{eqn:constrained_monotonicity_2}
\mathcal{A}(x_m) + \mathcal{A}(y_m) = \mathcal{A}(x_m \otimes y_m) \le \mathcal{A}(\Phi_{X\setminus W}(\sigma)) \le \mathcal{A}(\sigma) = \mathfrak{s}_\sigma(Y_+,\alpha_+).\end{equation}
The constrained monotonicity property follows immediately from (\ref{eqn:constrained_monotonicity_1}) and (\ref{eqn:constrained_monotonicity_2}).
\vspace{3pt}

{\bf Axiom C.} Fix $\epsilon > 0$ and consider the cobordism $X:(Y,e^\epsilon \cdot \alpha) \to (Y,\alpha)$ given by $X = [0,\epsilon]_r \times Y$ equipped with the standard Liouville form $\lambda = e^r \alpha$. The $U$-map $U_P$ is the constrained map
\[\Phi_{X,P}:CH(Y) \to CH(Y).\]
By the usual monotonicity and scaling axioms, we therefore know that
\[\mathfrak{s}_{U_P(\sigma)}(Y,\alpha) \le \lim_{\epsilon \to 0} \mathfrak{s}_{\sigma}(Y,e^{\epsilon} \cdot \alpha) = \mathfrak{s}_{\sigma}(Y,\alpha).\]

\vspace{3pt}

{\bf Axiom D.} To prove monotonicity, assume that $f: Y \to [0, \infty)$ and choose $\epsilon > 0$. Consider the cobordism $X:(Y,e^f \cdot \alpha) \to (Y,e^{-\epsilon} \cdot \alpha)$ given by
\[X := \{(r,y) \; : \; 0 \le r \le f(y)\} \subset \R_r \times Y \quad \text{with Liouville form} \quad \lambda = e^r \lambda\] 
The cobordism map and scaling axioms imply that $\epsilon^{-1} \cdot \mathfrak{s}_\sigma(Y,\alpha) \le \mathfrak{s}_\sigma(Y,e^f \cdot \alpha)$. Thus, we take $\epsilon \to 0$ to acquire the monotonicity inequality. 

To deduce continuity, let $\alpha_i = f_i \cdot \alpha$ be a sequence of contact forms that $C^0$ converges to $\alpha$. Then, $f_i \xrightarrow{C^0} 1$ and so there exists a sequence of constants $C_i > 0$ such that
\[
C_i \to 1 \text{ as }i \to \infty \qquad \text{and}\qquad C_i > f_i > C_i^{-1}.
\]
By the cobordism map and scaling axioms, we see that
\[
C_i^{-1} \cdot \mathfrak{s}_\sigma(Y,\alpha) = \mathfrak{s}_\sigma(Y,C_i^{-1} \cdot \alpha) \le \mathfrak{s}_\sigma(Y,f_i \cdot \alpha) \le \mathfrak{s}_\sigma(Y,C_i \cdot \alpha) = C_i \cdot \mathfrak{s}_\sigma(Y,\alpha) .
\]
By taking the limit as $i \to \infty$, we see that $\mathfrak{s}_\sigma(Y,\alpha_i) \to \mathfrak{s}_\sigma(Y,\alpha)$.

\vspace{3pt}

{\bf Axiom E.} Assume that $(Y,\alpha)$ is non-degenerate and let $\sigma \in CH (Y)$. Then there exists a cycle $x$ representing $\sigma$ such that
\[x = \sum_\Gamma {c}_\Gamma \cdot x_\Gamma \qquad\text{with}\qquad \mathcal{A}(x) = \on{max}\{x_\Gamma \; : \; c_\Gamma \neq 0\} = \mathfrak{s}_\sigma(Y,\alpha).\]
If $\sigma$ has pure homological grading $|\sigma|$, then we can assume that $c_\Gamma = 0$ for $|\Gamma| \neq |\sigma|$. Let $\Gamma$ be the maximal action orbit tuple with $c_\Gamma \neq 0$, then
\[
\mathfrak{s}_\sigma(Y,\alpha) = \mathcal{A}(\Gamma) \qquad\text{and}\qquad |\Gamma| = |\sigma|
\]
as desired. 

If $\alpha$ is degenerate, we can take a sequence of non-degenerate contact forms $\alpha_i$ that $C^\infty$ converges to $\alpha$. Then, the corresponding orbit tuples $\Gamma_i$ have bounded total action and thus converge to an orbit tuple $\Gamma$ of $\alpha$ as $i \to \infty$. This convergence can be seen via arguments similar to those in the Proof of Theorem \ref{thm:periodicity_criterion} (\ref{pf:periodicity_criterion}). Since $\mathfrak{s}_\sigma$ is continuous in the $C^\infty$-topology on contact forms, this implies that
\[
\mathfrak{s}_\sigma(Y,\alpha) = \lim_{i \to \infty} \mathfrak{s}_\sigma(Y,\alpha_i) = \lim_{i \to \infty} \mathcal{A}(\Gamma_i) = \mathcal{A}(\Gamma). \qedhere
\]\end{proof}

\begin{thm} \label{thm:CHCapacityAxioms} The contact homology capacities $\mathfrak{c}_P$ satisfy the following properties.
\begin{itemize}
    \item[A.] (Conformality) If $(W,\lambda)$ is a Liouville domain and $a > 0$ is a constant, then
    \[\mathfrak{c}_P(W,a \cdot \lambda) = a \cdot \mathfrak{c}_P(W,\lambda). \] 
    \item[B.] (Monotonicity) If $(W,\lambda) \to (V,\mu)$ is a (weakly) exact embedding of Liouville domains, then
    \[\mathfrak{c}_P(W,\lambda) \le \mathfrak{c}_P(V,\mu).\]
    \item[C.] (Tensor Product) If $P$ and $Q$ are two abstract constraints, then
    \[\mathfrak{c}_P(W,\lambda) \le \mathfrak{c}_{P \otimes Q}(W,\lambda).\]
    \item[D.] (Reeb Orbits) If $P$ is an abstract constraint, then there is a tuple of Reeb orbits $\Gamma$ of $\partial W$ such that
    \[\mathfrak{c}_{P}(W,\lambda) = \mathcal{A}(\Gamma).\]
    Furthermore, if $\lambda|_{\partial W}$ is non-degenerate, then
    \[|\Gamma| = \on{codim}(P) \mod 2m \qquad\text{where}\qquad m := \on{min} |c_1(W) \cdot A|. \]
\end{itemize}
\end{thm}

\begin{proof} Axioms A, B, and D are proven by approaches that are essentially identical to the analogous properties (respectively A, D and E) in Theorem \ref{thm:rsftSpectralAxioms}. For Axiom C, we note that
\[\Phi_{W,P \otimes Q} = \Phi_{W,P} \circ U_Q.\]
Therefore, if $\Phi_{W,P \otimes Q}(\sigma) \neq 0$, we have $\Phi_{W,P}(U_Q(\sigma)) \neq 0$. We thus acquire the inequality
\[
c_P(W,\lambda) \le \on{min}\{\mathcal{A}(\tau) \; : \; \Phi_{W,P}(\tau) \neq 0\} \le \on{min}\{\mathcal{A}(U_Q(\sigma)) \; : \; \Phi_{W,P \otimes Q}(\sigma) \neq 0\}\]
\[ \le \on{min}\{\mathcal{A}(\sigma) \; : \; \Phi_{W,P \otimes Q}(\sigma) \neq 0\} = c_{P \otimes Q}(W,\lambda). \qedhere
\]
\end{proof}

\subsection{Spectral Gap} We are now ready to introduce the contact homology spectral gap.  

\begin{definition} Let $Y$ be a closed contact manifold with contact form $\alpha$ and let $\sigma \in CH(Y)$ be a contact homology class. The \emph{spectral gap} of $(Y,\alpha)$ in class $\sigma$ is defined to be
\begin{equation}
    \gap_\sigma (Y,\alpha) := \inf \Big\{\frac{\s_\sigma (Y,\alpha)-\s_{U_P(\sigma)}(Y,\alpha)}{\mathfrak{c}_P(B^{2n})} \; : \; P \text{ is an abstract constraint}\Big\}.
\end{equation}
The \emph{(total) spectral gap} of the contact manifold $(Y,\alpha)$ is given by
\begin{equation}
    \gap(Y,\alpha) := \inf_{\sigma\in CH(Y)}\ \gap_\sigma (Y,\alpha).
\end{equation}
\end{definition}

\begin{thm}[Theorem~\ref{thm:spectral_gap_to_closing_lemma_intro}] \label{thm:gap_0_closing_property} Let $(Y,\alpha)$ be a closed contact manifold with contact form such that
\[\gap(Y,\alpha)=0.\]
Then, $(Y,\alpha)$ satisfies the strong closing property, namely, for every non-zero $f:Y\rightarrow [0,\infty)$ there exists $t\in[0,1]$ such that $(1+tf)\alpha$ has a closed Reeb orbit passing through the support of $f$. Moreover, if $\gap_{\sigma}(Y,\alpha)=0$ for some $\sigma\in CH(Y)$, then the period of this orbit is bounded by $\s_\sigma(Y,\alpha)$. 
\end{thm}

\begin{proof}
Let $f:Y\rightarrow[0,\infty)$ and fix  $L\in[0,\infty]$. Assume, if possible, 
that for all $t\in[0,1]$, the contact form $(1+tf)\alpha$ does not have a periodic Reeb orbit of action up to $L$ through the support of $f$. In this case, the action spectrum of  $(1+tf)\alpha$ up to $L$ remains the same as $t$ varies. Note that the action spectrum of $\alpha$ is a measure zero set. As observed in \cite[Lemma $2.2$]{i2015}, this is a consequence of the fact that the critical values of the contact action functional are contained in the critical values of a smooth function on a finite dimensional manifold, which can be constructed by adapting the proof of \cite[Lemma $3.8$]{schwarz00} from the Hamiltonian setting. So, the continuity (in $t$) of the spectral invariants, stated in Theorem~\ref{thm:rsftSpectralAxioms}, guarantees that 
\begin{equation}\label{eq:spec_inv_stays_same}
    \s_\sigma (Y,(1+tf)\alpha) = \s_\sigma(Y,\alpha),  
\end{equation}
for all $t\in[0,1]$ and $\sigma\in CH(Y)$ such that $\s_\sigma(Y,\alpha)\leq T$. We will show that the cobordism property gives a positive lower bound for the spectral gap of such contact homology classes. 

Fix $\varepsilon>0$ small and let $X$ be the cobordism from $(Y,e^\varepsilon(1+f) \alpha)$ to $(Y,\alpha)$ given by 
\begin{equation*}
    X:= [0,\varepsilon]\times Y\ \cup \  \left\{(\varepsilon+ \log(1+t\cdot f(y)),y):  t\in[0,1] \right\}\ \subset \hat Y.
\end{equation*}
There exists a number $r=r(f,Y,\alpha)$ such that the ball $B^{2n}(r)$ of radius $r$ embeds into $X$. By the cobordism property of spectral invariants stated in Theorem~\ref{thm:rsftSpectralAxioms}, for any homology class $\sigma\in CH(Y)$ and an abstract constraint $P$ it holds that
\begin{equation*}
     \s_\sigma (Y, e^\varepsilon(1+f)\alpha)-\s_{\Phi_{X,P}(\sigma)}(Y,\alpha)\geq \mathfrak{c}_{P}(B^{2n}(r)).
\end{equation*}
By the conformality property of the capacities stated in Theorem~\ref{thm:CHCapacityAxioms}, we have $\mathfrak{c}_{P}(B^{2n}(r))= r^2\cdot\mathfrak{c}_{P}(B^{2n}) $. Rearranging the above inequality we obtain
\begin{equation}\label{eq:gap_lower_bound}
   r^2\leq  \frac{\s_\sigma (Y,e^\varepsilon(1+f)\alpha)-\s_{\Phi_{X,P}(\sigma)}(Y,\alpha)}{\mathfrak{c}_P(B^{2n})} .
\end{equation}
We will show that the latter lower bound contradicts the vanishing of the spectral gap. Let $X':= [0,\varepsilon/2]\times Y $ be a trivial cobordism and decompose $X$ as $X= X'\ \cup (X\setminus X')$. By the functoriality property of the constrained cobordism map we have
\begin{equation*}
    \Phi_{X,P}\cong  \Phi_{X',P}\circ \Phi_{X\setminus X'} =    U_{P}\circ \Phi_{X\setminus X'}.
\end{equation*}
Combining this with inequality (\ref{eq:gap_lower_bound}) and equation (\ref{eq:spec_inv_stays_same}), and using again the properties of spectral invariants, we conclude that for every homology class $\sigma\in CH(Y)$ such that $\s_\sigma(Y,\alpha)\leq T$,
\begin{align*}
    r^2&\leq  \frac{\s_\sigma (Y,e^\varepsilon(1+f)\alpha)-\s_{U_{P}\circ \Phi_{X\setminus X'}\sigma}(Y,\alpha)}{\mathfrak{c}_P(B^{2n})}= \frac{e^\varepsilon\s_\sigma (Y,(1+f)\alpha)-\s_{U_{P}\sigma}(Y,\alpha)}{\mathfrak{c}_P(B^{2n})}\\
    &\overset{(\ref{eq:spec_inv_stays_same})}{=}  \frac{e^\varepsilon\s_\sigma (Y,\alpha)-\s_{U_{P}\sigma}(Y,\alpha)}{\mathfrak{c}_P(B^{2n})}
    \leq \frac{\s_\sigma (Y,\alpha)-\s_{U_{P}\sigma}(Y,\alpha)}{\mathfrak{c}_P(B^{2n})} +\frac{r^2}{2},
\end{align*}
where the last inequality holds when we take $\varepsilon\leq \log \left(1+\frac{r^2}{2\s_{\sigma}(Y,\alpha)}\right)$.
Applying this estimate for  every abstract constraint $P$, we get a positive lower bound for the spectral gap,
\begin{equation}\label{eq:lower_bnd_gap_sigma}
    \frac{r^2}{2}\leq \inf_{P} \frac{\s_\sigma (Y,\alpha)-\s_{U_{P}}(Y,\alpha)}{\mathfrak{c}_P(B^{2n})}=\gap_\sigma(Y,\alpha).
\end{equation}
To prove the first assertion of the theorem, take $T=\infty$. Then the above lower bound for $\gap_\sigma(Y,\alpha)$ for all $\sigma$ implies that $\gap(Y,\alpha)$ is positive, in contradiction with the hypothesis. To prove the second part of the theorem, suppose $\gap_\sigma(Y,\alpha)=0$ for some $\sigma\in CH(Y)$ and take $T=\s_{\sigma}(Y,\alpha)$. Then (\ref{eq:lower_bnd_gap_sigma}) yields a contradiction.  
\end{proof}

\begin{remark}[Semi-continuity of the spectral gap]\label{rem:semicontinuity_of_gap}
    The spectral gap is upper semi-continuous, that is, for every sequence of contact forms $\alpha_i$ converging to $\alpha$ and for every class $\sigma\in CH(Y)$ we have
    \begin{align*}
        \limsup_{i\rightarrow\infty}\gap_\sigma(Y,\alpha_i) &\le \gap_\sigma(Y,\alpha),\\
        \limsup_{i\rightarrow\infty}\gap(Y,\alpha_i) &\le \gap(Y,\alpha).
    \end{align*}
    This is due to the fact that the spectral gaps are defined as an infimum over continuous functions.
\end{remark}

In fact, for a fixed class $\sigma$, the $\sigma$-spectral gap is continuous. More generally, under some conditions on the rate of convergence of $\alpha_i$ to $\alpha$, the spectral gap of $\alpha$ can be bounded from above by spectral gaps of $\alpha_i$.  In order to give a more precise statement of this fact we adapt the following notation.
\begin{notation}\label{not:order_on_ctct_forms}
    Let $\alpha$ and $\alpha'$ be two contact forms on $Y$. We write $\alpha\leq \alpha'$ if there exists an exact cobordism from $(Y,\alpha')$ to $(Y,\alpha)$.
\end{notation}
\begin{prop}\label{prop:limit_of_gap0}
    Let $(Y,\alpha)$ be a contact manifold and suppose there exist sequences $\{\alpha_i\}$ of contact forms, $\{\sigma_i\}$ of contact homology classes, and $\{\epsilon_i\}$ of positive numbers, such that:
    \begin{enumerate}
        \item $\alpha_i\leq \alpha\leq (1+\epsilon_i)\alpha_i$, and
        \item $\epsilon_i\cdot \s_{\sigma_i}(Y,\alpha_i) \xrightarrow[i\rightarrow \infty]{} 0$.
    \end{enumerate}
    Then, $\gap(Y,\alpha)\leq \liminf_{i\rightarrow\infty} \gap_{\sigma_i}(Y,\alpha_i)$.
    In particular, if $\gap_{\sigma_i}(Y,\alpha_i)\rightarrow 0$, then $\gap(Y,\alpha)=0$. 
\end{prop}
\begin{proof}
Fix $\delta>0$. For each $i$, let $P_i$ be an abstract constraint such that 
\begin{equation}
    \frac{\s_{\sigma_i}(Y,\alpha_i)-\s_{U_{P_i}(\sigma_i)}(Y,\alpha_i)}{\mathfrak{c}_{P_i}(B^{2n})}<\gap_{\sigma_i}(Y,\alpha_i)+\delta.
\end{equation}
By our assumption that $\alpha_i\leq \alpha\leq (1+\epsilon_i)\alpha_i$ we have the following commutative diagram of filtered homologies
\begin{equation}
    \begin{array}{ccc}
        CH(Y,(1+\epsilon_i)\alpha_i) & \xrightarrow{\ U_{P_i}^{(1+\epsilon_i)\alpha_i}\ } &  CH(Y,(1+\epsilon_i)\alpha_i)\\
        & & \\
         \Phi_1\Big\downarrow\hspace{0.4cm} & & \Big\downarrow \Phi_1\\
         & & \\
         CH(Y,\alpha) & \xrightarrow{\quad U_{P_i}^\alpha\quad} &  CH(Y,\alpha)\\
          & & \\
        \Phi_2\Big\downarrow\hspace{0.4cm} & & \Big\downarrow \Phi_2\\
         & & \\
        CH(Y,\alpha_i) & \xrightarrow{\quad U_{P_i}^{\alpha_i}\quad} &  CH(Y,\alpha_i) . 
    \end{array}
\end{equation}
Here, in order to distinguish between the $U$-maps on contact homologies filtered by the contact forms $\alpha$ and $\alpha_i$, we adapt the notations $U_{P_i}^{\alpha}$ and $U_{P_1}^{\alpha_i}$ respectively.
By definition, for $\sigma_i\in CH(Y,(1+\epsilon_i)\alpha_i)$ we have
\begin{align*}
    \gap(Y,\alpha) &\leq \frac{\s_{\Phi_1(\sigma_i)}(Y,\alpha)-\s_{U_{P_i}^\alpha \Phi_1(\sigma_i)}(Y,\alpha)}{\mathfrak{c}_{P_i}(B^{2n})}
    =  \frac{\s_{\Phi_1(\sigma_i)}(Y,\alpha)-\s_{ \Phi_1 U_{P_i}^{(1+\epsilon_i)\alpha_i}(\sigma_i)}(Y,\alpha)}{\mathfrak{c}_{P_i}(B^{2n})}\\
    &\leq \frac{\s_{\sigma_i}(Y,(1+\epsilon_i)\alpha_i)-\s_{ \Phi_2\Phi_1 U_{P_i}^{(1+\epsilon_i)\alpha_i}\sigma_i}(Y,\alpha_i)}{\mathfrak{c}_{P_i}(B^{2n})},
\end{align*}
where the last inequality follows from the cobordism property of spectral invariants, stated in Theorem~\ref{thm:rsftSpectralAxioms}. By Theorem~\ref{thm:contact_homology_well_defined}, the composition $\Phi_2\Phi_1:CH(Y,(1+\epsilon_i)\alpha_i)\rightarrow CH(Y,\alpha_i)$ of maps between the contact homologies is the indentity map with filtration rescaled by $1/(1+\epsilon_i)$. Using the conformality property of spectral invariants (Theorem~\ref{thm:rsftSpectralAxioms}) we obtain
\begin{align*}
    \gap(Y,\alpha) &\leq \frac{\s_{\sigma_i}(Y,(1+\epsilon_i)\alpha_i)-\s_{U_{P_i}^{\alpha_i}\sigma_i}(Y,\alpha_i)}{\mathfrak{c}_{P_i}(B^{2n})}
    = \frac{(1+\epsilon_i)\s_{\sigma_i}(Y,\alpha_i)-\s_{ U_{P_i}^{\alpha_i}\sigma_i}(Y,\alpha_i)}{\mathfrak{c}_{P_i}(B^{2n})}\\
    &=\left(\frac{\s_{\sigma_i}(Y,\alpha_i)-\s_{ U_{P_i}^{\alpha_i}\sigma_i}(Y,\alpha_i)}{\mathfrak{c}_{P_i}(B^{2n})}\right) + \frac{\epsilon_i\cdot \s_{\sigma_i}(Y,\alpha_i)}{\mathfrak{c}_{P_i}(B^{2n})}\\
    &\leq \gap_{\sigma_i}(Y,\alpha_i) + \delta + c \epsilon_i\cdot\s_{\sigma_i}(Y,\alpha_i),
\end{align*}
where the last inequality follows from our choice of the abstract constraint $P_i$, and the fact that the capacities $\mathfrak{c}_{P_i}(B^{2n})$ have a positive lower bound which is uniform in $i$. Since our $\delta$ can be taken to be arbitrarily small, we conclude that if the product $\epsilon_i\cdot\s_{\sigma_i}(Y,\alpha_i)$ converges to zero, then $\gap(Y,\alpha)\leq \liminf_{i\rightarrow\infty}\gap_{\sigma_i}(Y,\alpha_i)$.
\end{proof}

\subsubsection{Detecting periodicity via spectral gaps.}
Theorem~\ref{thm:periodicity_criterion} from the introduction states that if there exists a homology class $\sigma\in CH(Y,\xi)$ such that $\gap_\sigma(Y,\alpha)=0$, then the Reeb flow of $\alpha$ is periodic. {The following proof is similar to an argument from ECH that was pointed out to us by Oliver Edtmair.}

\begin{proof}[Proof of Theorem~\ref{thm:periodicity_criterion}]\label{pf:periodicity_criterion}
A classical theorem of Wadsley \cite{wadsley1975geodesic} implies that if all Reeb orbits are closed then the flow is periodic. Therefore, it is sufficient to show that there is a periodic orbit passing through every point of $Y$.

Fix a point $y\in Y$, let $\{V_n\}_{n=1}^\infty$ be open sets in $Y$ such that $\cap_{n} V_n=\{y\}$. Let $f_n:Y\rightarrow [0,\infty)$ be a smooth non-zero function supported in $V_n$, such that $\|f_n\|_{C^3}\leq \frac{1}{n}$. Recall our assumption that $\gap_\sigma(Y,\alpha)=0$ and denote $T:=\s_\sigma(Y,\alpha)$. 
By Theorem~\ref{thm:gap_0_closing_property}, there exists $t_n\in [0,1]$ such that the contact form $\alpha_n:=(1+t_nf_n)\alpha$ has a periodic orbit $\gamma_n$ of period $T_n\leq T$ passing through the support of $f_n$. By extracting a subsequence we may assume that $T_n$ converge to some $T_\star\leq T$.
We think of $\gamma_n$ as a map $\R\rightarrow Y$, and denote $\underline{\gamma}_n(-):=\gamma_n(T_n\cdot - ):S^1\cong\R/\Z\rightarrow Y$. Denote by $R_{\alpha_n}$ and $R_\alpha$ the Reeb vector fields of $\alpha_n$ and $\alpha$, respectively. 
The sequence $\underline{\gamma}_n$ is equicontinuous since the derivatives $\dot{\underline{\gamma}}_n = T_n\cdot R_{\alpha_n}$ are uniformly bounded. Therefore, 
we can apply  Arzel\`a-Ascoli theorem and conclude that there is a subsequence $\{\underline{\gamma}_{n_k}\}_k$ that converges to a limit $\underline{\gamma}:S^1\rightarrow Y$. Clearly, $\underline\gamma$ passes through $y$ since $\underline{\gamma}_n$ passes through $V_n$. Let us show that $\underline{\gamma}$ is differentiable and that its derivative is $T_\star\cdot R_\alpha$. Fix $t\in S^1$ and consider a chart in $Y$ around $\underline{\gamma}(t)$. For large enough $n$, $\underline{\gamma}_n(t)$ lies in this chart and satisfies
\[
\underline{\gamma}_n(t+\varepsilon)-\underline{\gamma}_n(t) = \int_0^\varepsilon T_n\cdot R_{\alpha_n} \circ \underline{\gamma}_n(t+\tau) d\tau,
\]
for small enough $\varepsilon$.
Taking the limit when $n\rightarrow \infty$ over this equation, we obtain
\[
\underline{\gamma}(t+\varepsilon)-\underline{\gamma}(t) = \int_0^\varepsilon T_\star\cdot R_{\alpha} \circ \underline{\gamma}(t+\tau) d\tau
\]
for every $\varepsilon>0$ small enough. Therefore, $\underline\gamma$ is indeed differentiable, its derivative is $T_\star\cdot R_\alpha$, and its reparamterization $\gamma(t):=\underline{\gamma}(t/T_\star)$ is a periodic orbit of $R_\alpha$ that passes through $y$, as required.
\end{proof}

\section{Analysis of the constrained moduli spaces}\label{sec:moduli_space}
In this section, we analyze the compactified moduli spaces of holomorphic cylinders between select orbits of a Morse-Bott contact form. We prove that, under some conditions, these moduli spaces are cut-out transversely and consist of a single point.
Before formally stating the main result for this section, let us fix the setting and present the relevant definitions and notations.
\begin{setup}\label{set:intersection_theory}
Throughout this section we fix the following structural assumptions. 
\begin{enumerate}
    \item A contact manifold $(Y,\xi)$ of dimension $2n-1$, for $n\geq 3$.
    \item A contact form $\alpha$ on $Y$ such that the flow $\varphi_t$ of the corresponding Reeb vector field $R$ is periodic. 
    \item  An almost complex structure $J$ on $\xi$ and a Riemannian metric $ g_J(-, -) = \alpha \otimes \alpha + d\alpha(-, J-)$  on $Y$.
    \item A Morse-Bott function $f:Y\rightarrow \R$  that is $\varphi_t$-invariant and whose critical manifolds are 1-dimensional. Note that the $\varphi_t$-invariance of $f$ implies that the critical manifolds are disjoint unions of periodic Reeb orbits.
    For such $f$, we let:
    \begin{itemize}
        \item $X_f$ be the vector field defined by $\alpha(X_f) \equiv 0$ and $d\alpha( -,X_f) = df(-)$; and
        \item $\nabla_J f$ be the vector field defined by $g_J(-,\nabla_J f)= df(-)$. Note that $J\nabla_J f= -X_f$.
    \end{itemize}
    \item \label{itm:f_perturbation} Given $f$ as above and $\varepsilon>0$ we define perturbations of $\alpha$, $R$ and $J$ as follows.
    \begin{itemize}
    \item The perturbation of $\alpha$ is $\alpha_{f} := e^{\varepsilon f}\alpha$.
    \item The Reeb vector field of $\alpha_{f}$ is $R_{f} = e^{-\varepsilon f}(R - \varepsilon X_f).$
    \item Let $J_f$ be the $\R$-invariant almost complex structure on  $\hat{Y} = \mathbb{R}_r \times Y$ satisfying:
    \begin{itemize}
        \item $J_f(\partial_r) = R_f$,
        \item $J_f$ preserves the bundle $\xi$, and
        \item the restriction of $J_f$ to $\xi$ is equal to $J$.
    \end{itemize}
    \end{itemize}

    \item Let $\gamma_+$ and $\gamma_-$ be circles of global maxima and minima of $f$, respectively. Considering $\gamma_\pm$ as (not necessarily simple) Reeb orbits of  $\alpha$, assume that their periods are  the same and coincide\footnote{For the arguments in Sections~\ref{subsec:lifting_flow_lines}-\ref{subsec:transversality}, it is enough to assume that the periods of $\gamma_\pm$ are the same and \emph{divide} the minimal period of the flow $\varphi_t$. The assumption that the periods $\gamma_\pm$ coincide with the minimal period of the flow is used only in Section~\ref{subsec:Umap_moduli_space_vs_point_moduli_space} for regularity purposes. See Remark~\ref{rem:sequences_of_U_rational_ellipsoids}} with the minimal period of the flow $\varphi_t$.
    \item Let $z\in Y$ be a point in the intersection of the unstable manifold of $\gamma_+$ and the stable manifold of $\gamma_-$. 
    \item \label{itm:contact_flag} Consider a sequence of nested manifolds 
    $
    Y_3\subset \cdots \subset Y_{2n-1} = Y,
    $
    such that for each $j$:
    \begin{enumerate}[label=(\alph*)]
        \item \label{itm:zero_H2_assumption} $\dim{ Y_{2j-1} }=2j-1$ and $H_2(Y_{2j-1})=0$;
        \item $Y_{2j-1}$ is invariant under the Reeb flow $\varphi_t$;
        \item $(Y_{2j-1}, \xi\cap{TY_{2j-1}})$ is a contact manifold and $\alpha|_{TY_{2j-1}}$ is a contact form for it. Moreover, $\xi\cap {TY_{2j-1}}$ is $J$-invariant;
        \item $\nabla_J f$ is tangent to $Y_{2j-1}$ and $f|_{Y_{2j-1}}$ is Morse-Bott. In particular, $Y_{2j-1}$ is invariant under the gradient flow of $f$;
        \item \label{itm:Hessian_positive} Along any critical circle $\gamma$ of $f$, other than $\gamma_+$ which lies in $Y_{2n-3}$, the restriction of the Hessian of $f$ to the symplectic orthogonal of  $\xi\cap {TY_{2j-1}}$ in  $\xi\cap {TY_{2j+1}}$ is positive definite\footnote{In particular, this means that $\gamma_+$ is the only maximum of $f$ in $Y_{2n-3}$. This assumption will be used for computing certain Conley-Zehnder indices of an $f$-perturbation of the Reeb flow.};
        \item $z\in Y_3\subset \cdots \subset Y_{2n-1}$.
    \end{enumerate}
    We call such a sequence $\{Y_{2j-1}\}_{j=2}^n$ a \emph{contact flag}. This is illustrated in Figure~\ref{fig:contact_flag}.
\end{enumerate}
    
\end{setup}

\begin{figure}[h]
\centering
\includegraphics[width=.6\textwidth]{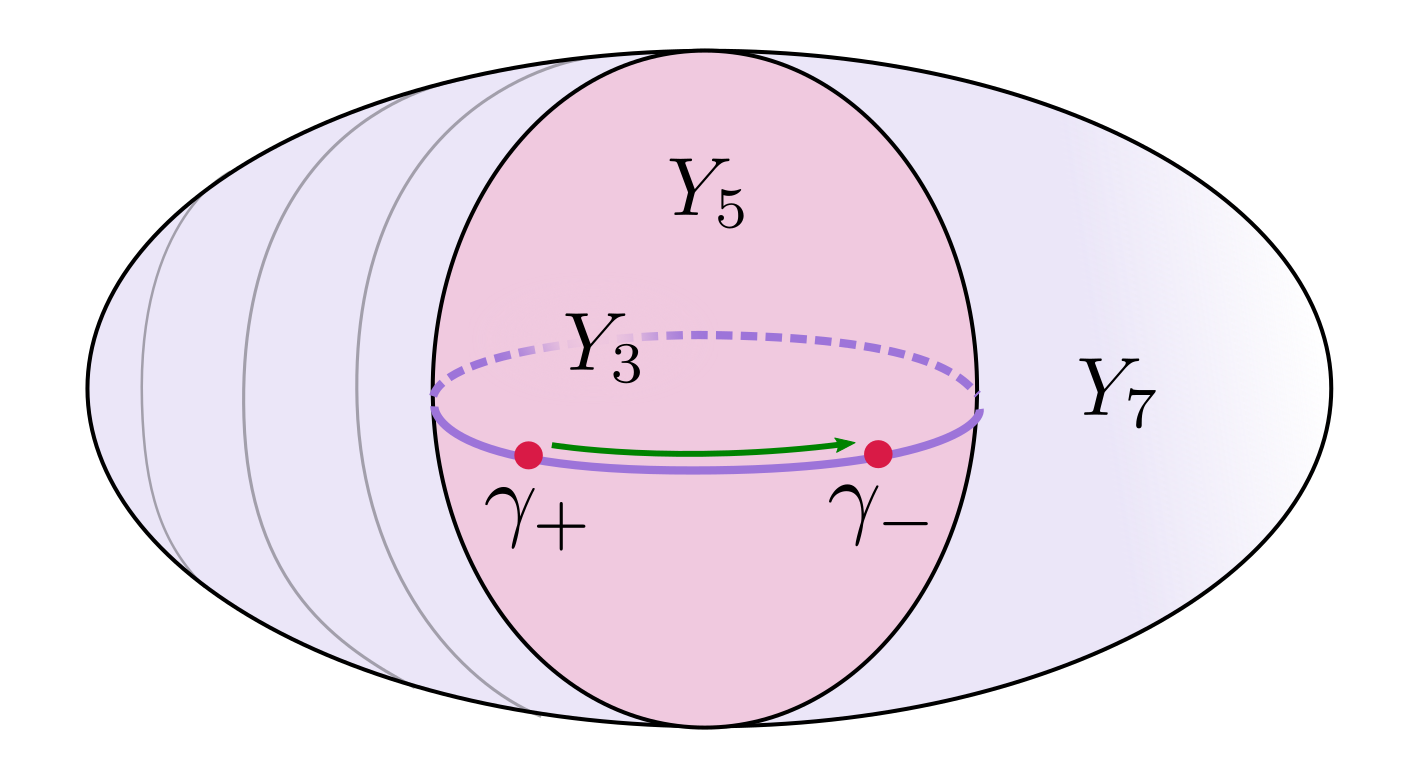}\\
\label{fig:contact_flag}
\caption{An illustration of a contact flag, as defined in Setup~\ref{set:intersection_theory}. The lavender sphere $Y_7$, the pink disk $Y_5$ and the purple circle $Y_3$ represent consecutive nested sub-manifolds in the flag. The orbits $\gamma_+$ and $\gamma_-$ lie in $Y_3$, and are connected by a Morse flow line in green.}
\end{figure}

\begin{example}
The example of interest of this setting is when $Y=\partial E$ is the boundary of an ellipsoid with rationally dependent entries. The contact flag is given by intersecting $\partial E$ with complex linear subspaces: $Y_{2j-1} = \partial E\cap \left(\C^{j-1}\times \{0\}^{n-j}\times \C\right)$. In Section~\ref{subsec:rational_ellipsoids} we study this example in detail and explain how it fits into Setup~\ref{set:intersection_theory}.
\end{example}
\begin{remark}\label{rem:non-toric}
    Assumption (\ref{itm:contact_flag}).\ref{itm:zero_H2_assumption} that $H_2(Y_{2j-1})=0$ has two purposes:
    \begin{itemize}
        \item To ensure that the relative CZ index is well defined over $\Z$,
        \item To guarantee that the homological intersection of any 2-torus in $Y_{2j+1}$ with the hypersurface $Y_{2j-1}$ is zero (see Lemma~\ref{lem:vanishing_relative_intersection} and Figure~\ref{fig:intersection_zero_proof}).
    \end{itemize}
    As a consequence, assumption (\ref{itm:contact_flag}).\ref{itm:zero_H2_assumption} can be replaced with any weaker conditions that ensure the two above properties. This is done in Section~\ref{sec:non_toric_example}.
\end{remark}

The main purpose of this section is to prove the following claim.
\begin{prop}\label{prop:moduli_space_count} Consider Setup~\ref{set:intersection_theory}. The compactified moduli space\footnote{See section~\ref{subsubsec:buildings_in_contact_homology} for the definition of this moduli space.}
\[
\moduli
\]
of $J_f$-holomorphic cylinders in $\hat Y$ between $\gamma_-$ and $\gamma_+$ containing $(0,z)$ consists of a single point which is transversely cut-out. 
\end{prop}
\begin{remark}
The single point in the moduli space $\moduli$ is a certain lift of the gradient flow line of $f$ from $\gamma_-$ to $\gamma_+$ that passes through $z$. This lift is described in section~\ref{subsec:lifting_flow_lines} below.
\end{remark}

\begin{remark}\label{rem:non-deg_n_CZ}
\begin{itemize}
    \item For any periodic flow, the fixed point set at any time $T$ is a submanifold satisfying the Morse--Bott condition. This is due to the fact that $N$ any periodic diffeomorphism $\psi$ is an isometry for the metric $\hat g:= \frac{1}{N}\sum_{j=1}^N \psi^*g$, where $g$ is any Riemannian metric. 
    \item 
The contact form $\alpha_f:=e^{\varepsilon f}\alpha$ is non-degenerate in a finite action window when $\varepsilon$ is small enough. That is, fixing $T_\star>0$, there exists $\varepsilon>0$ such that every periodic orbit of $\alpha_f$ with action less than $T_\star$ is non-degenerate. These periodic orbits lie over the critical circles of $f$, and thus correspond to pairs $(S, p)$ of a Morse--Bott family $S$ of orbits of $\alpha$ and a circle $p$ of $f$-critical points. The Conley-Zehnder index of such an orbit can be calculated from the Robbin-Salamon index of the family $S$, its dimension, and a Morse-type index of $p$, as shown in the following lemma.
\end{itemize}
\end{remark}

\begin{lemma}\label{lem:CZ_via_RS_n_Morse}
    Fix a trivialization $\tau$ of the contact structure on a neighborhood of the perturbed orbit $\gamma = (S,p)$ corresponding to a pair of a Morse--Bott family $S$ of orbits of $\alpha$ and a circle $p$ of $f$-critical points. The Conley-Zehnder index of $\gamma = (S,p)$ with respect to $\tau$ is given by 
\begin{equation}\label{eq:CZ_ind}
    \indCZ_\tau(S,p) = \indRS_\tau(S) + \frac{1}{2}\dim{S} -\ind_{Morse}^S(f;p),
\end{equation}
where $\ind_{Morse}^S(f;p)$ is the number of negative eigenvalues of the restriction of the Hessian of $f$ to the tangent space of the image of $S$ in $Y$. 
\end{lemma}
\begin{proof}
To see this, write the linearized Reeb flow of $\alpha_f$ as a composition
\begin{equation*}
    d\varphi_t^f = (d\varphi_t^f \circ d \varphi_t^{-1})\circ d\varphi_t.
\end{equation*}
As explained in \cite{gutt2014generalized}, the composition is homotopic to the concatenation of paths. By the additivity and homotopy invariance of the Conley-Zehnder index (Proposition~\ref{prop:CZ}), the index of the orbit $(S,p)$ with respect to $\alpha_f$ is equal to the sum of the Robbin-Salomon index of the family $S$ with respect to $\alpha$ and the CZ-index of the path $d\phi_t := (d\varphi_t^{-1}\circ d\varphi_t^f)\circ d\varphi_T $, where $T$ is the period of $S$. Let us compute the latter. First, notice that for every $t\leq T$, the image $\phi_t(\gamma(0))$ lies in a small neighborhood of the point $\varphi_T(\gamma(0))$, since the flows $\varphi_t$ and $\varphi_t^f$ differ by a small time reparametrization on critical circles of $f$. Identifying a neighborhood of $\varphi_T\gamma(0)$ with its Darboux chart, the path $d\phi_t$ solves the ODE 
$d \dot{\phi}_t = J_{0} \on{Hess}(\varepsilon f)_{\gamma(0)}|_\xi d\phi_t$.  When $\varepsilon$ is sufficiently small, the path $\Phi(t)$ crosses the Maslov cycle only at the origin and the crossing form is $\Gamma(\gamma(0)):=-\on{Hess}(\varepsilon f)_{\gamma(0)}|_{\ker(d\phi_0-\id)}$. The kernel of $d\phi_0 -\id= d\varphi_T-\id$ coincides with the tangent space to the image of $S$ in $Y$ by our assumption that $S$ is a Morse--Bott family.  Since the Hessian of $f$ is degenerate only in the Reeb direction, the signature of this crossing form is by definition the number of positive eigenvalues minus the number of negative eigenvalues, and hence coincides $\dim{S} -2\ind_{Morse}^S(f;p)$. This shows that (\ref{eq:CZ_ind}) holds. 
\end{proof}
The next lemma shows that the parity of the CH-grading coincides with the parity of the Morse indices of the perturbing function $f$. We will use it to conclude that the differential in CH vanishes.
\begin{lemma}\label{lem:RS_parity}
    The contact homology grading of a pair $(S, p)$ of a Morse--Bott family $S$ of orbits of $\alpha$ and a circle $p$ of $f$-critical points satisfies:
    \begin{equation}
        |(S,p)|:= n-3 + CZ(S,p) = \ind_{Morse}^S(f;p)\mod 2.
    \end{equation}
\end{lemma}
\begin{proof}
We start by showing that the dimension of the family $S$ is even.  Indeed, let $T$ be the period of the family $S$ and denote by $N_T\subset Y$ the fixed-point set of $\varphi_T$, or equivalently, the submanifold composed of the orbits in $S$. Its tangent space is
$$
T N_T = \left<R\right>\oplus \ker(d\varphi_T-\id).
$$
Therefore, $\dim{S} =\dim{N_T}-1 = \dim{ \ker(d\phi_T-\id)}$, which is even, since it is the 1-eigenspace of a linear symplectic map. 

Set $d:=\frac{1}{2}\dim S\in \N$. We will show that 
\begin{equation}\label{eq:RS_mod_2}
    RS(S) = n-d-1 \mod 2.
\end{equation} 
Together with the formula for  $ CZ(S,p) $ given in Lemma~\ref{lem:CZ_via_RS_n_Morse}, this yields the required result. 
Our  first step towards (\ref{eq:RS_mod_2}) is to notice that the parity of the Robbin-Salamon index of a path of matrices depends only on the ends and not on the path it self. This can be seen, for example, from the definition of the Robbin-Salamon index given at (\ref{eq:RS_def}). Given a path with non-degenerate crossings, the index is a sum of the signatures of the crossing forms. These forms are non-degenerate, and are defined over even dimensional spaces, $\ker(d\varphi_t-\id)$. Therefore, the signatures are all even. Since the signatures of the ends of the path are are multiplied by $1/2$, they determine the parity of the index. 

So now we focus on $d\varphi_T$ and identify it with a symplectic matrix. We can write it as a $\id_{\xi\cap TN_T} \oplus \Psi$. By the additivity of the RS index, The Morse--Bott condition implies that $\Psi$ does not have $1$ as an eigenvalue. Moreover, there exists $N$ such that $\Psi^N = \id$, since the Reeb flow on $Y$ is periodic. Let us show that, up to conjugation, the matrix $\Psi$ decomposes to a direct sum of elements of $U(1)$. Fix a compatible metric $g$, then $\hat g:=\frac{1}{N}\sum_{k=1}^N (\Psi^k)^*g$ is a $\Psi$-invariant compatible metric. We conclude that up to a change of basis, $\Psi$ is orthogonal, and hence unitary. In particular, $\Psi$ is diagonalizable, and after another change of basis can be written as a direct sum of two-dimensional rotations 
$$\Psi = \oplus_{j=1}^{n-1-d} e^{i\theta_j}$$,
with $\theta_j\neq 0 \mod 2\pi$ due to the Morse--Bott condition.
The RS index of the path $\Psi(t) = \oplus_{j=1}^{n-1-d} e^{i\theta_jt}$ is $\sum_{j=1}^{n-1-d} 1 = n-1-d$. Therefore,
\begin{align*}
RS(S) &= RS(\Psi(t))\mod 2\\
&= n-1-d \mod 2.     
\end{align*}
This shows that (\ref{eq:RS_mod_2}) holds and concludes the proof.
\end{proof}

The rest of this section is dedicated to proving the above proposition. In subsection~\ref{subsec:lifting_flow_lines} we describe an element of $\moduli$. In subsection~\ref{subsec:flow_lines_unique} we prove that there are no other elements in this moduli space. In subsection~\ref{subsec:transversality} we show that this moduli space is cut-out transversely. 

\subsection{Lifting flow lines to holomorphic cylinders.}\label{subsec:lifting_flow_lines}
In this section we lift a gradient flow line of the Morse-Bott function $f$ in $Y$ to a $J_f$-holomorphic curve in the symplectization $\hat Y$. 
{Let $\eta: \R \to Y$ be a gradient flow line of $f$, that is, $\dot{\eta}(s) = \nabla_J f(\eta(s))$, such that $\lim_{s \rightarrow \pm \infty} \eta(s)$ is contained in $\gamma_\pm$, respectively. Further, assume that $\eta(0) = z$.  Such a gradient flow line exists by our choice of the point $z$ stated in Setup~\ref{set:intersection_theory} and is clearly unique.}
Denoting by $T>0$ the common period of  $\gamma_\pm$ with respect to the contact form $\alpha$, we can define a $J_f$ holomorphic map $u_\eta:\R\times S^1\rightarrow \hat Y$ by 
\begin{equation}\label{eq:lifted_flow_line}
    u_\eta(s,t):=\left(a(s), \varphi_{T\cdot t}(\eta(\varepsilon T s))\right),
\end{equation}
where $a(s)$ is defined by the ODE
\begin{equation*}
    \dot a(s) = T\cdot e^{\varepsilon f(\eta(\varepsilon T s))}, \qquad a(0)=0.
\end{equation*} 
Note that $u_\eta$ is indeed a $J_f$-holomorphic curve that limits to $\gamma_\pm$ at the ends, and contains $(0,z)$ in its image, since $u_\eta(0,0)=(0, \eta(0))$.
Therefore, $u_\eta$ lies in the moduli space $\moduli$. 

\subsection{Uniqueness}\label{subsec:flow_lines_unique}
In this section we show that the moduli space $\moduli$ has no elements other than the lift $u_\eta$ constructed above. 

\begin{prop}\label{prop:uniqueness}
Consider Setup~\ref{set:intersection_theory}. The only point in the moduli space  $\moduli$  is the curve $u_\eta$ constructed in Section~\ref{subsec:lifting_flow_lines}.
\end{prop}
The proof of Proposition~\ref{prop:uniqueness} uses the intersection theories in \cite{siefring3d, moreno2019holomorphic}. We will show inductively that any element of $\moduli$ is contained in the symplectizations of all of the submanifolds $Y_3\subset \cdots\subset Y_{2n-1}=Y$. Then it will follow from a result in \cite{siefring3d} that this element is in fact contained in the image of $u_\eta$. 
The structure of this section is as follows. First, we show that the buildings in $\moduli$ consist only of cylinders, then we give an overview of the holomorphic intersection theory of Moreno-Siefring, and finally we explain how to use it to show that $\moduli$ consists of only one element.

\subsubsection{Ruling out non-cylindrical buildings}
Our first step towards proving that $u_\eta$ is the unique element in the moduli space $\moduli$ is to show that all buildings in $\moduli$ consist only of cylinders.   
\begin{lemma}\label{lem:no_buildings}
For $\varepsilon>0$ sufficiently small, each building in $\moduli$ consists solely of $J_f$-holomorphic cylinders.
\end{lemma}
\begin{proof}
    Fix a building $\overline{u}\in \overline{\cM}(\gamma_+,\gamma_-;J_f)$ with $k\geq 1$ levels. Let  $\Gamma^1, \Gamma^2, \dots, \Gamma^{k+1}$ be the associated limits and notice that $\Gamma^1 = (\gamma_-)$ and $\Gamma^{k+1} = (\gamma_+)$.  
        
    We start by finding a sequence of orbits $\{\gamma^j\in \Gamma^j\}_{j=1}^{k+1}$ of non-decreasing actions:
    $$\int_{\gamma^1} \alpha_f \leq \int_{\gamma^2} \alpha_f \leq \ldots \leq \int_{\gamma^{k+1}} \alpha_f.$$
    This sequence is constructed inductively as follows. {Pick $\gamma^1 = \gamma_-$.} Given $\gamma^j$, there exists a unique connected $J_f$-holomorphic curve $u^j$ in $\overline{u}$ such that $\gamma^j$ is one of its negative ends. Recall that every connected component of a level in $\overline{u}$ has a single positive end, as explained in section~\ref{subsubsec:buildings_in_contact_homology}. We take $\gamma^{j+1}$ to be the unique positive end of $u^j$. See Figure~\ref{fig:orbit_subsequence} for an illustration.
    \begin{figure}[h]
    \centering
    \includegraphics[width=.8\textwidth]{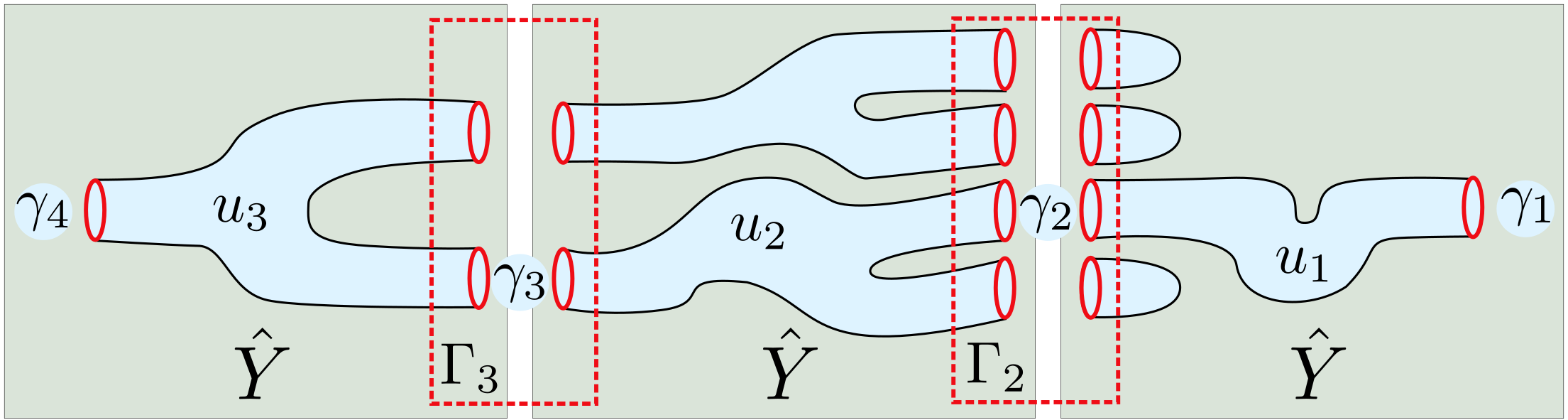}
    \caption{An illustration of a holomorphic building $\overline{u}$, and the sequence of orbits $\{\gamma^j\}$.}
    \label{fig:orbit_subsequence}
    \end{figure}
    
    We will now show that the connected components $u^j$ must be all cylinders. This will conclude the proof, since the positive end of $u^{k+1}$ is $\gamma_+$. 
    Let $\gamma^j_1 := \gamma^j, \gamma^j_2, \dots, \gamma^j_{N_j}$ be the negative ends of $u^j$. We need to show that $N_j=1$ for all $j$. 
    For any given punctured $J_f$-holomorphic curve in $\hat{Y}$, the total integral of $\alpha_f$ over the positive limits is greater than or equal to the total integral of $\alpha_f$ over the negative limits. Applying this to $u^j$, we find that 
    \begin{equation}\label{eq:action_dec_building}
        \int_{\gamma^{j+1}} \alpha_f \geq \sum_{i=1}^{N_j}\int_{\gamma^{j}_i} \alpha_f \geq \int_{\gamma^{j}} \alpha_f + (N_j -1)\cdot T_{min},
    \end{equation}
    where $T_{min}$ is the minimal period of a periodic orbit of $\alpha_f$. In particular, the $\alpha_f$-integrals are non-decreasing along the sequence $\{\gamma^j\}$.
    The $\alpha_f$ integrals of the first and last orbits in this sequence are
    $$\int_{\gamma_{\pm}} \alpha_f = e^{\varepsilon f(\gamma_{\pm})} T$$
    where $T > 0$ is the common period of $\gamma_\pm$ under the Reeb flow of $\alpha$. Combining this with (\ref{eq:action_dec_building}) we find that for all $j$,
    \begin{align*}
        e^{\varepsilon f(\gamma_+)}T &= \int_{\gamma_+}\alpha_f\geq \int_{\gamma^{j+1}} \alpha_f \\
        &\overset{(\ref{eq:action_dec_building})}{\geq} \int_{\gamma^{j}} \alpha_f + (N_j -1)\cdot T_{min}\\
        &{\geq} \int_{\gamma_-} \alpha_f + (N_j -1)\cdot T_{min}\\
        &= e^{\varepsilon f(\gamma_-)} T + (N_j -1)\cdot T_{min}.
    \end{align*}
    Rearranging the above, we obtain 
    \begin{equation}\label{eq:action_bound_building}
        (N_j -1)\cdot T_{min}\leq (e^{\varepsilon f(\gamma_+)} - e^{\varepsilon f(\gamma_-)})T.
    \end{equation} 
    The minimal action $T_{min}$ of a Reeb orbit of $\alpha_f = e^{\varepsilon f}\alpha$ is uniformly bounded {(in $\varepsilon$)} below by a constant depending only on $\alpha$, as long as $\varepsilon$ is taken to be smaller than some constant depending only on $f$. This implies that the inequality (\ref{eq:action_bound_building}) can only hold if $N_j-1=0$ for all $j \in \{1, \ldots, k + 1\}$. We thus conclude that $u^j$ are all cylinders, and therefore $\overline{u}$ is composed solely of cylinders as well. 
\end{proof}

\subsubsection{An overview of holomorphic intersection theory}\label{subsec:intersection_theory_overview}
Our main tool in proving the uniqueness result stated in Proposition~\ref{prop:uniqueness} is the holomorphic intersection theory developed in \cite{siefring3d,moreno2019holomorphic}. We give here an overview of this theory, following \cite{moreno2019holomorphic} and adapted to our case.

Let $(Y, \xi, \alpha)$ be a closed contact manifold with a contact form and denote the Reeb vector field by $R$. Let $J$ be an almost complex structure on $\hat Y$ and consider a  codimension-$2$ submanifold $Z$ of $\hat Y$ such that:
\begin{itemize}
    \item there exist closed codim-2 submanifolds $Z_\pm\subset Y$ such that $Z$ is \emph{asymptotically cylindrical} over $Z_\pm$ (see Definition~\ref{def:asymptotically_cylindrical} below),
    \item the sub-bundles $\xi_{Z_\pm}:=\xi\cap TZ_\pm$ are $J$-invariant,
    \item $(Z_\pm,\xi_{Z_\pm})$ are contact manifolds and $\alpha|_{TZ_\pm}$ are contact forms on them, and
    \item $Z$ is invariant under the flow of $R$.
\end{itemize}
The contact structure $\xi$ splits along $Z_\pm$ into a pair of symplectic vector bundles
$$(\xi|_{Z_\pm}, d\alpha) \simeq ({\xi_{Z_\pm}}, \omega_{Z}) \oplus (\xi_N, \omega_N)$$ 
where ${\xi_{Z\pm}} = TZ_\pm \cap \xi$ as mentioned above,  $\xi_N$ is the symplectic complement of $\xi_Z$ in $\xi$, and $\omega_Z$, $\omega_N$ are the restrictions of $d\alpha$ to $\xi_Z$, $\xi_N$ respectively.
\begin{definition}\label{def:asymptotically_cylindrical}
The hypersurface $Z \subset \hat{Y}$ is \emph{asymptotically cylindrical} over $Z_\pm\subset Y$ if there exists $s$-families of sections of $\xi_N$, i.e. $\zeta_+:[R',\infty)\rightarrow \Gamma(\xi_N|_{Z_+})$ and  $\zeta_-:(-\infty,-R']\rightarrow \Gamma(\xi_N|_{Z_-})$ such that  
\begin{align*}
    Z\cap \left([R', \infty)\times Y\right) &= \bigcup_{(r,p)\in[R,\infty)\times {Z_+}} \widetilde{\exp}_{(r,p)}(\zeta_+(r)(p)),\\
    Z\cap \left((-\infty, R'])\times Y\right) &= \bigcup_{(r,p)\in(-\infty,R']\times {Z_-}} \widetilde{\exp}_{(r,p)}(\zeta_-(r)(p)).
\end{align*}w
\end{definition}
\begin{remark}
We will consider two simple cases of asymptotically cylindrical hypersurfaces. The first is a symplectization of a contact submanifold, namely $Z:= \hat Z_+$. The second case is when $Y$ is 3-dimensional and $Z$ is a pseudoholomorphic cylinder. It follows from \cite[Theorem 2.2]{moreno2019holomorphic} that any pseudoholomorphic cylinder is asymptotically cylindrical over its ends (see also Definition~\ref{def:deformation_for_intersection} below).
\end{remark}
Given a $J$-holomorphic cylinder $u$ in $\hat{Y}$ with  non-degenerate ends $\gamma_{\pm}^u$, the works \cite{moreno2019holomorphic, siefringHigherDim} define a ``holomorphic intersection number'' $u * {Z}$ of $u$ with the manifold ${Z}$ and prove a ``positivity of intersection'' property (see Theorem \ref{thm:generalized_intersection}). The holomorphic intersection number is defined as a sum of a ``relative intersection number" and a contribution from the Conley-Zehnder index of the ends in the normal contact direction $\xi_N$. 

\begin{definition}[Holomorphic intersection number,  \cite{moreno2019holomorphic, siefringHigherDim}]
Let $Y$, $Z$, $u$ and $\gamma_\pm^u$ be as above. Fix a trivialization   of $\xi_N$ along the orbits of $R$ that lie in $Z$ (if there are any) and collectively denote it by $\tau$.
\begin{itemize}
    \item (Relative intersection number). Let $u^\tau$ be a deformation of $u$ as described in Definition~\ref{def:deformation_for_intersection} below. Roughly speaking, this deformation uses a $\tau$-constant section of $\xi_N$ to push any ends of $u$ that lie in $Z$ off of it. In particular, if the ends of $u$ do not lie in $Z$ then $u^\tau=u$. Define the \emph{relative intersection number} of $u$ with $Z$ to be
    \begin{equation}\label{eq:relative_intersection_number}
        \iota^\tau(u,Z):= u^\tau \cdot  Z\in \Z,
    \end{equation}
    namely, the standard transverse intersection number of the deformation $u^\tau$ with the hypersurface $Z$.\\
    
    \item (Normal Conley--Zehnder index). Let $\gamma$ be a non-degenerate Reeb orbit which lies in $Z_\pm$. Its  \emph{normal Conley--Zehnder index} $\indCZ_\tau^N(\gamma)$ is the Conley--Zehnder index of the path of $2 \times 2$ symplectic matrices defined by applying the trivialization $\tau$ to the projection of the linearized Reeb flow to $\xi_N$.
    For an orbit $\gamma$ that does not lie in $Z$, we define $\indCZ_\tau^N(\gamma):=0$.\\
    
    \item The \emph{holomorphic intersection number} of $u$ and $Z$ is defined to be 
    \begin{equation}\label{eq:holomorphic_intersection}
        u*Z  := \iota^\tau(u, Z) + \lfloor\indCZ_\tau^N(\gamma^u_+)/2\rfloor -\lceil \indCZ_\tau^N(\gamma^u_-)/2\rceil.
    \end{equation}
\end{itemize}
\end{definition}

The deformation $u^\tau$ required for the definition of the relative intersection number (\ref{eq:relative_intersection_number}) is constructed as follows.
\begin{definition}[Deformation $u^\tau$ of $u$]\label{def:deformation_for_intersection}
We write down the deformation at the positive end, the deformation at the negative end is analogous with the appropriate changes in sign. If $\gamma_+^u$ is not contained in $Z_+$, then we do not deform $u$ near the positive end. Suppose $\gamma_+^u$ is contained in $Z_+$ and extend the trivialization $\tau$ to a trivialization $\tau: U \times \R^2 \to \xi_N|_{U}$ on some open neighborhood $U\subset Z_+$ of $\gamma_+^u$. Then, \cite[Theorem $2.2$]{moreno2019holomorphic} states that the map $u$ may be written as
    \begin{align*}
        (s,t) \mapsto \widetilde{\exp}_{u_T(s,t)}(u_N(s,t)),
    \end{align*}
    where:\begin{itemize}
        \item $(s,t) \in [R, +\infty) \times S^1$ for  $R \gg 1$ large enough, 
        \item $u_T(s,t)$ lies in $\R\times U\subset \R\times Z_+$, 
        \item $u_N$ is a smooth section of $u_T^*\pi^*\xi_N$, where $\pi$ is the projection $\R\times Z_+\rightarrow Z_+$.
    \end{itemize} 
    We then perturb the map $u$ by replacing the above parametrization of $u$ near $+\infty$ by the map
    \begin{align}\label{eq:deformation_for_intersection}
        (s,t) \mapsto \widetilde{\exp}_{u_T(s,t)} \left(u_N(s,t) + \beta(|s|)\tau(u_T(s,t), \epsilon)\right)
    \end{align}
    where $\beta: [0,\infty) \to [0,1]$ is a smooth cut-off function equal to $0$ for $s < R + 1$ and equal to $1$ for $s > R+2$, and $0\neq \epsilon\in\R^2$. 
\end{definition}

The following theorem is a restriction of Theorem 2.5 from \cite{moreno2019holomorphic} adapted to our notations.
\begin{thm}[{\cite[Theorem 2.5]{moreno2019holomorphic}}] \label{thm:generalized_intersection}
Let $Y$, $Z$, $u$ be as above, and assume further that the image of $u$ is not contained in ${Z}$. Then, $u * Z \geq 0$ and it is equal to $0$ if and only if the image of $u$ does not intersect $Z$. 
\end{thm}

\subsubsection{Proof of uniqueness} In this section we use the intersection theory reviewed above to prove Proposition~\ref{prop:uniqueness}. 
We continue with Setup~\ref{set:intersection_theory}. Our first step is to show that the relative intersection number $\iota^\tau$ of any $J_f$-holomorphic cylinder $u$ in $Y_{2i+1}$ with a Reeb invariant codimension-2 hypersurface $Z$ is zero. We will later apply this lemma to $Z=\hat Y_{2i-1}$ when $i>1$ and to $Z=im(u_\eta)$ when $i=1$. 
\begin{lemma}\label{lem:vanishing_relative_intersection}
    Let $u$ be a $J_f$-holomorphic cylinder in $\hat Y_{2i+1}$, and let $Z\subset \hat Y_{2i+1}$ be a codimension-2 asymptotically cylindrical hypersurface that is invariant under $\varphi_t$. Then, $\iota^\tau(u, Z)=0$.
\end{lemma}
\begin{proof}
 Let $u^\tau$ be the deformation of $u$ as in Definition~\ref{def:deformation_for_intersection}, and denote by $\gamma_\pm^{u^\tau}$ the ends after the deformation.
 Let $\ell:\R\rightarrow Y_{2i+1}$ be a curve that connects $\gamma_\pm^{u^\tau}$, namely $\lim_{s\rightarrow\pm\infty}\ell(s)\in \gamma_\pm^{u^\tau}$. Since $Y_{2i+1}\setminus Z$ is connected, we can choose the path $\ell$ to not intersect $Z$.
 Consider the cylinder $$\sigma:\R\times S^1\rightarrow \hat Y_{2i+1}, \qquad \sigma(s,t):=(s, \varphi_t \ell(s))\in \R\times Y_{2i+1}=\hat Y_{2i+1}.$$ 
 Since $\ell$ does not intersect the $\varphi_t$-invariant submanifold $Z$, its orbit under this action does not intersect it as well. Therefore, the cylinder $\sigma$ is disjoint from $Z$ (see Figure~\ref{fig:intersection_zero_proof}). Moreover, our assumption that $H_2(Y_{2i+1})=0$ guarantees that the union of $u^\tau$ and $\overline{\sigma}$ is null-homologous. Consider the compactification $\overline{Y}$ of $\hat Y$ into a manifold with boundary, diffeomorphic to $[0,1]\times Y$. By the homology invariance of the standard algebraic intersection number (e.g., \cite[Part VI, Section 11]{bredon2013topology}), the intersection number of $u\#\overline{\sigma}$ with the compactification of $Z$ vanishes. Thus,
 \begin{equation*}
     \iota^\tau(u, Z)=u^\tau\cdot Z {=} \sigma\cdot Z = 0. \qedhere
 \end{equation*}
\end{proof}
\begin{figure}[h]
    \centering
    \includegraphics[width=0.9\textwidth]{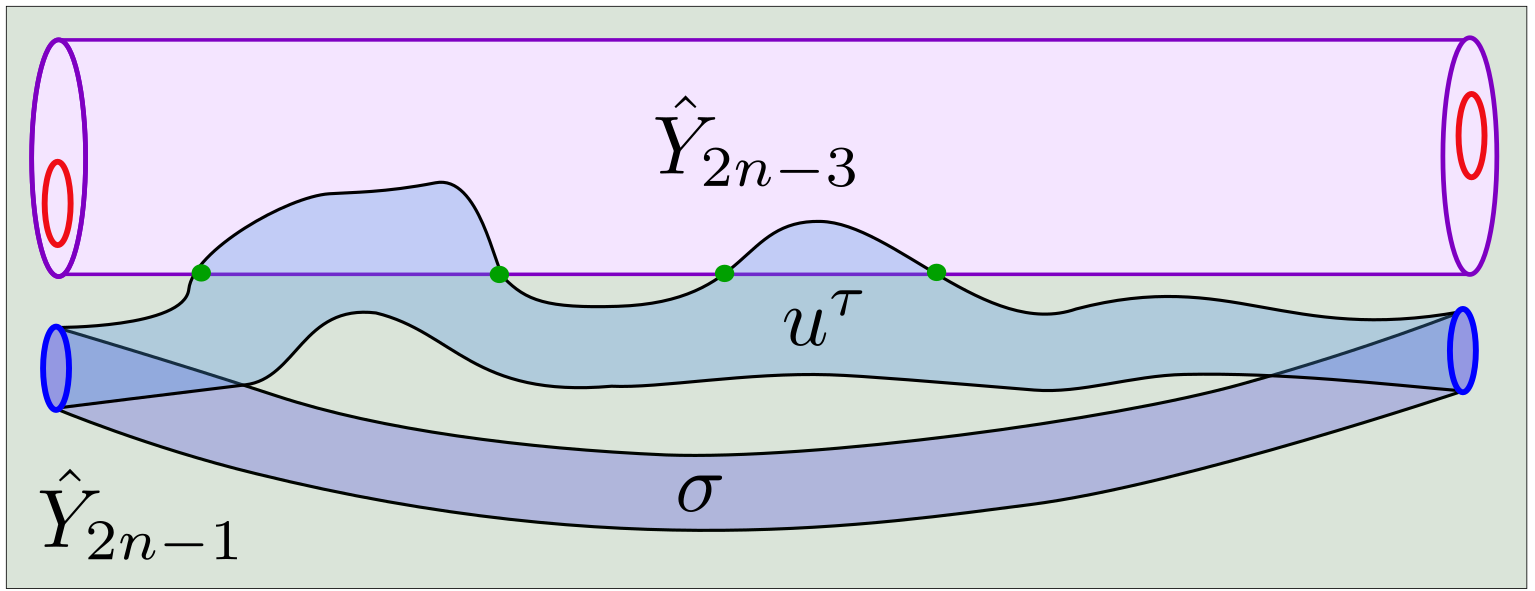}
    \caption{An illustration of the deformation $u^\tau$ of $u$, whose ends are disjoint from the hypersurface, and the cylinder $\sigma$ that has the same ends and does not intersect $\hat Y_{2n-3}$.}
    \label{fig:intersection_zero_proof}
\end{figure}

The next lemma concerns the normal Conley-Zehnder indices of the periodic orbits of the $f$-perturbed Reeb flow $\varphi_t^f$. This will be useful for computing the holomorphic intersection number of a $J_f$-holomorphic cylinder with the submanifolds in the contact flag.
\begin{lemma}\label{lem:compute_CZ_N}
    Take $Y=Y_{2i+1}$ and $Z=\hat Y_{2i-1}$ for $i\in\{2,\dots,n-1\}$. Let $\gamma$ be a periodic orbit of the perturbed contact form $\alpha_f$ that  lies in $Y_{2i-1}$. There exits a trivialization $\tau$ of $\xi_N$ such that the normal CZ index of $\gamma$ is 
    \begin{equation*}
        \indCZ_\tau^N(\gamma)=\begin{cases}
        +1 & \text{if } \gamma=\gamma_+,\\
        -1 & \text{if } \gamma\neq\gamma_+.
        \end{cases}
    \end{equation*}
\end{lemma}
\begin{proof}
    Let $\tau$ be any trivialization of $\xi_N$ along $\gamma$ that is invariant under the periodic Reeb flow $\varphi_t$, namely, $\tau:\R^2\times \gamma\rightarrow \xi_N$ satisfying 
    \begin{equation}\label{eq:tau_is_R_compatible}
       d \varphi_t\circ \tau(\cdot, x) = \tau(\cdot, \varphi_t(x))
    \end{equation}
    for all $x \in \gamma$. In this trivialization the linearized flow of $R_f$ is given by
\begin{align*}
    \Phi(t) &:= \tau(\cdot, \gamma(t))^{-1}\circ d\varphi^{f}_t\circ \tau(\cdot, \gamma(0))\\
    &=  \tau(\cdot, \varphi^{f}_t\gamma(0))^{-1}\circ d\varphi^{f}_t\circ \tau(\cdot, \gamma(0)).
\end{align*}    
Since $\gamma$ is a periodic orbit of $R_f$, it is a critical circle of $f$. By the definition of $R_f$ stated in (\ref{itm:f_perturbation}), it is proportional to $R$ wherever $X_f$ vanishes. Therefore,  
\begin{align*}
    \Phi(t) &=  \tau(\cdot, \varphi_{e^{\varepsilon f}\cdot t}\gamma(0))^{-1}\circ d\varphi^{f}_t\circ \tau(\cdot, \gamma(0))\\
    &\overset{(\ref{eq:tau_is_R_compatible})}{=} \tau(\cdot, \gamma(0))^{-1}\circ (d\varphi_{e^{\varepsilon f}\cdot t})^{-1} \circ d\varphi^{f}_t\circ \tau(\cdot, \gamma(0)).
\end{align*}
Denoting $\phi_t:= (\varphi_{e^{\varepsilon f}\cdot t})^{-1} \circ \varphi^f_t$, the linearization $d\phi_t$  is conjugate to $\Phi(t)$  by $\tau(\cdot, \gamma(0))$.  Identifying a neighborhood of $\gamma(0)$ with its Darboux chart, the path $d\phi_t$ of matrices solves the ODE $\dot{d\phi_t} = -J_0 \cdot \on{Hess}(\varepsilon f) d\phi_t$.  When $\varepsilon$ is sufficiently small, the path $\Phi(t)$ crosses the Maslov cycle only at the origin and the crossing form is $\Gamma(\gamma(0)):=-\on{Hess}(\varepsilon f)_{\gamma(0)}|_{\xi_N}$. By assumption (\ref{itm:contact_flag}).\ref{itm:Hessian_positive} from Setup~\ref{set:intersection_theory}, the Hessian of $f$ is positive definite on $\xi_N$ unless $\gamma=\gamma_+$, in which case it is negative definite. Therefore, the normal Conley-Zehnder index of $\gamma$ is $\frac{1}{2}\cdot (+2) = +1$ if $\gamma = \gamma_+$ and $\frac{1}{2}\cdot (-2) = -1$ otherwise.
\end{proof}

We are now ready to prove that, under the assumptions of Setup~\ref{set:intersection_theory}, the only point in the compactified moduli space $\moduli$ is $u_\eta$.
\begin{proof}[Proof of Proposition~\ref{prop:uniqueness}]
Let $\overline u = (u_1,\dots, u_k)\in \moduli$ be a building of $J_f$-holomorphic curves. Lemma~\ref{lem:no_buildings} above states that the components $u_j$ of $\overline{u}$  are all cylinders. Therefore, we may apply Lemma~\ref{lem:vanishing_relative_intersection} and conclude that $\iota^\tau(u_j,\hat Y_{2n-3})=0$ for all $j$. Denoting the negative and positive ends of $u_j$ by $\gamma_{j-1}$ and $\gamma_j$ respectively, we notice that $\gamma_{j-1}\neq \gamma_+$ for all $j$. This is due to the fact that $\gamma_+$ is a global maximum for $f$ and that the $\alpha_f$-action is decreasing along $J_f$-holomorphic curves. Therefore, Lemma~\ref{lem:compute_CZ_N} asserts that, for a properly chosen trivialization $\tau$, $\indCZ_\tau^N(\gamma_{j-1}) = -1$ and
$\indCZ_\tau^N(\gamma_{j}) \leq 1$ for all $j$. It follows that the holomorphic intersection numbers of each piece with the hypersurface $Z=\hat Y_{2n-3}$ is non-positive. Indeed,
\begin{align*}
    u_j*\hat Y_{2n-3} &:= \iota^\tau(u_j, \hat Y_{2n-3} ) +\lfloor \indCZ_\tau^N(\gamma_{j})/2\rfloor - \lceil\indCZ_\tau^N(\gamma_{j-1})/2\rceil\\
    &= 0+\lfloor\indCZ_\tau^N(\gamma_{j})/2\rfloor - \lceil-1/2\rceil = \lfloor\indCZ_\tau^N(\gamma_{j})/2\rfloor\leq 0.
\end{align*}
Theorem~\ref{thm:generalized_intersection} (which is a special case of \cite[Theorem 2.5]{moreno2019holomorphic}) states that, in this case, each $u_j$ is either disjoint from $\hat Y_{2n-3}$ or contained in it.  By definition of $\moduli$, the image of $\overline{u}$ contains the point $(0,z)\in \hat Y_3\subset \cdots\subset \hat Y_{2n-3}$. Denoting by $u_{j_0}$ the component of $\overline{u}$ whose image contains $(0,z)$, we find that $u_{j_0}$ is not disjoint from $\hat Y_{2n-3}$, and thus is contained in it. 

\begin{figure}[h]
\centering
\includegraphics[width=.8\textwidth]{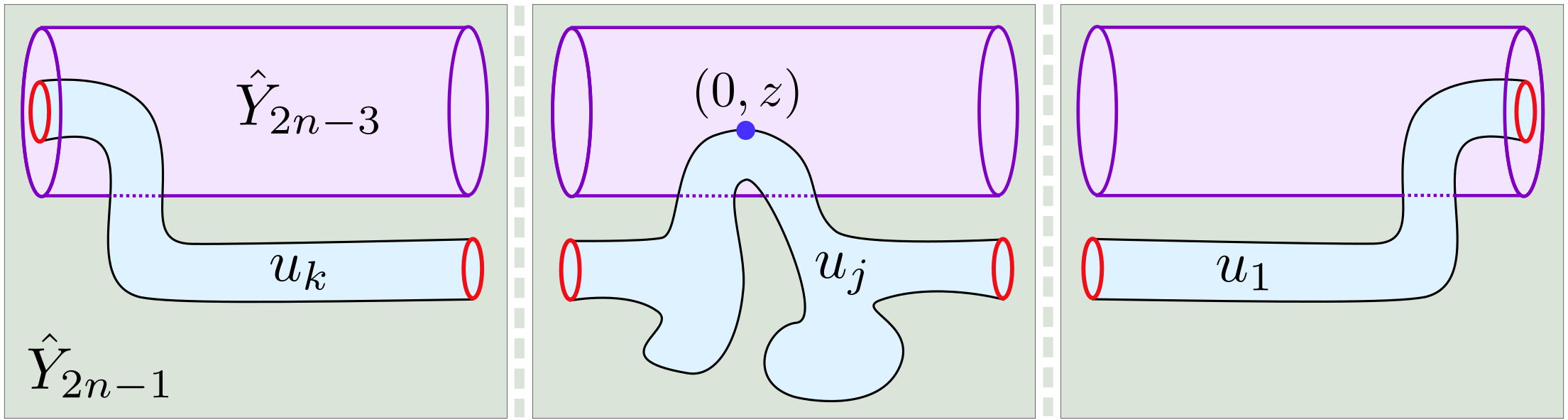}
\caption{The proof of Proposition~\ref{prop:uniqueness} uses positivity of intersection to rule out the illustrated scenario, in which the piece of the building that intersects the hypersurface is not contained in it.}
\label{fig:uniqueness_of_cylinder}
\end{figure}

We now consider the curve $u_{j_0}$ in $\hat Y_{2n-3}$. Arguing as above, we see that its holomorphic intersection number with the hypersurface $\hat Y_{2n-5}$ is again non-positive. Since its image intersects $\hat Y_3\subset \cdots \hat Y_{2n-5}$ we conclude that $u_{j_0}\subset \hat Y_{2n-5}$. Continuing by induction we conclude that the image of $u_{j_0}$ is contained in $\hat Y_3$. 

To finish the proof, notice that the above arguments imply that the holomorphic intersection number of $u_{j_0}$ and $u_\eta$ in $\hat Y_3$ is non-positive. Since these two $J_f$-holomorphic curves intersect at $(0,z)$, it follows from Theorem~\ref{thm:generalized_intersection} that $u_{j_0} = u_\eta$. In particular, their ends agree:  $\gamma_j=\gamma_+$ and $\gamma_{j-1} = \gamma_-$. Recalling that $\gamma_+$ and $\gamma_-$ are global maxima and minima for $f$, it follows from action considerations that the only ends that could coincide with them are $\gamma_k$ and $\gamma_0$ respectively, we conclude that the building $\overline{u}$ consists of a single cylinder, $u_{j_0} = u_\eta$, which concludes the proof. 
\end{proof}

\subsection{Transversality}\label{subsec:transversality}
The aim of this section is to show that the $J_f$-holomorphic curve $u_\eta$, described in Section~\ref{subsec:lifting_flow_lines}, is cut-out transversely. 

\subsubsection{Local transversality of the moduli space without a point constraint.}

We start by showing that the moduli space without the point constraint is locally transversely cut-out at $u_\eta$.  
The notion of a moduli space being locally  transversely cut-out at a curve appears in \cite[Section 2.11]{p2015}; we give a brief overview here for convenience.  
\begin{definition}[{\cite[Section 2.11]{p2015}}] Let $u$ be a $J_f$-holomorphic curve in $\hat{Y}$. 
    \begin{itemize}
        \item For $k\geq 0$ and $\delta<1$, define the associated \emph{weighted Sobolev space}
        \begin{equation}
            W^{k,2,\delta}(u^*T\hat Y) = \left\{ f \ | \ \mu\cdot f\in W^{k,2}(u^*T\hat Y)\right\}, \qquad \|f\|_{k,2,\delta}:=\|\mu\cdot f\|_{k,2},
        \end{equation}
        where $\mu:\R\times S^1\rightarrow \R_{>0}$ equals 1 away from the ends and equals $e^{\delta|s|}$ near the ends.
        \item The linearized operator\footnote{The domain of the linearized operator considered by Pardon in \cite{p2015} is slightly bigger. Since our aim here is to prove surjectivity this does not cause any issue.}
        \begin{equation}
        D_u:W^{k,2,\delta}(u^*T\hat Y)  \rightarrow W^{k-1,2,\delta}\left(u^* T\hat Y \otimes_\C \Omega^{0,1}(\R\times S^1)  \right),
        \end{equation}
        of $u$ is defined as follows. Given any symmetric connection $\nabla$ on $T\hat Y$, the linearized operator takes the form \cite[Section 2.1]{wendlsymplectic} 
        \begin{equation}\label{eq:lin_op_connection}
            \D_u w = \nabla w + J_f(u)\circ \nabla w \circ j +(\nabla_w J_f)\cdot du \circ j
        \end{equation}
        \item A $J_f$-holomorphic curve $u$ is \emph{transversely cut-out} if the linearized operator $D_u$ is surjective. 

    \end{itemize}

\end{definition}
\begin{prop}\label{prop:transversality}
Consider Setup~\ref{set:intersection_theory}. The linearized Cauchy-Riemann operator at $u_\eta$ is surjective.
\end{prop}

Our strategy for proving Proposition~\ref{prop:transversality} is to show that the linearized operator splits into a direct sum of real-linear Cauchy--Riemann operators on complex line bundles. Then, we apply the automatic transversality results from \cite{w2010} to deduce surjectivity. 

Recall the contact flag $Y_3\subset\cdots\subset Y_{2n-1}=Y$. It defines a splitting 
\begin{equation}\label{eq:ctct_struct_splitting}
    \xi=\xi_3\oplus\cdots\oplus \xi_{2n-1}
\end{equation}
of the contact distribution into two-dimensional sub-bundles as follows. The bundle $\xi_3$ is defined to be the contact structure of $Y_3$. For $i>1$, the bundle $\xi_{2i+1}$ is defined to be the symplectic complement of $TY_{2i-1}$ in the contact structure of $Y_{2i+1}$. Denoting by $V$ the span of $\partial_r$ and $R$, we have a splitting 
\begin{equation*}
    T\hat Y_{2i+1} = V\oplus \left(\bigoplus_{k=1}^{i}\xi_{2k+1}\right), \quad i=1,\dots,n-1.
\end{equation*}
Moreover, it follows from the assumptions in Setup~\ref{set:intersection_theory} that $J_f$ preserves this splitting. For any $i = 1, \ldots, n - 1$, we write $\D_i$ for the linearized operator of $u_\eta$, considered as a $J_f$-holomorphic curve in $\hat{Y}_{2i+1}$. We will also denote the full linearized operator $\D_{n-1}$ by $\D$. 

\begin{lemma}[Decomposing the linearized operator] \label{lem:linearized_op_splits}
    For any $2\leq i \leq n+1$ there exists a Fredholm real-linear Cauchy--Riemann operator $L_i$ on the bundle $\xi_{2i+1}$ such that the operator $\D_i$ has the matrix form  $$\begin{pmatrix} \D_{i-1} & 0 \\ 0 & L_i \end{pmatrix}$$
    with respect to the splitting
    $T\hat{Y}_{2i+1} = T\hat{Y}_{2i-1} \oplus \xi_{2i+1}.$
\end{lemma}
The linearized operator can be written in terms of a symmetric connection on $T\hat Y$, as in (\ref{eq:lin_op_connection}). It will be convenient, for the purposes of computation, to use the Levi--Civita connection of a metric which is well-adapted to the contact flag.
\begin{lemma}[Choice of connection] \label{lem:connection_split}
    There exist symmetric connections $\nabla^i$ on $T\hat Y_{2i+1}$, such that for any $v$ tangent to the hypersurface $Y_{2i-1}$, the operator $\nabla^i_v$ can be written as
    \begin{equation}\label{eq:connection_split}
        \nabla^{i}_v = \begin{pmatrix} \nabla^{{i-1}}_v & 0 \\ 0 & \nabla'_v \end{pmatrix}
    \end{equation}    
    with respect to the splitting $T\hat Y_{2i+1}= T\hat Y_{2i-1}\oplus \xi_{2i+1}$.
\end{lemma}
\begin{proof}
Let $h_1$ be the metric on $Y_3$ equal to the one induced by $J_f$. Define for each $i \in \{2, \ldots, n-1\}$ a metric $h_i$ on $Y_{2i+1}$ which satisfies the following conditions:
\begin{itemize}
    \item The submanifold $Y_{2i-1}$ is totally geodesic with respect to $h_i$. 
    \item The splitting
    $$TY_{2i+1}|_{Y_{2i-1}} = TY_{2i-1}\oplus \xi_{2i+1}$$
    is $h_i$-orthogonal. 
\end{itemize}
The metrics $h_i$ can be constructed inductively by repeated applications of the following standard trick. Fix any Riemannian metric $h$ on $Y_{2i+1}$ which restricts to $h_{i-1}$ on $Y_{2i-1}$ and makes $\xi_{2i+1}$ orthogonal to $TY_{2i+1}$. Use the restriction of the exponential map to $\xi_{2i+1}$ to define a tubular neighborhood $U$ of $Y_{2i-1}$ in $Y_{2i+1}$. The involution $v \mapsto -v$ on $\xi_{2i+1}$ then defines a smooth involution $\iota: U \to U$ that fixes $TY_{2i-1}$. Set $h_i := \frac{1}{2}(h + \iota^* h)$ on $U$ and extend it outwards.  Then $\iota$ is an isometry for $h_i$. Given any $(p,w)\in T Y_{2i-1}$, the corresponding geodesic is mapped under $\iota$ to a geodesic corresponding to $(\iota(p),\iota(w)) = (p,w)$. In other words this geodesic is preserved by $\iota$ and this is in $Y_{2i-1}$. 


For each $i \in \{2, \ldots, n-1\}$, let $\hat{h}_i$ denote the cylindrical metric over $h_i$ on the cylinder $\hat{Y}_{2i+1}$. Let $\nabla^{i}$ denote the Levi--Civita connection on $\hat{Y}_{2i+1}$ with respect to $\hat{h}_i$. By definition, $\nabla^i$ is symmetric. The splitting (\ref{eq:connection_split}) follows from the fact that $\hat Y_{2i-1}$ is totally geodesic with respect to $\hat h_i$ and hence the $\hat h_i$ parallel transport preserves the orthogonal direct sum decomposition  $TY_{2i+1}|_{Y_{2i-1}} = TY_{2i-1}\oplus \xi_{2i+1}$.
\end{proof}

The next lemma states that the asymptotic operators of $\D$, the linearized operator of $u_\eta$ in $\hat{Y}$, respect the direct sum decomposition induced by the contact flag, and that the pieces in the contact direction are non-degenerate. \begin{lemma}\label{lem:asymptotic_split}
    The asymptotic operators $A_{\pm}$ of $\D$
    respect the splitting (\ref{eq:ctct_struct_splitting}). Moreover, their restrictions to $\gamma_\pm^*\xi_{2j+1}$ are non-degenerate.
\end{lemma}
\begin{proof}
Consider the connections from Lemma~\ref{lem:connection_split}, and write $\nabla:=\nabla^{n-1}$ for the connection on the total manifold $\hat Y$. Following \cite{w2010}, the asymptotic operators $A_\pm$ are defined by $\lim_{s\rightarrow \pm\infty} D(\partial_s) = \nabla_s -A_\pm$.
Evaluating (\ref{eq:lin_op_connection}) as $s\rightarrow\pm \infty$ we obtain
\begin{align*}
    \lim_{s\rightarrow \pm\infty} D(\partial_s) &= \lim_{s\rightarrow \pm\infty} \nabla_s +J_f\nabla_t+ (\nabla J_f)\partial_t u\\
    &=\nabla_s+J_f\nabla_{T\cdot R_f}+(\nabla J_f) T\cdot R_f = \nabla_s+TJ_f\nabla_{R_f} + T\nabla (J_f R_f) - TJ_f\nabla R_f \\
    &=\nabla_s +T \cdot J_f(\nabla_{R_f}-\nabla R_f) + T\nabla (J_f R_f).
\end{align*}
Notice that $J_fR_f = \partial_r$ is covariantly constant, since the metric defining $\nabla$ is cylindrical. Therefore the asymptotic operators are \begin{equation}
    A_\pm =- T \cdot J_f(\nabla_{R_f}-\nabla R_f) + T\nabla (J_f R_f) = -T\cdot J_f[R_f,-].
\end{equation}
Since $J_f$ and the flow of $R_f$ preserve the splitting, the operator $J_f[R_f,-]$ does as well. Denoting by $A_\pm^i$ the restrictions of $A_\pm$ to $\gamma_\pm^*\xi_{2i+1}$, it remains to show that they are non-degenerate. Notice that
the operators $A_\pm^i$ coincide with the restriction to $\gamma_\pm^*\xi_{2i+1}$ of the linearization of the flow $\varphi_t^f$ of $R_f$. Let $\tau_\pm$ be trivializations of $\xi_{2i+1}$ along $\gamma_\pm$ that are invariant under the periodic Reeb flow $\varphi_t$, as in \ref{eq:tau_is_R_compatible}. 
As shown in the proof of Lemma~\ref{lem:compute_CZ_N}, the linearization of the flow $\varphi_t^f$ is conjugate to a path of matrices $d\phi_t$ that  solves the ODE $\dot{d\phi_t} = -J_0 \cdot \on{Hess}(\varepsilon f) d\phi_t$. Note that the Hessian of $f$ is degenerate in the direction of the Reeb vector field $R$, since $f$ is $R$-invariant. However, by the Morse-Bott condition, the restriction of the Hessian to  $\xi_{2i+1}$, is non-degenerate. When $\varepsilon$ is sufficiently small, this implies that $\{d\phi_t\}_{t\in[0,T]}$ is non-degenerate as well.
\end{proof}

We are now ready to show that the linearized operator splits into a direct sum. 
\begin{proof}[Proof of Lemma~\ref{lem:linearized_op_splits}]
Let $\nabla^i$ be the connection from Lemma~\ref{lem:connection_split}, and write $\D_i$ as
\begin{equation*}
    \D_i = \underbrace{\nabla^i + J_f\circ \nabla^i\circ j}_{(i)} + \underbrace{(\nabla^i J_f) \cdot du_\eta\circ j}_{(ii)}.
\end{equation*}
The splitting of the first part, $(i)$, follows immediately from the splitting of $\nabla^i$ and the fact that $J_f$ preserves the decomposition. Indeed, since $\partial_su_\eta$ and $\partial_t u_\eta$ lie in the subspace $T\hat Y_{2i-1}$, Lemma~\ref{lem:connection_split} guarantees that $\nabla^i_{\partial_su_\eta} = \nabla^{i-1}_{\partial_s u_\eta}\oplus \nabla'_{\partial_s u_\eta}$, and the same for $\partial_t u_\eta$. Therefore,
$$(i) = \begin{pmatrix} \nabla^{{i-1}} +  J_f(u_\eta) \nabla^{{i-1}}\circ j & 0 \\ 
0 & \nabla' + J_f(u_\eta)\nabla'\circ j.\end{pmatrix}.$$
Let us rewrite the second part of $\D_i$ applied to $\partial_s$ (the computation for $\partial_t$ is completely analogous):
\begin{align}
    (ii)(\partial_s) &= (\nabla^i J_f) \cdot du_\eta \circ j \partial_s = (\nabla^i J_f) \cdot \partial_t u_\eta \nonumber\\
    &=  (\nabla^i J_f) \cdot T\cdot R_f {=} T\cdot \left(\nabla^i(J_fR_f)- J_f\nabla^i R_f \right) \nonumber\\
    &=T\cdot \left(\nabla^i(-\partial_r)- J_f\nabla^i R_f \right)\nonumber\\
    &= T\cdot \left(0- J_f\nabla^i R_f \right) = -TJ_f\nabla^i R_f.
\end{align}
where, in the second and third lines, we used the fact that $\partial_tu_\eta = T\cdot R_f$ by construction, and that $R_f=-J_f\partial_r$ respectively. In the last row we used the fact that the metric $\hat h^i$ is cylindrical in the $r$-direction, which implies that the vector field $\partial_r$ is covariantly constant. We now evaluate $(ii)(\partial_s) = -TJ_f\nabla^i R_f$ on a tangent vector field $v$ and split into cases with respect to the decomposition   $TY_{2i+1}|_{Y_{2i-1}} = TY_{2i-1}\oplus \xi_{2i+1}$:
\begin{itemize}
    \item \underline{$v\in \Gamma(u_\eta^*TY_{2i-1})$}: Using Lemma~\ref{lem:connection_split} and the fact that $R_f$ is tangent to $\hat Y_{2i-1}$ as well, we obtain $\nabla^i_v R_f = \nabla^{i-1}_v R_f$. Therefore $(ii)(\partial_s) v= -TJ_f\nabla^{i-1}_v R_f$.
    \item \underline{$v\in \Gamma(u_\eta^* \xi_{2i+1})$}: Extend $v$ into a vector field $\tilde v$ defined on a neighborhood of the image of $u_\eta$ in $\hat Y_{2i+1}$. Then, along the hypersurface $\hat Y_{2i-1}$ we have 
    \begin{align*}
        (ii)(\partial_s)v &= -T J_f \nabla^i_v R_f = -T J_f \left(\nabla^i_{R_f}\tilde v +[\tilde v, R_f] \right) \\
        &=T J_f \left([R_f,\tilde v]-\nabla^i_{R_f}\tilde v \right) = T J_f \left([R_f, v]-\nabla^i_{R_f} v \right)\\
        &\overset{Lemma~\ref{lem:connection_split}}{=} T J_f \left([R_f, v]-\nabla_{R_f}' v \right) = J_f [\partial_t u_\eta, v]-J_f \nabla_{\partial_t u_\eta}' v\\
        &=  J_f [du_\eta\circ j(\partial_s), v]-J_f \nabla' \circ j (\partial_s) (v).
    \end{align*}
    Above, we used the fact that the connection $\nabla^i$ is symmetric and that $R_f$ is tangent to the hypersurface $\hat Y_{2i-1}$. 
\end{itemize}
Overall we conclude that the second part of $\D_i$ decomposes as 
$$(ii) = \begin{pmatrix} (\nabla^{i-1} J_f) \cdot du_\eta\circ j & 0 \\ 
0 & J_f [du_\eta\circ j, -]-J_f \nabla' \circ j\end{pmatrix}.$$
Summing $(i)$ and $(ii)$ we obtain a decomposition of $\D_i$:
$$\D_i = \begin{pmatrix} \D_{i-1} & 0 \\ 
0 & \nabla' + J_f [du_\eta\circ j, -] \end{pmatrix}.$$
A simple computation shows that $L_i:= \nabla' +J_f [du_\eta\circ j, -]$ is a Cauchy-Riemann operator.
Moreover, since $\partial_t u_\eta = TR_f$, the asymptotic operators of $L_i$ are indeed the restrictions of $A_\pm$ to $\gamma_\pm^* \xi_{2i+1}$. By  Lemma~\ref{lem:asymptotic_split}, they are non-degenerate. This implies that $L_i$ is a Fredholm operator (e.g. \cite[Section 2.1]{w2010}), and thus concludes the proof.  
\end{proof}
Having established the splitting of the linearized operators $\D_i$, we are ready to prove surjectivity.
\begin{proof}[Proof of Proposition~\ref{prop:transversality}] 
    We will prove by induction that the Cauchy-Riemann operators $\D_i$ are surjective for all $i$. This will use the decomposition from Lemma~\ref{lem:linearized_op_splits}, as well as automatic transversality \cite{w2010} for $\D_1$ and $L_i$. Starting with the base case of the induction, $\D_1$ is the linearized operator at the (non-constant) curve $u_\eta$ inside the 4-dimensional manifold $\hat Y_3$. Theorem~1 from \cite{w2010} states that\footnote{We only state a special case of Wendl's theorem, adapted to our notations.}, in a 4-dimensional manifold, an immersed pseudoholomorphic curve is regular if 
    \begin{equation}\label{eq:Wendl_cond_4d}
        \ind(u)> \frac{1}{2}(\ind(u) -2+2g+\#\Gamma_0 +\#\pi_0(\partial \Sigma)),
    \end{equation}
    where: \begin{itemize}
        \item $\Sigma$ is the domain of the curve,
        \item $g$ is the genus of $\Sigma$,
        \item $\Gamma_0$ is the set of ends of $u$ with even CZ index.
    \end{itemize}
    In our case, $\Sigma = \R\times S^1$, and hence $g=0$ and $\#\pi_0(\Sigma)=0$. Moreover, $\#\Gamma_0\leq 2$ since $u_\eta$ has only two ends. Therefore, condition (\ref{eq:Wendl_cond_4d}) holds if $\ind(\D_1)>0$. The latter inequality obviously holds in our case (in fact, $\ind(\D_1)=\dim{\hat Y_3}-2 =2$).
    
    Having established the base case of the induction we move on to the induction step. We assume that $\D_{i-1}$ is surjective for some $i$ and show surjectivity of $\D_i$. By the decomposition of $\D_i$ given in Lemma~\ref{lem:linearized_op_splits}, this is equivalent to showing that the real-linear Cauchy-Riemann operator $L_i$ is surjective. This is an operator on sections of a complex line bundle. Proposition~2.2 from \cite{w2010} states that\footnote{Again, we state a special case adapted to our notations.} given a complex line bundle $E\rightarrow \Sigma$, a Cauchy-Riemann Fredholm operator $L:W^{1,2}(E)\rightarrow L^2(E\otimes_\C \Omega^{0,1}(\R\times S^1)$ is surjective if
    \begin{equation}\label{eq:Wendl_cond_line_bundle}
        \ind(L)\geq 0 
        \quad\text{ and }\quad 
        \ind(L)>  \frac{1}{2}(\ind(L) -2+2g+\#\Gamma_0 +\#\pi_0(\partial \Sigma)),
    \end{equation}
    where:
    \begin{itemize}
        \item $g$ is the genus of $\Sigma$,
        \item $\Gamma_0$ is the set of ends for which the CZ index of the asymptotic operator is even.
    \end{itemize}
    As before, in our case $g=0$ and $\#\pi_0(\partial \Sigma)=0$. Therefore the second condition in (\ref{eq:Wendl_cond_line_bundle}) amounts to $\ind(L)>-2+\#\Gamma_0$. Since the domain $\Sigma$ in our case is a cylinder,  $\#\Gamma_0\leq 2$, and we are left with the requirement $\ind(L)>0$. A simple computation shows that $\ind(L_i) =2 $ for all $i$. This essentially follows from the direct sum decomposition given in Lemma~\ref{lem:linearized_op_splits}, the additivity of the Fredholm index, and the fact that $\ind(D_i) = 2i$. We therefore conclude that $L_i$ is surjective, and this completes the proof.
\end{proof}
\begin{remark}
In the above proof we showed that $L_i$ is surjective as an operator between the spaces $W^{1,p}$ and $L^p$. Note that this implies surjectivity of the corresponding operator between $W^{k,2}$ and $W^{k-1,2}$, by elliptic regularity. Indeed, the cokernel of $L_i$ is the kernel of its formal adjoint operator. Applying elliptic regularity to the formal adjoint, we find that its kernel is independent of $k$ and $p$. 
\end{remark}

\subsubsection{The evaluation map is a submersion}
Let $z$, $\eta$, $u_\eta$ be as given in Section~\ref{subsec:lifting_flow_lines}. We showed in Proposition \ref{prop:transversality} that the linearized operator
$$D: W^{k,2,\delta}(u_\eta^*T\hat{Y}) \to W^{k-1,2,\delta}\big(u^*T\hat{Y} \otimes_{\mathbb{C}} \Omega^{0,1}(\mathbb{R} \times S^1)\big)$$
is surjective. This implies that the moduli space $\mathcal{M}(\hat{Y}, J_f; \gamma_+, \gamma_-)$ of $J_f$-holomorphic cylinders near $u_\eta$ admits the structure of a smooth manifold of dimension $\text{ind}(D)$ near $u_\eta$, with tangent space at $u_\eta$ identified with $\ker(D) \oplus \R \cdot \partial_r$. The extra factor comes from the fact that we are \emph{not} quotienting by the translation action; recall that the operator $D$ controls the deformations of the quotient moduli space $\mathcal{M}(Y, J_f; \gamma_+, \gamma_-) = \mathcal{M}(\hat{Y}, J_f; \gamma_+, \gamma_-)/\R$. 

Our ultimate goal is to show that the point-constrained moduli space $\moduli$ is transversely cut-out at $u_\eta$. The point constrained moduli space is defined as the inverse image $\text{ev}^{-1}(0,z)$, where $\text{ev}$ is the evaluation map
$$
\text{ev}: \mathcal{M}(\hat{Y}, J_f; \gamma_+,\gamma_-) \times\R_s\times S^1_t\rightarrow \hat Y,\qquad (u,s,t) \mapsto u(s,t).
$$

It is transversely cut-out at $u_\eta$ if and only if $\text{ev}$ is a submersion at the point $(u_\eta, 0, 0)$. The linearization of $\text{ev}$ at $(u_\eta, 0, 0)$, which we write as $D_{\text{ev}}$, is the linear map
$$D_{\text{ev}}: \ker(D) \oplus \R \oplus T_{(0,0)}(\R \times S^1) \to T_{(0,z)}\hat{Y}$$
defined by
\begin{equation}\label{eq:D_ev}
    (V, c, v) \mapsto V(0, 0) + c \cdot \partial_r + (\partial_v u_\eta)(0,0).
\end{equation}

In order to show that the moduli space $\moduli$ is cut-out transversely at $u_\eta$, it remains to prove the following proposition. 

\begin{prop}\label{prop:ev_is_submersion} \label{prop:evaluation_map_surjective}
The map $D_{\text{ev}}$ is surjective. 
\end{prop}

The proof of Proposition~\ref{prop:ev_is_submersion} uses the surjectivity of the linearized evaluation map on Morse flow lines, as stated in the following lemma.
\begin{lemma}\label{lem:ev_Morse}
    Let $\text{Morse}(\gamma_+, \gamma_-)$ be the space of Morse flow lines from $\gamma_+$ to $\gamma_-$. Then it is a smooth manifold and, for every $x\in \text{Morse}(\gamma_+, \gamma_-)$ the map
    $$D_{\text{ev}}^{\text{Morse}}: T_x\text{Morse}(\gamma_+, \gamma_-) \to T_{x(0)} Y, \quad \text{defined by } \quad V \mapsto V(0)$$
    is surjective.
\end{lemma}
\begin{proof}
 The space $\text{Morse}(\gamma_+, \gamma_-)$ is composed of  all smooth maps $x: \R \to Y$ such that $\dot x(s) = \nabla_J f(x(s))$, which limit to points in $\gamma_+$ and $\gamma_-$ respectively as $s \to \pm \infty$.
Consider the natural evaluation map
$$\text{ev}_{\text{Morse}}: \text{Morse}(\gamma_+, \gamma_-) \to Y$$
sending a flow line $x$ to the point $x(0)$. The image of $\text{Morse}(\gamma_+, \gamma_-) $ under this evaluation map is the intersection of the unstable manifold of $\gamma_+$ and the stable manifold of $\gamma_-$. This intersection is a smooth open submanifold of $Y$. This follows from the fact that the unstable manifold of $\gamma_+$ and the stable manifold of $\gamma_-$ are both smooth open submanifolds of maximal dimension in $Y$, so they must intersect transversely. 
By existence and uniqueness of ODEs, this implies that the space  $\text{Morse}(\gamma_+, \gamma_-)$ is a smooth manifold and  that  the evaluation map is a diffeomorphism onto an open subset of $Y$. 

The tangent space $T_x\text{Morse}(\gamma_+, \gamma_-)$ at any flow line $x$ is a $2n-1$-dimensional vector space of smooth vector fields $V$ along $x$ satisfying the linearized Morse flow line equation
$$\nabla_s V - \text{Hess}_Jf(V) = 0,$$
where $\nabla$ denotes the Levi--Civita connection of the metric $g_J$ and $\text{Hess}_J(f)$ is the Hessian of $f$ with respect to this metric. 
The linearization of evaluation map at $x$ is the operator
$$D_{\text{ev}}^{\text{Morse}}: T_x\text{Morse}(\gamma_+, \gamma_-) \to T_{x(0)} Y, \qquad V \mapsto V(0).$$

Since $\text{ev}_{\text{Morse}}$ is a diffeomorphism, its linearization is surjective. More explicitly, by existence of solutions to ODEs, for any point $x_0'$ in a neighborhood of $x(0)$, there is a flow line $x'$ such that $x'(0) = x_0'$. 
\end{proof}

\begin{proof}[Proof of Proposition~\ref{prop:ev_is_submersion}]
There is a natural parameterization of a set of cylinders near $u_\eta$ in $\mathcal{M}(\hat{Y}, J_f; \gamma_+,\gamma_-)$ with a neighborhood of $\eta$ in the space of Morse flow lines from $\gamma_+$ to $\gamma_-$. Lemma~\ref{lem:ev_Morse} states that the corresponding evaluation map, which sends a flow line to its image at the point $0 \in \R$, is a submersion. We will show that the linearization of this evaluation map at $\eta$ factors through $D_{\text{ev}}$. 

In Section~\ref{subsec:lifting_flow_lines} we explained that every Morse flow line $x$ can be lifted to a $J_f$ holomorphic curve given by \begin{equation*}
u_x(s,t):=\left(a(s), \varphi_{T\cdot t}(x(\varepsilon T s))\right),
\qquad \text{where} \qquad
    \dot a(s) = T\cdot e^{\varepsilon f(x(\varepsilon T s))}, \quad a(0)=0.
\end{equation*}  
Consider the lift map $L: \text{Morse}(\gamma_+, \gamma_-) \to \mathcal{M}(\gamma_+, \gamma_-; J_f)$ that sends a flow line $x$ to its lift $u_x$. 
The space $\mathcal{M}(\hat{Y}, J_f; \gamma_+, \gamma_-)$ is cut out transversely near $u_\eta$ by Proposition \ref{prop:transversality}. Therefore we can consider the linearization of $L$ at $\eta$,
$$T_\eta L: T_\eta\text{Morse}(\gamma_+, \gamma_-) \to \ker(D) \oplus \mathbb{R},$$
which sends a vector field $V \in T_\eta\text{Morse}(\gamma_+, \gamma_-)$ to the vector field
$$T_\eta {V} (s, t) :=  {\int_0^s}\varepsilon df(V(\varepsilon T s'))\cdot \dot{a}(s') {ds'}\cdot \partial_r + d\varphi_{T \cdot t} \cdot V(\varepsilon T s)$$
in $\ker(D)$. 
Denoting by $D\Pi$ the projection $T_{(0,z)}\hat{Y} \to T_z Y$, we claim that
\begin{equation} \label{eq:eval_map_factors} D_{\text{ev}}^{\text{Morse}} = D\Pi \circ D_{\text{ev}} \circ T_\eta L.
\end{equation}
Indeed, given  $V \in T_\eta\text{Morse}(\gamma_+, \gamma_-)$, the operator $D_{\text{ev}}^{\text{Morse}}$ sends it to $V(0)$. On the other hand, since the coefficient of $\partial_r$ in $T_\eta V$ vanishes when $s=0$, the projection  $D\Pi \circ D_{\text{ev}}$ sends 
$T_\eta V(0,0)$ to $V(0)$ as well. 

Now we use \eqref{eq:eval_map_factors} to conclude the proposition. By Lemma~\ref{lem:ev_Morse}, the operator $D_{\text{ev}}^{\text{Morse}}$ is surjective. 
This implies that the map $D\Pi \circ D_{\text{ev}}$ must be surjective, so by definition of $D\Pi$ we deduce 
$$\text{Span}(\partial_r) + \text{Im}(D_{\text{ev}}) = \text{Im}(D\Pi\circ D_{\text{ev}}) = T_{(0,z)}\hat{Y}.$$ 
Recalling (\ref{eq:D_ev}), we see that $\text{Span}(\partial_r) \subset \text{Im}(D_{\text{ev}})$. Therefore,  
$\text{Im}(D_{\text{ev}}) = \text{Span}(\partial_r) + \text{Im}(D_{\text{ev}}) = T_{(0,z)}\hat{Y}$, i.e., $D_{\text{ev}}$ is surjective.
\end{proof}

\subsection{Relation to abstract constraints} \label{subsec:Umap_moduli_space_vs_point_moduli_space} Thus far in \S \ref{sec:moduli_space}, we have counted points in moduli spaces of cylinders with a point constraint. To conclude this section, we relate this count to a point count in the moduli spaces involved in the $U$-maps constructed in \S \ref{subsec:constrained_cobordism_maps}. 

\vspace{3pt}

We continue to work with the objects and notation described in Setup \ref{set:intersection_theory}. We can also work with the weaker hypothesis discussed in Remark~\ref{rem:non-toric}.

\vspace{3pt}

Let $X:Y \to Y$ be the trivial cobordism $[-\delta,\delta]_s \times Y$ equipped with the Liouville form
\[
\lambda = e^{s} \alpha_f = e^{s + \epsilon f}\alpha
\]and let $E \subset X$ be an embedded, irrational ellipsoid that is the image of an embedding
\[\iota:E(b_1,\dots,b_n) \to X \qquad \qquad\text{such that}\qquad \iota(0) = 0 \times z \in [-\delta,\delta] \times Y\]
The boundary $\partial E$ has a shortest simple orbit $\kappa$. Let $J'$ be a compatible almost complex structure  on $X \setminus E$ and fix the shorthand notation
\[
\Umoduli := \mathcal{M}_{0,A}(X\setminus E,J';\gamma_{+},\Gamma_-), \qquad\text{where}\qquad \Gamma_- = (\gamma_-,\kappa).
\]
This is the moduli space of genus $0$ $J$-holomorphic curves in the cobordism $X \setminus E:Y \to Y \cup \partial E$ asymptotic to $\gamma_+$ at $+\infty$ and $\gamma_- \cup \kappa$ at $-\infty$.
 Moreover, let $J''$ be an almost complex structure on $E$ and fix the shorthand notation \[
\Emoduli := \on{ev}^{-1}(0) \subset \mathcal{M}_{0,A,1}(E,J'';\kappa,\emptyset).
\]
This is the moduli space of $J''$-holomorphic planes in $E$ that pass through $0 \in E$ and that are asymptotic at $+\infty$ to $\kappa$. The goal for the rest of the section is to prove the following result.

\begin{prop} \label{prop:moduli_correspondence} Assume that the period of $\gamma_{\pm}$ is equal to the minimal period of the Reeb flow of $\alpha$. There exists a choice of complex structures $J'$ on $X \setminus E$ and $J''$ on $E$ such that the moduli spaces
\[\Umoduli \qquad\text{and}\qquad \Emoduli\]
are transversely cut out, $0$-dimensional and equal to their compactifications (i.e. there are no buildings in their compactifications). Furthermore, they satisfy
\begin{equation} \label{eqn:moduli_correspondence} \#\moduli = \#\Umoduli \cdot \#\Emoduli \mod 2\end{equation}
\end{prop}

\begin{remark} In fact, the equality (\ref{eqn:moduli_correspondence}) should be true over $\Q$ for any choice of complex structures, as a virtual count of points in the framework of Pardon \cite{p2015,p2016}. This version of this result is proven by Siegel \cite{sie_hig_19}, using $J$-holomorphic cascades and the Morse-Bott formalism. 

\vspace{3pt}

In the spirit of the other results of this paper and for the sake of completeness, we provide a proof that uses only transversely cut out holomorphic curves and avoids Morse-Bott theory. 
\end{remark}

Our strategy to prove Proposition \ref{prop:moduli_correspondence} is quite standard. First, we choose a complex structure satisfying a number of regularity hypotheses for somewhere injective curves. Second, we show the desired compactness results for $\Umoduli$ and $\Emoduli$ for the chosen complex structures. Last, we use a parametric moduli space to construct a topological cobordism
\[
\wt{\mathcal{M}} \qquad\text{from}\quad \moduli \qquad\text{to}\quad \Umoduli \times \Emoduli.
\]
The existence of such a cobordism proves the desired result.

\vspace{3pt}

To begin the argument, we fix almost complex structures such that the relevant moduli spaces of somewhere injective curves are sufficiently regular. We are permitted to do this by standard generic transversality results (reviewed in \S \ref{subsubsec:generic_transversality}, see Proposition \ref{prop:generic_transversality}).

\begin{setup} \label{set:complex_structures_on_XE_and_E} For the rest of the section, fix compatible almost complex structures
\[J'\text{ on }X \setminus E \qquad \text{and}\qquad J''\text{ on }E\]
satisfying the following properties:
\begin{enumerate}
    \item the moduli spaces of finite energy, somewhere injective curves
    \[\mathcal{M}^{\on{i}}_{g,A,m}(X\setminus E,J';\Gamma_+,\Gamma_-) \qquad\text{and} \qquad \mathcal{M}^{\on{i}}_{h,b,l}(E,J'';\Xi_+,\Xi_-) \qquad\text{are regular},\]
\item the evaluation maps
    \[\on{ev}:\mathcal{M}^{\on{i}}_{h,b,1}(E,J'';\Xi_+,\Xi_-) \to \hat{X} \qquad \text{are transverse to $z$.}\]
\end{enumerate}
\end{setup}

Given $J'$ and $J''$ as in Setup \ref{set:complex_structures_on_XE_and_E}, there is a compactified moduli space that we denote by
\[\overline{\mathcal{M}}_{X \setminus E \sqcup E}\]
consisting of genus $0$ buildings $\bar{u}$ with the following levels.
\begin{itemize}
    \item A sequence of levels $u^+_1,\dots,u^+_k$ in $\hat{Y}$.
    \item A single level $u$ in the cobordism $\widehat{X\setminus E}$ of genus $0$ with one positive end.
    \item A sequence of levels $u^-_1,\dots,u^-_l$ in $\hat{Y}$.
    \item A sequence of levels $v_1,\dots,v_m$ in $\hat{\partial E}$.
    \item A single level $v$ in $\hat{E}$. 
\end{itemize}
where the positive ends and negative ends of adjacent levels match in the usual way. A depiction of a possible building $\bar
u$ is depicted in Figure~\ref{fig:building_in_X_minus_E}.
\begin{figure}
    \centering
    \includegraphics{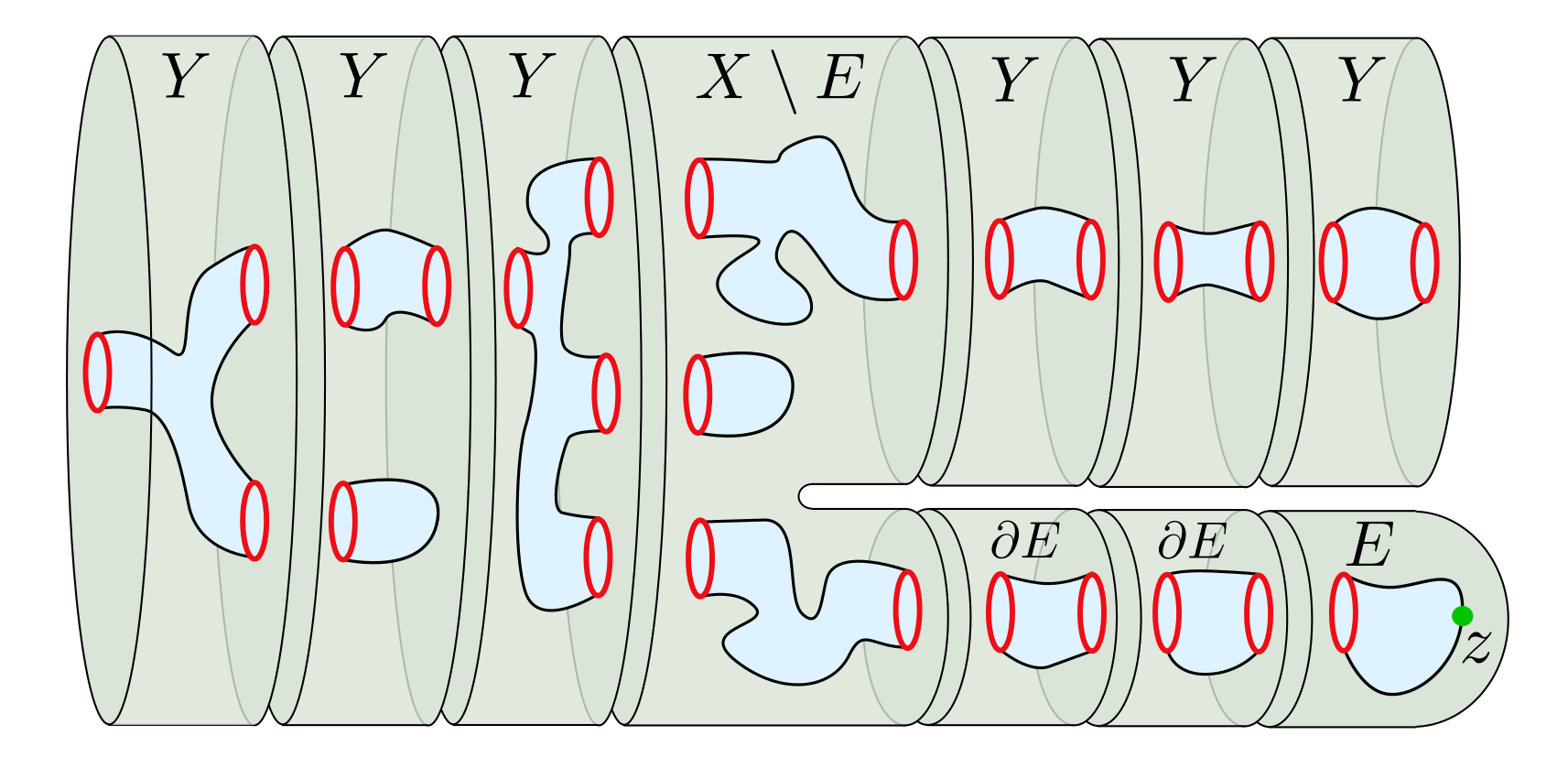}
    \caption{A possible building in the moduli space $\overline{\mathcal{M}}_{X \setminus E \sqcup E}$. We will show in Lemma \ref{lem:XE_curves_are_simple} that (generically) these buildings must be much simpler than this.}
    \label{fig:building_in_X_minus_E}
\end{figure}

\vspace{3pt}

We now show that (under our hypotheses) the only buildings appearing in $\overline{\mathcal{M}}_{X \setminus E \sqcup E}$ are those of the simplest possible form. We will require the following lemma.

\begin{lemma} \label{lem:XE_curves_are_simple} Consider $\varepsilon>0$ from Setup \ref{set:intersection_theory} and let $J'$ as in Setup \ref{set:complex_structures_on_XE_and_E}. Let $\zeta_+$ and $\zeta_-$ be orbits of $Y$ satisfying
\begin{equation} \label{eqn:XE_curve_1} \mathcal{A}(\gamma_+) \ge \mathcal{A}(\zeta_+) \ge \mathcal{A}(\zeta_-) \ge \mathcal{A}(\gamma_-)\end{equation}
Finally, let $u$ be a finite energy, connected $J'$-holomorphic map in $\widehat{X \setminus E}$ of genus 0 such that
\begin{equation}\label{eqn:XE_curve_2} u \to \zeta_+ \text{ at }+\infty \qquad\text{and}\qquad u \to \zeta_- \cup \Xi\text{ at }-\infty\end{equation}
where $\Xi$ is a non-empty orbit sequence in $Y \cup \partial E$. Then for sufficiently small $\varepsilon$, we must have
\[\zeta_+ = \gamma_+ \qquad \zeta_- = \gamma_- \quad\text{and}\quad \Xi = \kappa\]
\end{lemma}

\begin{proof} First, factor $u$ as a branched cover $\psi$ and a somewhere injective map $v$,
\[u:\Sigma' \xrightarrow{\psi} \Sigma \xrightarrow{v} \hat{Y}.\]
Here $v$ has a positive and negative ends asymptotic to orbits
\[\beta_\pm \text{ covered by }\zeta_\pm \qquad\text{and}\qquad \Xi' \text{ covered by }\Xi.\]
It suffices to show that $\beta_\pm = \gamma_\pm$ and $\Xi' = \kappa$. Note that this will imply that $u = v$. 

\vspace{3pt}

First, note that $v$ is  genus $0$. Indeed, we can compactify $\psi$ to a branched cover $S^2 \to \bar{\Sigma}$ where $\bar{\Sigma}$ is a closed surface given by compactifying $\Sigma$ along the punctures. Any closed surface with a non-trivial branched cover from $S^2$ is genus $0$.

\vspace{3pt}

Next, note that $\gamma_+$ and $\gamma_-$ are orbits corresponding to perturbations of Morse-Bott orbits of period equal to the period $T$ of the Reeb flow of $\alpha_f$. Thus, using the action bounds (\ref{eqn:XE_curve_1}), we have
\[
\mathcal{A}(\beta_+) - \mathcal{A}(\beta_-) =\frac{\mathcal{A}(\zeta_+) - \mathcal{A}(\zeta_-)}{\on{deg}(\psi)}  \le \frac{\mathcal{A}(\gamma_+) - \mathcal{A}(\gamma_-)}{\on{deg}(\psi)} = O(\varepsilon).
\]
Thus, the orbit set $\Xi'$ cannot contain any orbits in $Y$ for small $\varepsilon$, otherwise  the action of the orbit set $\beta_- \cup \Xi'$ would be larger than that of $\beta_+$. Moreover, any pair of orbits of $\alpha_f$ satisfying
\[\mathcal{A}(\beta_+) - \mathcal{A}(\beta_-) \le O(\varepsilon) \qquad\text{and}\qquad \mathcal{A}(\beta_-) \le \mathcal{A}(\beta_+) \le \mathcal{A}(\gamma_+) \simeq T + O(\epsilon)\]
must be perturbations of Morse-Bott orbits in the same Morse-Bott family $S_L$ for sufficiently small $\varepsilon$. In particular, $\beta_\pm$ correspond to pairs $(S_L,p_\pm)$, where $L \le T$ (where $T$ is the period of the flow of $E$) and $p_\pm$ are critical circles of $f$. By Lemma~\ref{lem:CZ_via_RS_n_Morse}, the Conley-Zehnder indices $\CZ(\beta_\pm,\tau_\pm)$ in any trivialization $\tau_\pm$ of $\xi$ over $\beta_\pm$ is given by
\[
\CZ(\beta_\pm,\tau_\pm) = \on{RS}(\beta_\pm,\tau_\pm) - \dim{S_L}/2 + \on{ind}_{Morse}^{S_L}(f;p_\pm)\]
where $\on{RS}(\beta_\pm,\tau_\pm)$ is the Robbin-Salamon index of the unperturbed orbit and $\on{ind}_{Morse}^{S_L}(f;p_\pm)$ denotes the Morse-Bott index of the critical circle $p_\pm$ in the tangent directions to $S_L$. In particular, let $\Sigma$ be a cylinder connecting $\beta_+$ and $\beta_-$ formed by a $1$-parameter family of orbits of period $L$ connecting $\beta_+$ and $\beta_-$, and let $\tau$ be a trivialization of $\xi$ over $\Sigma$. Then
\[
\CZ(\beta_+,\tau) - \CZ(\beta_-,\tau) \le 2n - 2,
\]
Since $H_2(Y;\Q) = 0$ by hypothesis, there is a well-defined difference of SFT grading between $\beta_+$ and $\beta_-$, and in those terms the above inequality states that
\[|\beta_+| - |\beta_-| \le 2n - 2\]
Equality can occur if and only if $\on{dim}(S_L) = 2n - 2$, $p_+$ is the maximum circle of $f$ and $p_-$ is the minimum of $f$. Since $L \le T$ and $T$ is the period of the flow of $\alpha_f$, this  only if $S_L=S_T$, and $\beta_\pm = \gamma_\pm$.

\vspace{3pt}

Finally, note that the Fredholm index of the holomorphic curve $v$ is given in terms of relative SFT gradings by
\[
\on{ind}(v) = |\beta_+| - |\beta_-| - |\Xi'| \le 2n - 2 - |\Xi'|.
\]
Here $\tau$ is a trivialization of $\xi$ over $v$ . By the calculation in Example \ref{ex:irrational_ellipsoid_orbits}, we know that
\[|\Xi'| = n - 3 + \sum_{\gamma \in \Gamma} \CZ(\gamma) \ge 2n - 2,\]
since every orbit $\gamma$ on the boundary of an ellipsoid had $CZ(\gamma)\geq n+1$. Moreover, equality holds if and only if $\Xi'$ is the length $1$ sequence of the minimum action orbit $\kappa$. Overall we conclude that
\[\on{ind}(v) \le 0\]
with equality if and only if
\[u = v, \qquad \beta_+ = \gamma_+,\qquad \beta_- = \gamma_- \qquad \text{ and }\qquad \Gamma = \kappa.\]
By hypothesis, every somewhere injective curve is non-negative index. Therefore $\on{ind}(v)=0$ and thus it coincides with $u$, and their ends coincide with $\gamma_+$ and $\gamma_-\cup \kappa$ as required.
\end{proof}

\begin{remark} The reasoning in Lemma \ref{lem:XE_curves_are_simple} relies on Setup \ref{set:intersection_theory}(vi). In particular, we use the fact that $\gamma_+$ and $\gamma_-$ are in the \emph{minimal period} Morse-Bott family of orbits of dimension $2n-2$. 
\end{remark}

\begin{lemma} \label{lem:J_and_Jprime_moduli} Let $\varepsilon > 0$ in Setup \ref{set:intersection_theory} be sufficiently small and choose $J',J''$ as in Setup \ref{set:complex_structures_on_XE_and_E}. Then
\begin{enumerate}
    \item The moduli spaces $\Umoduli$ and $\Emoduli$ are compact, regular and $0$-dimensional.
    \item We have an equality of moduli spaces
    \[\overline{\mathcal{M}}_{X \setminus E \sqcup E} = \Umoduli \times \Emoduli\]
\end{enumerate}
\end{lemma}

\begin{proof} The regularity and dimension of both moduli spaces follows from our choice of $J'$ and $J''$. We now argue for the compactness of $\Umoduli$ and $\Emoduli$, and the equality (ii).

\vspace{3pt}

First, to show that $\Umoduli$ is compact, consider an arbitrary building
\[
\bar{u} \in \overline{\mathcal{M}}(\gamma_+,\gamma_- \cup \kappa;J') =\overline{\mathcal{M}}_{0,A}(X \setminus E,J';\gamma_+,\gamma_-)
\]
Denote by $u$ the level of $\bar u$ in $X \setminus E$. Then, $u$ is genus $0$ and  has one positive end asymptotic to an orbit $\zeta_+$ with $\mathcal{A}(\zeta_+) \le \mathcal{A}(\gamma_+)$, a negative end $\zeta_-$ satisfying $\mathcal{A}(\zeta_-) \ge \mathcal{A}(\gamma_-)$ and some other negative ends $\Xi$. Thus by Lemma \ref{lem:XE_curves_are_simple}, we have
\[\zeta_+ = \gamma_+,\qquad \zeta_- = \gamma_- ,\qquad\text{and}\qquad \Xi = \kappa.\]
This implies that $\bar{u} = u$ and so $\Umoduli$ is compact.

\vspace{3pt}

Next, to show that $\Emoduli$ is compact, let $\bar{v}$ be a building in the compactification $\overline{\mathcal{M}}(\kappa;z,J'')$, then there is a bottom level
\[v:\Sigma \to \hat{E} \qquad\text{with}\qquad z \in \on{Im}(v)\]
Since the whole building is genus $0$, so is $v$, and $v$ has no negative ends since $E$ is a cobordism to the empty set. Moreover, the action of the positive ends $\Xi_+$ must be less than $\mathcal{A}(\kappa)$. Since $\kappa$ has the minimum action over all orbit sets, we must have $\Xi_+ = \kappa$, so
\[v \in \Emoduli \qquad\text{and}\qquad v = \bar{v}\]

\vspace{3pt}

Finally, to show (ii), let $\bar{u}$ be a building in the compactified moduli space $\overline{\mathcal{M}}_{X \setminus E \sqcup E}$. Then the level $u$ in $X \setminus E$ satisfying the hypotheses of Lemma \ref{lem:XE_curves_are_simple}, and so we have
\[u \in \Umoduli\]
The level $v$ of $\bar{u}$ in $E$ must thus have positive ends of action bounded by the actions of the negative end $\kappa$ of $u$. This is only possible if
\[v \in \Emoduli\]
The only remaining levels in $\bar{u}$ must be trivial for action reasons, so $\bar{u}$ consists only of the levels $u$ and $v$. This proves the result. \end{proof}

Given a choice of $J'$ and $J''$, there is a family $J_{[1,\infty)}$ of almost complex structures on $X$,
\[J' \#_R J'' \qquad\text{for} \qquad R \in [1,\infty).\]
This family is acquired by identifying $X$ with the space
\[X \simeq X \setminus E \cup_{\partial E} [0,R] \times \partial E \cup_{\partial E} E\]
and setting $J' \#_R J''$ be equal to $J'$ on $X \setminus E$, $J''$ on $E$ and letting $J' \#_R J''$ be translation invariant on $[0,R] \times \partial E$ in the $[0,R]$-direction. This family has associated parametric moduli space of genus $0$ curves from $\gamma_+$ to $\gamma_-$, i.e. the space
\[
\mathcal{M}_{0,A,1}(X,J_{[1,\infty)};\gamma_+,\gamma_-).
\]
This map has an evaluation map $\on{ev}$ to $\hat{X}$, and we now need to consider the parametric moduli space of cylinders passing through $z$. We adopt the shorthand notation
\[
\widetilde{\mathcal{M}}_{[1,\infty)} := \on{ev}^{-1}(z) \subset \mathcal{M}_{0,A,1}(X,J_{[1,\infty)};\gamma_+,\gamma_-).
\]

See \S \ref{subsubsec:parametric_moduli_spaces} for a discussion of parametric moduli spaces. By the appropriate version of SFT compactness (cf. \cite[Prop. 10.6]{sftCompactness}), given a sequence of elements
\[(R_i,u_i) \in \widetilde{\mathcal{M}}_{[1,\infty)} \qquad\text{with}\qquad R_i \to \infty\]
there is a limiting building
\[
\bar{u} \in \overline{\mathcal{M}}_{X\setminus E \sqcup E} \overset{Lemma~\ref{lem:J_and_Jprime_moduli}}{=} \Umoduli \times \Emoduli
\]
that $u_i$ converges to with respect to the BEHWZ topology in \cite{sftCompactness}. Since $\Umoduli$ and $\Emoduli$ consist entirely of regular curves (by Setup \ref{set:complex_structures_on_XE_and_E}), there is a gluing map
\[
\on{glue}:\Umoduli \times \Emoduli \times (R,\infty) \to \widetilde{\mathcal{M}}_{[1,\infty)}
\]
that is a homeomorphism for sufficiently large $R$. The construction of this gluing map is carried out, for example, in \cite[\S 5]{p2015}\footnote{Note that the gluing result in \cite{p2015} is much more general than the one required here. In particular, we only require the gluing map for a specific stratum of the moduli space denoted $\mathcal{M}_{\on{IV},\on{reg}}$ in  \cite{p2015}.}.

 \begin{lemma} \label{lem:Jparametric_moduli} There exists a $[0,\infty)$-family of compatible almost complex structures on $X$, denoted
 \[J_{[0,\infty)} = \{J_R\}_{R \in [0,\infty)}\]
 that has the following properties.
 \begin{enumerate}
     \item $J_R = J_f$ for $R$ near $0$ and $J_R = J' \#_R J$ for $R$ sufficiently large.
     \item The parametric moduli space
     \[\widetilde{\mathcal{M}}_{[0,\infty)} := \on{ev}^{-1}(z) \subset \mathcal{M}_{0,A,1}(X,J_{[0,\infty)};\gamma_+,\gamma_-)\]
     is a $1$-manifold with boundary $\moduli$.
     \item The natural projection map to the parameter space
     \[
     \pi:\widetilde{\mathcal{M}}_{[0,\infty)} \to [0,\infty)
     \]
     is proper.
 \end{enumerate}
 \end{lemma}
 
 \begin{proof} Let $\{J_R\}_{R \in [0,\infty)}$ be a $1$-parameter family of compatible almost complex structures on $X$ satisfying (i). We use $\mathcal{M}_R$ to fiber of $\pi$ at $R$, i.e.
 \[
 \mathcal{M}_R := \{u\; : \; (R,u) \in \widetilde{\mathcal{M}}_{[0,\infty)}\}
 \]
 We also let $\mathcal{M}_R^{\on{i}}$ denote the subset of somewhere injective curves in $\mathcal{M}_R$ and $\overline{\mathcal{M}}_R$ denote the BEHWZ compactification. All of the holomorphic curves in the moduli spaces
 \[
\moduli \qquad\text{and}\qquad \mathcal{M}_R \text{ for large }R >  0 
 \]
 consist entirely of regular curves (and no buildings) by construction of $J'$ and $J''$, and Proposition \ref{prop:moduli_space_count}. As regularity is an open condition, we automatically know that $\mathcal{M}_R$ is compact and regular (in the usual sense, not parametrically) outside of $[a,b] \subset [0,\infty)$ for some $0<a<b<\infty$. Using Proposition \ref{prop:generic_transversality}, we perturb $\{J_R\}_{R \in [0,\infty)}$ on $(a-\delta,b+\delta)$ for small $\delta$ so that the moduli spaces
 \[\mathcal{M}^i_{g,A,m}(X,J_{[a,b]};\Gamma_+,\Gamma_-)\]
 are parametrically regular and the natural evaluation maps
 \[
 \on{ev}:\mathcal{M}^i_{g,A,1}(X,J_{[a,b]};\Gamma_+,\Gamma_-) \to \hat{X}
 \]
 are transverse to $z$. We now claim that under these hypotheses, we have
 \begin{equation} \label{eqn:parametric_moduli_space_claim}
 \mathcal{M}^{\on{i}}_R = \mathcal{M}_R = \overline{\mathcal{M}}_R \qquad\text{for any }R \in [a,b].
 \end{equation}
This implies that the space $\widetilde{\mathcal{M}}_{[0,\infty)}$ is a compact $1$-manifold, implying (ii) and (iii).

\vspace{3pt}

To prove (\ref{eqn:parametric_moduli_space_claim}), we consider a $J_R$-holomorphic building
\[\bar{u} \in \overline{\mathcal{M}}_R.\]
Let $u$ be the level of $\bar{u}$ that contains $z$ and let $v$ be the underlying simple curve. By identical reasoning to Lemma \ref{lem:XE_curves_are_simple}, $v$ must be a holomorphic cylinder in $\hat{X}$ of area $O(\varepsilon)$. For small $\varepsilon$, $v$ must therefore must be asymptotic to closed orbits $\beta_\pm$ that are Morse perturbations of orbits in the same Morse-Bott family $S_L$. In particular, the Fredholm index of $v$ satisfies
\[
\on{ind}(v) \le |\beta_+| - |\beta_-| \le 2n - 2,
\]
with equality if and only if
\begin{equation} \label{eqn:parametric_moduli_space_claim_2} v = u ,\qquad \beta_+ = \gamma_+ \quad\text{and}\quad \beta_- = \gamma_-.
\end{equation} 
Here $|\beta_+| - |\beta_-|$ again denotes the relative grading of $\beta_+$ and $\beta_-$.

\vspace{3pt}

Thus consider the case where $\on{ind}(v) < 2n - 2$. Then the parametric moduli space containing $v$ with a marked point added,
\[
\mathcal{M}^{\on{i}}_{0,A,1}(X,J_{[a,b]};\beta_+,\beta_-),
\]
is a manifold of dimension less than $2n$. By the Sard-Smale theorem, the image of the (smooth) evaluation map
\[
\on{ev}:\mathcal{M}^{\on{i}}_{0,A,1}(X,J_{[a,b]};\beta_+,\beta_-) \to \hat{X}
\]
is of the first category (i.e. a countable union of nowhere dense sets) and has dense complement. In particular, there are points arbitrarily close to $z$ that are not in the image of $\on{ev}$. Therefore, after pulling back $J_{[a,b]}$ by a $[a,b]$-family of small symplectomorphisms supported near $z$, we can  assume that
\[
z \not\in \on{ev}\big(\mathcal{M}^{\on{i}}_{0,A,1}(X,J_{[a,b]};\beta_+,\beta_-)\big)
\]
for any $\beta_+,\beta_-$ satisfying the hypotheses above. In particular, after modifying $J_{[a,b]}$, we may assume that $v$ does not contain $z$ unless $\on{ind}(v) = 2n - 2$, which implies (\ref{eqn:parametric_moduli_space_claim_2}). This concludes the proof. \end{proof}
 
We can now conclude this section with a proof of Proposition \ref{prop:moduli_correspondence}.

\begin{proof} The compactness, transversality and dimension of the moduli spaces
\[\Umoduli \qquad\text{and}\qquad \Emoduli\]
are demonstrated in Lemma \ref{lem:J_and_Jprime_moduli}. To prove (\ref{eqn:moduli_correspondence}), we simply note that for a choice of $J_{[0,\infty)}$ as in Lemma \ref{lem:Jparametric_moduli}, the parametric moduli space
\[
\widetilde{\mathcal{M}} := \pi^{-1}([0,R]) \subset \widetilde{M}_{[0,\infty)y}
\]
is a cobordism from $\moduli$ to $\Umoduli \times \Emoduli$ as long as $a > 0$ is close to $0$ and $b$ is sufficiently large. 
\end{proof}

\section{Vanishing spectral gap for ellipsoids} \label{sec:proof_of_main_thm}

Our main goal for this section is to prove  Theorem~\ref{thm:spectral_gap_intro}, which states that the strong closing property holds for all ellipsoids. We will show that the spectral gap vanishes for ellipsoids and therefore, the strong closing property will follow from Theorem~\ref{thm:gap_0_closing_property}.

\vspace{3pt}

Following the strategy discussed in \S \ref{subsec:main_results}, we proceed in two steps. First, in Section~\ref{subsec:rational_ellipsoids} we consider ellipsoids 
$
E(m)=E(m_1,\dots, m_n)
$
where $m_j\in \N$. The Reeb flow is periodic on the boundary $\partial E(m)$, so we may apply the results of Section~\ref{sec:moduli_space}. Specifically, we use Proposition~\ref{prop:moduli_space_count} and Proposition~\ref{prop:moduli_correspondence} to show that certain coefficients of the map $U_{P_0}$, defined in Example~\ref{exa:tangency_constraints}, are non-zero and deduce that the spectral gap of a certain class vanishes for such ellipsoids. Second, in Section \ref{subsec:irrational_ellipsoids}, we use the vanishing of spectral gaps for integer  ellipsoids to show that the total spectral gap vanishes for irrational ellipsoids as well.

\begin{notation}
We fix the following notation for the rest of the section. Given an $n$-tuple of positive numbers ${a}:=(a_1,\dots, a_n)$, consider the ellipsoid $E({a})$. 
\begin{itemize}
    \item We denote by $\{M_k\}_{k\in \N}$, or $\{M_k^a\}_{k\in \N}$, the sequence obtained by reordering the union $a_1\cdot \N\cup\cdots\cup a_n\cdot \N$ to be a non-decreasing sequence with repetitions.
    \vspace{3pt}
    \item We let $\gamma_k$, or $\gamma_k^a$ denote the $k$-th periodic Reeb orbit ordered by action. Here we introduce Morse-Bott perturbations if the Reeb flow on $\partial E({a})$ is degenerate. 
    \vspace{3pt}
    \item We let $x_k$ denote the generator of $CH(\partial E(a))$ corresponding to $\gamma_k$ (see Example \ref{ex:irrational_ellipsoid_orbits}). Note that
    \[|x_k| = 2n - 2 + 2k \qquad\text{and}\qquad \mathcal{A}(\gamma_k) = M^a_k\]
    \end{itemize} 
\end{notation}

\subsection{Rational Ellipsoids} \label{subsec:rational_ellipsoids}

Consider an integer ellipsoid  $E(m):=E(m_1,\dots,m_n)$ where $m_j\in\N$ for $j=1,\dots, n$. Denote the least common multiple of $m_1, \dots, m_n$ by
\[T := \on{lcm}(m_1,\dots,m_n)\]
The Reeb flow on $\partial E({m})$ is periodic of period $T$.
Since $m_j$ divides $T$ for all $j=1,\dots, n$, the common multiple $T$ appears in the sequence $\{M_k^m\}_k$ exactly $n$ times. We denote the first index in which $T$ appears in $M^m_k$ by 
\[
k_T:=\min{k\in \N: M_k^m=T}.
\]
The main goal of this subsection is to prove the following result.

\begin{thm}\label{thm:U_integer_E}
Let $E({m})$ be an integer ellipsoid as above, and let $T:=\lcm(m_1,\dots,m_n)$. Then
\[
\left< U_{P_0} x_{k_T + n - 1},x_{k_T}\right>\neq 0.
\]
Here $P_0$ is the tangency abstract constraint from Example \ref{exa:tangency_constraints}.
\end{thm}

\noindent As mentioned above, the least common multiple $T$ of $m_1,\dots,m_n$ occurs with multiplicity $n$ in the sequence $M_k^m$. Therefore, the action  $M_{k_T + n-1}$ of $x_{k_T + n - 1}$ and the action $M_{k_T}$ of $x_{k_T}$ are both equal to $T$. This implies the following corollary of Theorem \ref{thm:U_integer_E}.

\begin{cor} \label{cor:periodic_spectral_gap} (Theorem \ref{thm:periodic_spectral_gap_intro}) Let $E(m)$ be an integer ellipsoid. Then
\[
\mathfrak{s}_{U\sigma}(\partial E,\lambda|_{\partial E}) = \mathfrak{s}_{\sigma}(\partial E,\lambda|_{\partial E}) = T \qquad\text{where}\qquad U = U_{P_0}\text{ and }\sigma = x_{k_T + n - 1}.
\]
\end{cor}

\begin{remark}[Generalizations]\label{rmk:from_integer_E_to_rational} Theorem \ref{thm:U_integer_E} and Corollary \ref{cor:periodic_spectral_gap} can be generalized in two directions.
\begin{enumerate}
    \item If $E(a)$ is a rational ellipsoid, i.e. an ellipsoid where
\[a_i/a_j\in\Q \text{ for any }i,j, \qquad\text{or equivalently}\qquad c \cdot a = m \in \Z^n \text{ for some }c > 0.\]
In this case, the contact form on $\partial E(a)$ is simply a scaling of the one on $\partial E(m)$. Moreover, there is an equivalence of filtered groups
\[CH^L(\partial E(a)) = CH^{c \cdot L}(\partial E(c \cdot a))\]
that commutes with all (constrained) cobordism maps. Thus Theorem \ref{thm:U_integer_E} and Corollary \ref{cor:periodic_spectral_gap} generalize immediately to this case.

\vspace{3pt}

\item Any positive integer multiple $q$ of the period $T$ similarly corresponds to $n$ classes
\[x_{k_q},\dots,x_{k_q + n - 1} \in CH(\partial E(m))\]
of the same action. Our proof of Theorem \ref{thm:U_integer_E} almost generalizes to show that
\[
\left< U_{P_0} x_{k_T + n - 1},x_{k_T}\right>\neq 0.
\]
However, the proof uses Proposition \ref{prop:moduli_correspondence}, which we only show for the case of $q = T$. We believe that the discussion in \cite[\S 5.5]{sie_hig_19}, rigorously carried out using the VFC methods of \cite{p2015}, should suffice to generalize the moduli space correspondence in Proposition \ref{prop:moduli_correspondence}, but we do not carry out this analysis here.
\end{enumerate}
\end{remark}

\begin{proof}[Proof of Theorem~\ref{thm:U_integer_E}]
The result follows from combining Proposition~\ref{prop:moduli_space_count} and Proposition~\ref{prop:moduli_correspondence}, which together assert that the moduli space count corresponding to the coefficient of $x_{k_T}$ in $U_{P_0} x_{k_T+n-1}$ is non-zero. 

\vspace{3pt}

\emph{Step 1 - Setup.} We start by showing that our setting satisfies the assumptions of Propositions~\ref{prop:moduli_space_count} and \ref{prop:moduli_correspondence}, which are stated in Setup~\ref{set:intersection_theory}. Let $(Y,\alpha)$ be the boundary of the integer ellipsoid $E:=E(m)$ with its standard contact form. The Reeb flow associated with this contact form is given by
\[(z_1,\dots, z_n)\mapsto (e^{2\pi i t/m_1}z_1,\dots, e^{2\pi i t/m_n}z_n)\]
and has period $T$. Consider the function $f:Y\rightarrow \R$ given by the restriction to $\partial E$ of the map 
\begin{equation}\label{eq:f_on_Cn}
   \tilde f :\C^n\rightarrow\C^n,\quad  z=(z_1,\dots, z_n)\mapsto \frac{\sum_{j=1}^n (j-1)|z_j|^2} {\sum_{j=1}^n |z_j|^2}.
\end{equation}
The function $f$ is invariant under the $\mathbb{T}^n$-action on $\C^n$ and thus is invariant under the Reeb flow. Its critical circles are precisely the intersections of $\partial E$ with the complex axes
$$
\gamma_i:=\partial E\cap \C_i = \left\{(0,\dots, \sqrt{\frac{m_i}{\pi}}e^{2\pi i t/m_i},0, \dots, 0): t\in [0,m_i]\right\}.
$$
The function $f$ is Morse--Bott. Indeed, the restriction of its Hessian to the normal to $\gamma_i$ is non-degenerate, can be seen as follows. Along $\gamma_i$ the Hessian of the map $\tilde f$ is given by
\begin{equation}\label{eq:Hessian_diag_ellipsoid}
    \on{Hess}(\tilde f)|_{\gamma_i} = \frac{\pi}{m_i}\cdot \on{diag}(1-i, 1-i, 2-i, 2-i,\dots, n-i,n-i),
\end{equation}
which is degenerate only on the $i$-th $\C$ factor. The tangent space to $\partial E$ along $\gamma_i$ is given by 
\[
T\partial E|_{\gamma_i} = \C^{i-1}\oplus\left< R|_{\gamma_i} \right>\oplus \C^{n-i}, 
\]
where $R$ is the Reeb vector field. Therefore, the Hessian of $f$ along $\gamma_i$ is degenerate only in the Reeb direction which lies in the tangent space of $\gamma_i$, so $f$ is Morse-Bott with index
\[\ind(f, \gamma_i)=2(i-1) \qquad\text{at the critical circle }\gamma_i.\]
The orbits $\gamma_+$ and $\gamma_-$ are the period $T$ iterate of $\gamma_1$ and $\gamma_n$, i.e.
\[\gamma_+ := \gamma_1^{T/m_1} \qquad \text{and}\qquad \gamma_- := \gamma_n^{T/m_n}. \]
The contact flag is given by the following sequence  $Y_3\subset \cdots \subset Y_{2n-1}=Y$ of nested ellipsoids:
\[
Y_{2j-1}:=\partial E \cap \left(\C^{j-1}\times \{0\}^{n-j}\times \C \right), \qquad j=2,\dots, n.
\]
Since $Y_{2j-1} \simeq \partial E(a_1,\dots,a_{j-1},a_n)\subset \C^j$, we have $H_2(Y_{2j-1})=0$. Moreover, $Y_{2j-1}$ is invariant under the periodic Reeb flow and the gradient flow of $f$. The intersection $\xi_{2j-1}:=\xi\cap TY_{2j-1}$ is a $J$-invariant contact structure on $Y_{2j-1}$. The formula (\ref{eq:Hessian_diag_ellipsoid}) implies that the restriction of $f$ to $Y_{2j-1}$ is also Morse--Bott. In addition, for $i$ and $j$ such that $\gamma_i\subset Y_{2j-1}$ (that is, either $i=n$ or $i<j$), the restriction of $\on{Hess}(f)|_{\gamma_i}$ to the symplectic orthogonal of $\xi_{2j-1}$ in $\xi_{2j+1}$ is equal to $\on{diag}(j-i, j-i)$, which is positive definite unless $i=n$. 

\vspace{3pt}

Finally, take $z$ be any point in $Y_3\setminus\{\gamma_-\cup\gamma_+\}$. Since $\gamma_+$ and $\gamma_-$ are the only critical circles in $Y_3$, $z$ lies in the intersection of the stable manifold of $\gamma_-$ and the unstable manifold of $\gamma_+$ which concludes the assumptions stated in Setup~\ref{set:intersection_theory}. 

\vspace{3pt}

\emph{Step 2 - Applying \S \ref{sec:moduli_space}.} Now we apply the results of \S \ref{sec:moduli_space} to show that the coefficient $\left<U_{P_0}x_{k_T+n-1}, x_{k_T}\right>$ does not vanish. First, notice that the contact form $\alpha_f$ on $Y$ has generators only in even gradings. Indeed, this follows from Lemma~\ref{lem:RS_parity} together with the fact that the Morse indices of $f$ on any Morse--Bott family are even. This is due to the fact that the Morse--Bott families are products of some of the $\C$ factor, and the Hessian of $f$ is a scalar matrix on each $\C$ factor. We conclude that the differential vanishes and $CH(S^{2n-1},\xi_{std}) \cong A(Y,\alpha_f)$.
Consider the trivial exact cobordism
\[
X = [-\delta,\delta] \times Y:(Y,e^\delta\alpha_f) \to (Y,e^{-\delta}\alpha_f).
\]
We let $W \subset X$ be an embedded irrational ellipsoid with minimal Reeb orbit $\kappa$. By Proposition~\ref{prop:moduli_space_count} and \ref{prop:moduli_correspondence}, for any $\varepsilon > 0$ sufficiently small as in Setup \ref{set:intersection_theory}, we can choose a compatible complex structure $J$ on $X \setminus W$ and $\delta > 0$ (as in \S \ref{subsec:Umap_moduli_space_vs_point_moduli_space})  such that the moduli space
\[\overline{\mathcal{M}}_{0,A}(\gamma_+,\gamma_- \cup \kappa) = \mathcal{M}_{0,A}(\gamma_+,\gamma_- \cup \kappa)\]
is regular, contains no buildings and has point count 
\[\#\overline{\mathcal{M}}_{0,A}(\gamma_+,\gamma_- \cup \kappa) = \#\overline{\mathcal{M}}(Y,J_f;\gamma_+,\gamma_-;z) = 1\mod 2.\]
Here $A \in S(X \setminus W;\gamma_+,\gamma_- \cup \kappa)$ is the unique homology class from $\gamma_+$ to $\gamma_- \cup \kappa$ in $X \setminus W$. On the otherhand, the abstract constraint $P_0$ is the dual constraint (see Example \ref{ex:dual_constraint}) of the generator $\kappa$ of $CH(W)$. Thus by Lemma \ref{lem:constrained_map_curve_count}, we know that
\[
\langle U_{P_0}x_{k_T + n - 1},x_{k_T}\rangle = 1 \mod 2.
\]
In particular, $\langle U_{P_0}x_{k_T + n - 1},x_{k_T}\rangle $ is non-zero. \end{proof}

\subsection{Irrational Ellipsoids}
\label{subsec:irrational_ellipsoids}
We now use the non-vanishing of coefficients of $U_{P_0}$ for integer ellipsoids, established in Theorem~\ref{thm:U_integer_E}, to conclude that the spectral gap vanishes for all ellipsoids. 
\begin{thm}\label{thm:all_ellipsoids}
    For any ellipsoid $E({a}):= E(a_1,\dots,a_n)$, we have $\rsftgap(\partial E(a))=0$.
\end{thm}

We begin by giving a very simple proof using the rational ellipsoid case (Corollary \ref{cor:periodic_spectral_gap} and Remark \ref{rmk:from_integer_E_to_rational}) and Proposition \ref{prop:limit_of_gap0}. We start with the following approximation property for ellipsoids. This uses Dirichlet's approximation theorem.

\begin{lemma}\label{lem:approx_ellipsoid}
    Let $E(a)$ be any ellipsoid and fix $\varepsilon>0$. There exists an ellipsoid $E(r)=E(r_1,\dots,r_n)$ with rationally dependent entries such that
        \begin{equation} \label{eqn:approx_ellipsoid} E(r) \subseteq E(a)\subset \sqrt{1+\frac{\varepsilon}{T}}\cdot E(r)\end{equation}
Here $T$ denotes the period of the Reeb flow on $\partial E(r)$.
\end{lemma}
\begin{proof} We apply the simultaneous version of Dirichlet's approximation theorem to the sequence
\[\frac{1}{a_1},\dots ,\frac{1}{a_n}.\]
This result states that, for any natural number $N$, there exist integers $q_1,\dots, q_n$, and $T'$ such that $T'\leq N$ and for all $i=1,\dots ,n$, 
\[
\left|\frac{1}{a_i}-\frac{q_i}{T'}\right|\leq \frac{1}{T'\cdot N^{1/n}}, \quad\text{or equivalently,}\quad |a_iq_i - T'|\leq \frac{a_i}{N^{1/n}}.
\]
For $\varepsilon>0$ fixed as in the lemma statement, choose $N$ so that
\[N>\left(2\frac{\max{a_1,\dots,a_n}
}{\varepsilon}\right)^n\]
Under this hypothesis on $N$, we have the following bound for each $i$,
\begin{equation}\label{eq:rational_approx_bnd}
   |a_iq_i - T'|<{\varepsilon}/2. 
\end{equation}
Denote  $T:=T'-\varepsilon/2$ and $r_i:=\frac{T}{q_i}$, and consider 
\[
E({r}):=E(r_1,\dots,r_n)=E\left(\frac{T}{q_1},\dots,\frac{T}{q_n}\right).
\]
Clearly, the entries of $E({r})$ are rationally dependent. To check that $E({r})\subset E({a})$, we note that
\[a_i\overset{(\ref{eq:rational_approx_bnd})}{>}\frac{T'}{q_i}-\frac{\varepsilon}{2q_i}=\frac{T'-\varepsilon/2}{q_i}= r_i.\]
To check the second inclusion in (\ref{eqn:approx_ellipsoid}), we note that
     \[
     a_i\overset{(\ref{eq:rational_approx_bnd})}{<}\frac{T'}{q_i}+\frac{\varepsilon}{2q_i}=\frac{T'+\varepsilon/2}{q_i}=\frac{T+\varepsilon}{q_i} = (1+\frac{\varepsilon}{T})\cdot \frac{T}{q_i} =  (1+\frac{\varepsilon}{T})\cdot r_i.
     \]
     Therefore, we find that
     \[E({a})\subset \sqrt{1+\frac{\varepsilon}{T}}\cdot E({r}) = E\left((1+\frac{\varepsilon}{T})\cdot r\right)\]
     
     Note that $T$ as above is a positive integer satisfying $T/r_i = q_i$ for all $i=1,\dots,n$, but it is not necessarily the smallest positive integer with this property (i.e. the period of the Reeb flow on $\partial E(r)$). However, for any $T_0<T$ we have
     \[\sqrt{1+\frac{\varepsilon}{T}}\cdot E({r})\subset \sqrt{1+\frac{\varepsilon}{ T_0}}\cdot E({r})\]
     and so the assertion of the lemma still holds.
 \end{proof}

Having established a good enough approximation of a general ellipsoid by a rational one, we are ready to prove that the spectral gap vanishes for all ellipsoids.

\begin{proof}[Proof of Theorem~\ref{thm:all_ellipsoids}]
Let $E(a)$ be any ellipsoid. By Lemma~\ref{lem:approx_ellipsoid}, for each $\varepsilon > 0$ there exists an ellipsoid $E(r(\varepsilon))$ with rationally dependent entries $r_i(\varepsilon)$ such that 
\[E(r(\varepsilon)) \subseteq E(a)\subset \sqrt{1+\frac{\varepsilon}{T(\varepsilon)}}\cdot E(r(\varepsilon))\]
Here $T(\varepsilon)$ is the period of the Reeb flow on $\partial E(r(\varepsilon))$. This implies that the contact forms satisfy
\[ \lambda_{\on{std}}|_{\partial E(r(\epsilon))} \leq \lambda_{\on{std}}|_{\partial E(a)}  \leq (1+\frac{\varepsilon}{T})\cdot \lambda_{\on{std}}|_{\partial E(r)}.
\]
Here the contact forms on $\partial E(r(\epsilon))$ and $\partial E(a)$ are identified with contact forms on the sphere, and the order on contact forms is as in Notation \ref{not:order_on_ctct_forms}. By Corollary \ref{cor:periodic_spectral_gap} and Remark \ref{rmk:from_integer_E_to_rational}, there is a class
\[\sigma(\varepsilon) \in CH(\partial E(r(\varepsilon)) \qquad\text{with}\qquad \mathfrak{s}_{U_{P_0}\sigma(\varepsilon)}(\partial E(r(\varepsilon))) =\mathfrak{s}_{\sigma(\varepsilon)}(\partial E(r(\varepsilon))) = T(\varepsilon)\]
In particular, the spectral gap of $\partial E(r(\varepsilon))$ satisfies
\[\gap_{\sigma(\varepsilon)}(\partial E(r(\varepsilon)))=0\]
The result is now an immediate consequence of Proposition~\ref{prop:limit_of_gap0}, which asserts that 
\[
\gap(E(a))\leq \lim_{\varepsilon\rightarrow 0 }\gap_{\sigma(\varepsilon)}(E(r(\varepsilon))) =0 \qedhere
\]
\end{proof}

\subsection{Structure Of The $U$-Map} In \cite{irie2022strong}, Irie stated a second conjecture, which claims that certain sequence of coefficients of the $U_{P_0}$ map do not vanish for ellipsoids. Our methods suffice to prove a related (but weaker) structure result of this type.

\begin{thm}\label{thm:U_general_ellipsoid}
    Let $E(a)$ be an irrational ellipsoid. There exists a sequence $k(i) \xrightarrow[i\rightarrow\infty]{}\infty$ such that 
    \begin{equation} \label{eqn:U_general_ellipsoid}
        \left<U_{P_0}x_{k(i) +n-1} ,x_{k(i)}\right>\neq 0 \qquad \text{and}\qquad \s_{x_{k(i)+n-1}}(\partial E(a))-\s_{x_{k(i)}}(\partial E(a))\xrightarrow[i\rightarrow\infty]{}0.
    \end{equation}
\end{thm}

\noindent This statement can compared to \cite[Conj. 5.1]{irie2022strong}\footnote{However, we warn the reader that there is a sizeable difference in our notation from \cite{irie2022strong}.}. Theorem \ref{thm:U_general_ellipsoid} implies Theorem \ref{thm:all_ellipsoids} in the irrational case, and also uses the approximation of irrational ellipsoids by rational ones, given in Lemma~\ref{lem:approx_ellipsoid}. 
\begin{proof}
We will show that, for any $\varepsilon>0$, there exists $k\in \N$ such that  
\begin{equation}\label{eq:U_coef_gen_E}
        \left<U_{P_0}x_{k+n-1},x_{k}\right>\neq 0 \qquad \text{and}\qquad \s_{x_{k+n-1}}(E(a))-\s_{x_k}(E(a))\leq\varepsilon.
    \end{equation}
Since $E(a)$ is irrational, the spectral invariants $\s_{x_k}(\partial E(a)) = M_k^a$ are strictly increasing in $k$. Therefore, for any fixed $k$, the difference in (\ref{eq:U_coef_gen_E}) is non-zero. This implies that we can construct a sequence of distinct $k(i)$ diverging to $\infty$ satisfying (\ref{eqn:U_general_ellipsoid}).

\vspace{3pt}
    
Thus, fix $\varepsilon>0$. Let $E(r)$ be the ellipsoid with rationally dependent entries from Lemma~\ref{lem:approx_ellipsoid}, and let $T$ be the period of the Reeb flow on $\partial E(r)$. Our proof consists of three steps. First, we study the cobordism map
\[\Phi: A(\partial E((1+\frac{\varepsilon}{T})\cdot{r})) \rightarrow A(\partial E({a})) \qquad\text{of the cobordism map}\qquad X = E((1+\frac{\varepsilon}{T})\cdot{r}))\setminus E({a})\]
Next, we use the information we derive on $\Phi$ to compare the $U_{P_0}$-maps for the two ellipsoids, apply Theorem~\ref{thm:U_integer_E} to $E((1+\frac{\varepsilon}{T})\cdot{r}))$ and deduce the non-vanishing of a certain coefficient in the $U_{P_0}$-map of $E(a)$. Finally, we compare the actions of the relevant orbits in both ellipsoids and conclude that the difference in (\ref{eq:U_coef_gen_E}) is bounded by $\varepsilon$.

\vspace{3pt}

We let $\{y_k\}_k$ denote the generators of $A(\partial E(r))$. Note that $A(\partial E((1+\frac{\varepsilon}{T})r)) = A(\partial E(r))$ where the action is scaled by $1+\frac{\varepsilon}{T}$. We will only use $\mathcal{A}(y_k)$ to denote the action with respect to $\partial E(r)$. We also let $k_T$ be the first index in which $T$ appears in the sequence $\{M_k^r\}_k$. 

\vspace{3pt}

\emph{Step 1 - Cobordism Map.} The goal of this step is to show that the image of $y_{k_T + n - 1}$ under $\Phi$ does not contain $x_1\cdot x_{k_T}$ for sufficiently small $\varepsilon$. That is
\begin{equation}\label{eq:no_gamma_1_factor}
    \left<\Phi(y_{k_T+n-1}), x_1 \cdot x_{k_T} \right> =0 \qquad\text{for sufficiently small }\varepsilon > 0.
\end{equation} 
Suppose otherwise. Then since the map $\Phi$ decreases action and $\mathcal{A}(x_1) = M^a_1 = a_1$, we have
\begin{equation} \label{eq:no_gamma_1_factor_2} (1 + \frac{\varepsilon}{T}) \cdot M^r_{k_T + n - 1} = (1 + \frac{\varepsilon}{T}) \cdot \mathcal{A}(y_{k_T + n - 1}) \ge \mathcal{A}(x_1) + \mathcal{A}(x_{k_T}) = a_1 + M^a_{k_T}.\end{equation}
On the other hand, by the definition of $k_T$ and Lemma~\ref{lem:approx_ellipsoid}, we have
\[M_{k_T+n-1}^r =M_{k_T}^r =T \qquad\text{and}\qquad M_k^r\leq M_k^a\leq (1+\frac{\varepsilon}{T})\cdot M_k^r \quad\text{for all }k.\]
We can then calculate that
\[(1+\frac{\varepsilon}{T}) \cdot M_{k_T+n-1}^r -M_{k_T}^a = (1+\frac{\varepsilon}{T})\cdot T - M_{k_T}^{a} \leq (1+\frac{\varepsilon}{T})\cdot T - M_{k_T}^{r} = (1+\frac{\varepsilon}{T})\cdot T -T = \varepsilon.\]
When $\varepsilon$ is smaller than $a_1$, this contradicts (\ref{eq:no_gamma_1_factor_2}) and thus proves (\ref{eq:no_gamma_1_factor}).

\vspace{3pt}

\emph{Step 2 - $U$-map.} The goal of this step is to prove that
\begin{equation}\label{eq:U_a_contains_gamma_k}
    \left< U_{P_0}(x_{k_T+n-1}),x_{k_T}\right>\neq 0. 
\end{equation}
To prove this, we consider the commutative diagram
\[\begin{tikzcd}
A(\partial E({(1+\frac{\varepsilon}{T})\cdot r})) \arrow[r, "U_{P_0}"] \arrow[d, "\Phi"]
&  A(\partial E({(1+\frac{\varepsilon}{T})\cdot r})) \arrow[d, "\Phi"] \\
A(\partial E({a})) \arrow[r, "U_{P_0}"]
& A(\partial E({a}))
\end{tikzcd}\]
We need three observations about the terms in this diagram. First, by Theorem~\ref{thm:U_integer_E} and Remark \ref{rmk:from_integer_E_to_rational}, we know that that 
    \begin{equation} \label{eq:U_non-zero}
       c_1 := \left<U_{P_0}(y_{k_T+n-1}), x_{k_T}\right> \neq 0.
    \end{equation}
Moreover, since $\Phi$ is a $\Z$-graded isomorphism of algebras that decreases word length (by Lemma~\ref{lem:Phi_inv_wl_decreasing}), we must have
    \begin{equation} \label{eqn:Phi_non_zero}
       c_2 := \left<\Phi(y_k), x_k\right> \neq 0,
    \end{equation}
Finally, by (\ref{eq:no_gamma_1_factor}) we know that $\Phi(y_{k_T + n - 1})$ has a zero $x_1\cdot x_{k_T}$-coefficient. In other words
    \begin{equation} \label{eqn:Phi_no_gamma_1_ideal}
       \left<\Phi(y_{k_T+n-1}), x_1 \cdot x_{k_T}\right> = 0.
    \end{equation}
Combining the equations (\ref{eq:U_non-zero}-\ref{eqn:Phi_non_zero}), we can thus conclude that
\[
\left<U_{P_0} \circ \Phi (y_{k_T+n-1}),x_{k_T}\right> 
= \left<\Phi \circ U_{P_0} (y_{k_T+n-1}),x_{k_T}\right>
= c_1\cdot \left<\Phi(y_{k_T}),x_{k_T}\right> = c_1\cdot c_2 \neq 0.
\]

We now claim that (\ref{eqn:Phi_no_gamma_1_ideal}) implies the desired formula (\ref{eq:U_a_contains_gamma_k}).  Indeed, consider the element
\[h := \left<\Phi (y_{k_T+n-1}),x_{k_T+n-1}\right> \cdot x_{k_T+n-1} - \Phi(y_{k_T+n-1})\]
Since $\Phi$ preserves grading and respects word length, $h$ is a linear combination of monomials $x_{i_1}\dots x_{i_k}$ of length $k \ge 2$. Furthermore, $\langle h,x_1 \cdot x_{k_T}\rangle = 0$ by (\ref{eqn:Phi_no_gamma_1_ideal}). Since $U_{P_0}$ satisfies the Leibniz rule (Lemma \ref{lem:Leibniz_rule}), the only terms of word-length greater than 1 that can be mapped to $x_{k_T}$ are of the form $x_j \cdot x_{k_T}$, for $j$ such that $U_{P_0}(x_j)$ is a multiple of the unit. Since the $U_{P_0}$ decreases degree by $2n-2$, the only elements that could be mapped to a constant have to be of degree $2n-2$ and hence a multiple of $x_1$. Over all, $x_1\cdot x_{k_T}$ is the only term of word-length greater than 1 that can be mapped to $x_{k_T}$. As a consequence, we conclude that $\left<U_{P_0} (h), x_{k_T}\right>=0$. Therefore, 
\begin{align*}
    \langle \Phi(y_{k_T + n - 1}),x_{k_T + n - 1}\rangle \cdot \left<U_{P_0} (x_{k_T+n-1}),x_{k_T}\right> &=  \left(\left<U_{P_0} \circ\Phi (y_{k_T+n-1}),x_{k_T}\right> +\left<U_{P_0} (h),x_{k_T}\right> \right) \\
    &= \left<U_{P_0} \circ\Phi (y_{k_T+n-1}),x_{k_T}\right> \neq 0.
\end{align*}
Since $\langle \Phi(y_{k_T + n - 1}),x_{k_T + n - 1}\rangle \neq 0$, this implies (\ref{eq:U_a_contains_gamma_k}).

\vspace{3pt}

\emph{Step 3 - action comparison.} Now choose $k = k_T$. We proved the $U$-map formula in (\ref{eq:U_coef_gen_E}) in Step 2, and we conclude by showing that the difference in (\ref{eq:U_coef_gen_E}) is bounded by $\varepsilon$. First, note that by definition
\begin{equation}\label{eq:gap_diff_E}
    \s_{x_{k+n-1}}(\partial E({a}))-   \s_{x_k}(\partial E({a})) =M_{k+n-1}^a -  M_{k}^a.
\end{equation}
Recalling that $M_k^r\leq M_k^a\leq (1+\frac{\varepsilon}{T})\cdot M_k^r$ for all $k$, we have
\begin{equation*}
  \s_{x_{k + n - 1}}(\partial E({a}))-   \s_{x_k}(\partial E({a})) \leq (1+\frac{\varepsilon}{T})M_{k+n-1}^r - M_{k}^r =  (1+\frac{\varepsilon}{T})\cdot T-T = \varepsilon.
\end{equation*}
This proves the desired bound, and concludes the proof. \end{proof}

\begin{remark}[Generalizations] \label{rem:sequences_of_U_rational_ellipsoids}
Theorem~\ref{thm:U_general_ellipsoid} can be extended to rational ellipsoids given a more general correspondence result than Proposition~\ref{prop:moduli_correspondence}, such as in \cite{sie_hig_19}. This is due to the fact that our assumption that the period of the orbits $\gamma_\pm$ is \emph{equal} to the minimal period of the flow, rather than is \emph{divisible} by this minimal period, was used only in Section~\ref{subsec:Umap_moduli_space_vs_point_moduli_space}. Given a more general correspondence result, one would be able to prove Theorem~\ref{thm:U_integer_E} for any $T$ that is divisible by $\lcm(m_1,\dots, m_n)$, and hence conclude the assertion of  Theorem~\ref{thm:U_general_ellipsoid} for rational ellipsoids. 
\end{remark}

\section{Closing lemma for a Reeb flow that is not nearly integrable.} \label{sec:non_toric_example} In this section, we apply the results of \S \ref{sec:spectral_gaps} and \S \ref{sec:moduli_space} to prove the strong closing property for the family of non-toric contact forms $\alpha_a$ discussed in \S \ref{subsec:strong_closing_for_non_toric}.

\subsection{Contact Manifolds With Torus Actions} \label{subsec:contact_manifolds_with_toric} In order to study our family of examples, we will need a few results about contact manifolds with torus actions.

\subsubsection{Lifting Torus Actions} Recall that a contact manifold $Y$ with contact form $\alpha$ is \emph{periodic} if the Reeb flow is periodic. In this case, the quotient $X = Y/S^1$ is a symplectic orbifold. 

\vspace{3pt}

Here we investigate the lifting of toric structures from the symplectic orbifold $X$ to $Y$ via the projection $\pi:Y \to X$. We require the following lemma.

\begin{lemma} \label{lem:lifting_Hamiltonians} Let $(Y,\alpha)$ be a periodic contact manifold with symplectic orbifold quotient $X$. Then the bijection between Reeb-invariant smooth functions on $Y$ and smooth functions on $X$ 
\[C^\infty(X;\R) \simeq C^\infty_R(Y;\R) \qquad\text{given by}\qquad H \mapsto \tilde{H} = H \circ \pi\]
has the following properties.
\begin{itemize}
    \item[(a)] (Vector-fields) The Hamiltonian vector-field $X_H$ of $H$ on $X$ and the contact Hamiltonian vector-field $V_{\tilde{H}}$ of $\tilde{H}$ are related by
    \[V_{\tilde{H}} = \tilde{H} \cdot R + (X_H)_\xi,\]
    where $(X_H)_\xi$ denotes the lift of $X_H$ to a vector-field tangent to $\xi = \on{ker}(\alpha)$. In particular
    \[\mathcal{L}_{V_{\tilde{H}}} \alpha = 0.\]
    \item[(b)] (Commutator) If $F$ and $H$ are two Hamiltonians on $X$, then
    \[[V_{\tilde{F}},V_{\tilde{H}}] = \{F,H\} \cdot R + [X_F,X_H]_\xi.\]
    \item[(c)] (Periodic) If $X_H$ is periodic with period $1$ and $H$ has a critical point $p$ with $H(p) \in\Z$, then $V_H$ is periodic of period $1$.
\end{itemize}
\end{lemma}

\begin{proof} For (a), we note that the contact Hamiltonian vector-field $V_{\tilde{H}}$ of $\tilde{H}$ satisfies
\[
\alpha(V_{\tilde{H}}) = \tilde{H} \qquad\text{and}\qquad d\alpha(V_{\tilde{H}},-) + d\tilde{H} = d\tilde{H}(R) \cdot \alpha = 0
\]
The first identity says that the $R$-component of $V_{\tilde{H}}$ is $\tilde{H} \cdot R$ and the second identity says that the $\xi$-component of $V_{\tilde{H}}$ is $X_H$. For (b), note that
\[\alpha([V_{\tilde{F}},V_{\tilde{H}}]) = V_{\tilde{F}}(\alpha(V_{\tilde{H}})) - ( \mathcal{L}_{V_{\tilde{F}}}\alpha)(V_{\tilde{H}}) = d\tilde{H}(V_{\tilde{F}}) = dH(X_F) = \{F,H\}\]
Moreover, we have
\[
d\alpha([X_F,X_H]_\xi,-) = d\alpha([V_{\tilde{F}},V_{\tilde{H}}],-) = -d\{F,H\}
\]
These two identities imply that the $R$-components and $\xi$-components of the two vector-fields in (b) agree. Finally, to show (c), let $H$ be a Hamiltonian on $X$ with $H(p) \in \Z$ at some critical point $p$, and suppose that $X_H$ is periodic of period $1$. The contact Hamiltonian vector field of any constant Hamiltonian $C$ is $C\cdot R$. Therefore, the flow generated by the constant Hamiltonian $C=H(p)\in \Z$ is 1-periodic. As a consequence, it is enough to show that the flow of $H' = H-H(p)$ is 1-periodic. Notice that $H'(p)=0$ and let $q$ be a non-critical point of $X_{H'}$. Choose an arc $\eta$ from $p$ to $q$ and let $f:D \to X$ be the disk acquired by flowing $\eta$ by $X_{H'}$ for time $1$. The area of this disk is
\[
\int_D \omega = -\int_\eta dH' = -(H'(q) - H'(p)) = -H'(q) \]
Thus the flow of the horizontal lift $(X_{H'})_\xi$ of $X_{H'}$ over $\gamma$ carries a point $y$ with $\pi(y) = q$ to the time $-H'(q)$ Reeb flow of $y$. On the other hand
\[V_{\tilde{H'}} = \tilde{H'} \cdot R + (X_{H'})_\xi\]
The flow of $\tilde{H'} \cdot R$ flows $y$ to the time $H'(q)$ Reeb flow of $y$. Therefore, the flow by $V_{\tilde{H'}}$ carries $y$ to itself after time $1$. \end{proof}

\begin{remark} \label{rmk:values_of_H} Note that if $H:X \to \R$ generates a circle action on a symplectic orbifold $(X,\omega)$ and $\omega$ is rational (i.e. $[\omega] \in H^2(X;\Q)$) then by the same calculation as in the proof of Lemma~\ref{lem:lifting_Hamiltonians}(c) above, we have
\[
H(q) \in \Q \qquad\text{for any critical point $q$ of $H$}
\]
This hypothesis holds, for instance, when $X = Y/S^1$ is the orbifold quotient of a contact manifold with periodic Reeb flow. \end{remark}

An immediate consequence of Lemma \ref{lem:lifting_Hamiltonians} is the following lifting property for torus actions.

\begin{lemma} \label{lem:torus_action_lift} Let $(Y,\alpha)$ be a closed periodic contact manifold with symplectic orbifold quotient $X$.
Then any Hamiltonian $\T^k$-action on $X$ lifts to a Hamiltonian $\T^k$-action on $Y$ preserving $\alpha$. Conversely, any Hamiltonian $\T^k$-action on $Y$ preserving $\alpha$ descends to a unique Hamiltonian action on $X$.
\end{lemma}

\begin{proof} Let $\mu$ be the moment map for a Hamiltonian $\T^k$-action on $X$ with components $\mu_1,\dots,\mu_k$. After translating the moment map, we can assume that $\mu = 0$ at some critical point $p$ of $\mu$.  Consider the map
\[\tilde{\mu} = \mu \circ \pi:Y \to \R^k\]
By Lemma \ref{lem:lifting_Hamiltonians}(c), the components $\tilde{\mu}_i$ each generate an $S^1$-action that commute by Lemma \ref{lem:lifting_Hamiltonians}(b). Thus $\tilde{\mu}$ is a moment map for a $\T^k$-action generated by contact Hamiltonians. The other direction of the lemma is clear. \end{proof}

We say that a contact $(2n+1)$-manifold $Y$ (respectively, a symplectic $2n$-manifold $X$) is \emph{integrable} if there is a Hamiltonian $\T^{2n+1}$-action on $Y$ (respectively, $\T^{2n}$-action on $X$) that is free on a dense open set. Lemma \ref{lem:torus_action_lift} implies the following.

\begin{cor} \label{cor:toric_Y_vs_YmodS1} The circle action generated by the Reeb flow on a periodic contact $(2n+1)$-manifold $Y$ extends to an integrable system on $Y$ if and only if the symplectic orbifold quotient $X = Y/S^1$ is integrable.
\end{cor}

\subsubsection{Reeb Flows On Spaces With $\T^2$-Actions} Let $(Y,\xi)$ be a contact manifold with contact form $\alpha = \alpha_1$ and Reeb vector-field $R = R_1$ generating an $S^1$-action. Let $X = Y/S^1$ be the symplectic orbifold quotient, and consider a Morse function
\[H:X \to \R\]
generating an $S^1$-action on $X$ and with minimum value $1$. Then the lift
\[R_2 = V_{\widetilde{H}} = \widetilde{H} \cdot R + (X_H)_\xi\]
is a Reeb vector-field with contact form $\alpha_2 = \alpha/\widetilde{H}$. In light of Lemma~\ref{lem:lifting_Hamiltonians}, $R_1$ and $R_2$ generate a $\T^2$-action by Reeb vector-fields on $(Y,\xi)$. For each $a = (a_1,a_2)$ with $a_i > 0$, we have a Reeb vector-field
\[R_a = a_1 R_1 + a_2 R_2 \quad\text{with contact form}\quad \alpha_a = \frac{\alpha}{a_1 + a_2\widetilde{H}}.\]
\noindent Note that our assumption that $H$ is Morse implies that the above $\T^2$ action is effective. 

There are essentially two dynamical cases for the Reeb vector-field $R_a$: the periodic case and the non-periodic case. In order to describe the periodic case, we require the following generalization of the least common multiple.

\begin{definition} \label{def:lcm_generalization}
    The \emph{(generalized) least common multiple} of two real numbers $s_1,s_2 > 0$ is given by
    \begin{equation*}
        \lcm(s_1,s_2):=\inf\left\{c\cdot q_1 \cdot q_2 \; : \;  s_1=c\cdot q_1 \text{ and } s_2=c\cdot q_2 \; \text{for }\; c\in\R \text{ and } q_1,q_2 \in \Z\right\},
    \end{equation*}
    Here the infimum of an empty set is defined to be $+\infty$. 
\end{definition}

Note that $\lcm(s_1,s_2)$ is finite if and only if $s_1$ and $s_2$ are rationally dependent. Moreover, if $a_1,a_2\in \N$ then this definition coincides with the standard least common multiple. 

\begin{lemma}\label{lem:rational_a_period_and_orbits} If $a_1$ and $a_2$ are rationally dependent (i.e. $a_1/a_2 \in \Q$), then $R_a$ is periodic of period
\[T = \on{lcm}(1/a_1,1/a_2)\]
Moreover, the Reeb flow is Morse-Bott and the simple closed orbits come in two types.
\begin{itemize}
        \item[(a)] Orbits of period $T$, coming in Morse-Bott families $S$ of dimension $\on{dim}(X) = \dim{Y}-1$.
        \item[(b)] Isolated orbits whose projection to $X$ is a critical point $x$ of $H$, with period $(a_1 + a_2 H(x))^{-1}$.
\end{itemize}
\end{lemma}

\begin{proof} We start by showing that $R_a$ is periodic of the claimed period. Let $\varphi^t$ denote the flow of $R_a$, which we can write as
\[\varphi^t = \varphi_1^{a_1t} \circ \varphi_2^{a_2t}\]
Here $\varphi_i^t$ is the flow of the Reeb vector-field $R_i$, for $i = 1,2$. We claim that $\varphi^L = \on{Id}$ if and only if $L = cq_1q_2$ for a real number $c > 0$ that satisfies
\[\frac{1}{a_1} = cq_1 \quad\text{and}\quad \frac{1}{a_2}=cq_2\]
Indeed, if $L = cq_1q_2$ is of the stated form, then
\[a_1L = cq_1q_2a_1 = q_2 \qquad a_2L = cq_1q_2a_2 = q_1\]
Therefore $\varphi_1^{a_1L} = \varphi_2^{a_2L} = \on{Id}$ since $\varphi_1$ and $\varphi_2$ are $1$-periodic. Conversely, recall that the $\T^2$ action generated by $R_1$ and $R_2$ is effective. Therefore, there exists a free orbit of the torus action on $Y$, and thus a point $p$ where
\[
\varphi_i^L(p) = p \quad\text{if and only if}\quad L \in \Z
\]
It follows that $\varphi^L(p) = p$ if and only if $a_1L$ and $a_2L$ are integers, i.e. if is of the form $L = cq_1q_2$ stated above. This confirms the claim, and it now follows from Definition \ref{def:lcm_generalization} that $\varphi^t$ is periodic with period $T = \on{lcm}(1/a_1,1/a_2)$. 

\vspace{3pt}

Next, we show that the orbits of $R_a$ come in the types discussed above. First, we write
\[
R_a = a_1R_1 + a_2R_2 = (a_1 + a_2 \cdot \widetilde{H})R_1 + a_2 (X_H)_\xi
\]
where $(X_H)_\xi$ is the  horizontal lift of the Hamiltonian vector-field of $H$ on $X$. Now consider a point $p \in Y$ and the orbit $\gamma$ through it. If $\pi(p) \in X$ is a critical point of $H$, then $X_H = 0$ and $\gamma$ is simply an orbit of $\left( a_1 + a_2\widetilde{H}(\pi(p))\right) R_1$, falling into type (b). If $\pi(p)$ is not a critical point of $H$, then the orbit of $p$ under the $\T^2$-action generated by $R_1$ and $R_2$ is free, denote this orbit by $\Sigma$. Under the identification $\Sigma \simeq \T^2$, $\gamma$ is simply a curve of rational slope. In particular, its  period is $\on{lcm}(1/a_1,1/a_2)$,
which is the period of $R_a$.\end{proof}

\begin{lemma} \label{lem:irrational_a_period_and_orbits} If $a_1$ and $a_2$ are rationally independent (i.e $a_1/a_2$ is irrational), then
\begin{itemize}
    \item[(a)] The simple orbits $\gamma$ of $\alpha_a$ each project to a critical point of $H$ on $X$.
    \item[(b)] Every Reeb orbit $\gamma$ is non-degenerate, elliptic and has even SFT grading i.e.
    \[|\gamma| = 0 \mod 2.\]
\end{itemize}
\end{lemma}

\begin{proof} (a) follows from an identical analysis to the proof of Lemma \ref{lem:rational_a_period_and_orbits}. To see (b), choose a critical point $x$ of $H$ and let $\gamma$ be a closed Reeb orbit with $\pi(\gamma) = x$. The period of $\gamma$ is given by
\[
\frac{k}{a_1 + a_2 \cdot H(x)} \quad\text{for}\quad k \in \Z_+.
\]
Assume, first, that $X$ is a manifold at $x$ (i.e. an orbifold with no isotropy). Given a point $p \in \gamma$, the contact structure $\xi_p$ is identified with $T_xX$ via the differential $d\pi:\xi \to TX$. Moreover, the restriction $d\varphi^t_{R_1}|_\xi$ of the linearized Reeb flow of $R_a$ to $\xi$ can be identified with the linearized Hamiltonian flow with a time reparametrization, $d\varphi_H^{a_2\cdot t}$. Therefore, the linearized Reeb flow at time $\frac{k}{a_1 + a_2 \cdot H(x)}$ can be identified with the map
\[
d\varphi_H^{L}(x):T_xX \to T_xX \quad\text{where}\quad L = \frac{ka_2}{a_1 + a_2 \cdot H(x)} = \frac{k}{a_1/a_2 + H(x)}.
\]
Since $H$ is autonomous and its flow $\varphi^t_H$ is $1$-periodic, the matrices $A(t) = d\varphi_H^t(x)$ for $t \in \R$ form a $1$-periodic subgroup of the symplectic automorphism group $\on{Sp}(T_pX)$. Arguing similarly to the proof of Lemma~\ref{lem:RS_parity} (i.e., pick a compatible Riemannian metric, average with respect to the path $A(t)$ and obtain an invariant compatible metric) we see that after a change of basis, $A$ is identified with 
\[A(t) = \left[\begin{array}{cccc}
e^{2\pi im_1 t} & 0 & \dots & 0\\
0 & e^{2\pi im_2 t} & \dots & 0\\
0 & 0 & \dots & 0\\
0 & 0 & \dots & e^{2\pi i m_{n-1} t}\\
\end{array}\right] \qquad\text{for rationally independent integers }m_i \in \Z\]
It is an easy consequence of Remark \ref{rmk:values_of_H} and the rational independence of $a_1$ and $a_2$ that $L \in \R \setminus \Q$. This implies that $A(L)$ (and thus the linearized Reeb flow along $\gamma$ and $\gamma$ itself) is non-degenerate. The Conley-Zehnder index $\mod 2$ is then
\[
\CZ(\gamma) = \CZ(A(L)) = \sum_i \CZ(e^{2\pi i m_i L}) = n - 1 \mod 2
\]
Thus we find that the SFT grading is
\[(n-3) + \CZ(\gamma) = 2n - 4 = 0 \mod 2 \qedhere\]\end{proof}

The above lemma implies that, when $a_1/a_2$ is irrational, the contact dg-algebra of $Y$ with respect to the contact form $\alpha_a$ has no differential. Thus we can conclude the following corollary.

\begin{cor} The contact homology $CH(Y,\xi)$ satisfies
\[
CH(Y,\xi) \simeq \Lambda(\bigoplus_{i=1}^\infty H(X;\Q))
\]
\end{cor}

\subsubsection{Non-Degenerate Perturbation} We can alternately view the non-degenerate contact forms $\alpha_a$ as a perturbation of the Morse-Bott contact forms of the same form. More precisely, choose $r = (r_1,r_2)$ with $r_1/r_2$ rational. Choose $\epsilon > 0$ and let
\begin{equation}\label{eq:Morse-Bott_function_toric}
    f:Y \to \R \qquad\text{be given by}\qquad f = \log\Big(\frac{r_1 + r_2\widetilde{H}}{r_1 + (r_2 + \epsilon)\widetilde{H}}\Big)
\end{equation}

\begin{lemma} \label{lem:function_f} The function $f:Y \to \R$ has the following properties.
\begin{itemize}
    \item[(a)] $f$ converges to $0$ as $\epsilon \to 0$. 
    \item[(b)] For every $b=(b_1,b_2)$, $f$ is invariant under the Reeb flow $R_b$.
    \item[(c)] $f$ is Morse-Bott with $1$-dimensional families of critical points.
    \item[(d)] $e^f\alpha_r = \alpha_a$ where $a = (r_1,r_2 + \epsilon)$. 
\end{itemize}
\end{lemma}

\begin{proof} Properties (a) and (d) are evident from the formula for $f$. To show that properties (b) and (c) hold, we start by noticing that $f = g \circ \widetilde{H}$ where
\[g(s) = \frac{r_1 + r_2\cdot s}{r_1 + (r_2 + \epsilon)s}.\]
Since $\widetilde{H}$ is invariant under the flow of $R_b$ for all $b=(b_1,b_2)$ (see Section~\ref{subsec:contact_manifolds_with_toric}), $f=g\circ \widetilde{H}$ is invariant as well, which proves (b). To prove (c), first notice that $\widetilde{H}$ is Morse-Bott with $1$-dimensional families of critical points, as a lift of a Morse function to an $S^1$ bundle. Now, the identity $f=g\circ \widetilde{H}$, together with the fact that  $g'(s) < 0$ when $s > 0$, implies that $f$ is Morse-Bott with $1$-dimensional families of critical points as well, which proves (c).
\end{proof}

\subsubsection{Periodic Approximation} Let the contact manifold $(Y,\xi)$ and the contact forms $\alpha_a$ be as in the previous sub-section. We now prove the following approximation property.

\begin{prop}\label{prop:approx_torus_action_forms}
For any $a_1, a_2>0$ and any $\delta >0$ there exist rationally dependent $r_1,r_2>0$ such that
\[\alpha_r\leq \alpha_a \leq (1+\delta/T)\alpha_r\]
where $T = \on{lcm}(r_1^{-1},r_2^{-1})$ is the period of the Reeb flow of $\alpha_r$. 
\end{prop}

This approximation property will play the roll of Lemma \ref{lem:approx_ellipsoid} in our proof of Theorem \ref{thm:non_toric_ex_intro}. We require the following arithmetic fact, which follows from an argument similar to Lemma~\ref{lem:approx_ellipsoid}.

\begin{lemma}\label{lem:Dirichlet_torus_action_forms_new}
For any pair of positive numbers $a_1,a_2>0$ and any $\delta >0$ small, there exist $r_1,r_2>0$ with $T:=\lcm(r_1^{-1}, r_2^{-1})<\infty$ such that
\begin{equation}
     \frac{r_i}{1 + \delta/T} \leq a_i\leq r_i.
\end{equation}
    \end{lemma}
\begin{proof}
    By the simultaneous version of Dirichlet approximation theorem, for each $N>0$ there exist non-zero integers $q_1, q_2$ and $p<N$ such that 
    $$
    \Big|a_i - \frac{q_i}{p}\Big|<\frac{1}{p\cdot \sqrt{N}}, \quad \text{or} \quad |p-\frac{q_i}{a_i}|<\frac{1}{\sqrt{N}a_i}.
    $$
    Choose $N$ and $r_i$ for $i = 1,2$ so that
    \[\delta /2 > \max{\frac{1}{\sqrt{N} a_1},\frac{1}{\sqrt{N}a_2}} \quad\text{and}\quad r_i = \frac{q_i}{p - \delta/2}\]
    Given these choices, $\delta$ and $r_i$ satisfy the following inequalities.
    \[|\frac{p}{q_i} - \frac{1}{a_i}| < \frac{1}{\sqrt{N} q_i a_i} \le  \frac{\delta}{2q_i} \quad\text{and thus}\quad \frac{1}{r_i} = \frac{p - \delta/2}{q_i} \le \frac{1}{a_i}\]
    The second inequality yields $a_i \le r_i$. Moreover, note that $r_1$ and $r_2$ satisfy
    \[\frac{1}{r_1} = (\frac{p-\delta/2}{q_1q_2}) \cdot q_1 \quad\text{and}\quad \frac{1}{r_2} = (\frac{p-\delta/2}{q_1q_2}) \cdot q_2\]
    By Definition \ref{def:lcm_generalization}, this implies that $T:=\lcm(r_1^{-1}, r_2^{-1})\leq p-\delta/2$. Finally, we note that
    $$
    r_i = \frac{q_i}{p - \delta/2} = \frac{q_i}{p + \delta/2} \cdot \frac{p+\delta/2}{p - \delta/2} \le a_i \cdot (1 + \frac{\delta}{p - \delta/2}) \le a_i \cdot (1 + \delta/T) \qedhere $$
\end{proof}

\begin{proof}[Proof of Proposition~\ref{prop:approx_torus_action_forms}]
Given $a_1, a_2$ and $\delta>0$ let $r_1, r_2$ be the numbers from Lemma~\ref{lem:Dirichlet_torus_action_forms_new}. By Lemma~\ref{lem:rational_a_period_and_orbits}, the Reeb flow induced by $\alpha_r$ is periodic of minimal period $T:=\lcm(r_1^{-1}, r_2^{-1})$. By the assertion of Lemma~\ref{lem:Dirichlet_torus_action_forms_new}
\[\frac{r_i}{1 + \delta/T} \leq a_i\leq r_i.\]
This implies the required inequality between the contact forms:
\begin{align*}
    \alpha_r &= \frac{\alpha}{r_1 + r_2\widetilde{H}} \leq \frac{\alpha}{a_1 + a_2\widetilde{H}} = \alpha_a \le (1 + \delta/T) \cdot \frac{\alpha}{r_1 + r_2\widetilde{H}} = (1 + \delta/T) \cdot \alpha_r.\qedhere
\end{align*}
\end{proof}

\subsection{A family of non-integrable contact forms}\label{subsec:non-toric_ctct_forms}
In this section, we review the construction of the family of contact manifolds discussed in \S \ref{subsec:strong_closing_for_non_toric}, filling in some details that were omitted in the introduction.

\subsubsection{Preliminaries} Let us first, briefly, recall the discussion in \S \ref{subsec:strong_closing_for_non_toric}. Consider the prequantization space
\[E \to  \C P^1 \times \C P^1 \quad\text{with}\quad c_1(E) = A_1 + A_2\]
Here $A_1,A_2$ denote the cohomology classes Poincare dual to $[\C P^1 \times \on{pt}]$ and $[\on{pt} \times \C P^1]$. We let $\xi$ denote the contact structure, $\alpha_1$ denote the standard prequantization contact form and $R_1$ to denote its Reeb vector-field, which generates a $1$-periodic circle action on $E$. We next let $H$ denote the Hamiltonian
\[
H:X \to \R \quad\text{given by}\quad H(x,y) = \frac{\pi}{2}\Big(\frac{|x|^2}{(1 + |x|^2)} + \frac{|y|^2}{(1 + |y|^2)}\Big) + 1 \quad\text{on}\quad \C \times \C \subset \C P^1 \times \C P^1.
\]
By Lemma \ref{lem:lifting_Hamiltonians}, this action lifts to a free circle action on $E$ generated the Reeb vector-field 
\[R_2 = \widetilde{H} \cdot R_1 + (X_H)_\xi \quad \text{for the contact form}\quad \alpha_2 = \frac{\alpha}{\widetilde{H}}\]
Moreover, $R_2$ commutes with the Reeb vector-field $R_1$ of $\alpha$. This yields a contact $\T^2$-action
\[
\T^2 \curvearrowright (E,\xi) \quad\text{generated by}\quad R_1,R_2.
\]

Next, we consider the $\Z_4$-action on $X$ generated by the following $4$-periodic map.
\[
\C P^1 \times \C P^1 \to \C P^1 \times \C P^1 \quad\text{with}\quad \psi(x,y) = (\jmath(y),x)
\]
Here $\jmath:\C P^1 \to \C P^1$ denotes the antipode map $[u:v] \mapsto [-u:v]$. This $\Z_4$-action lifts to a free action of $\Z_4$ on $E$ as follows. We may view $E$ as the unit circle bundle of the tensor product
\[L_1 \otimes L_2 \to  \C P^1 \times \C P^1\]
where $L_i$ is the pullback (under projection to the $i$th factor) of the $O(1)$ line bundle over $\C P^1$. The antipode lifts to a unitary map $\jmath^*:L \to L$ since it descends from a unitary map of $\C^2$. We now define a lift of $f$ to $L_1 \otimes L_2$ by
\[
\Psi:L_1 \otimes L_2 \to L_1 \otimes L_2 \qquad \text{with}\qquad (x,y;u \otimes v) \mapsto (j(y),x;i \cdot j^*v \otimes u)
\]
Here $i$ denotes complex multiplication by $i$. $\Psi$ is unitary, and thus restricts to $E$. 
\begin{remark}\label{rem:Z4_action} $\Psi$ generates a free $\Z_4$-action on $E$ that preserves $\alpha$ and $\alpha_2$, and descends to the $\Z_4$-action generated by $\psi$ on $ \C P^1 \times \C P^1$.
\end{remark} 

\subsubsection{Construction Of Family} We can now define the contact manifold
\[Y := E/\Z_4\]
We may view $Y$ as a prequantization space over the symplectic orbifold $X =  \C P^1 \times \C P^1/\Z_4$ with a Hamiltonian $S^1$-action induced by $H$. By Remark \ref{rem:Z4_action}, the contact structure $\xi$, contact forms $\alpha_1,\alpha_2$ and Reeb vector-fields $R_1,R_2$ all descend to $Y$. We continue to denote them by $\xi,\alpha_1,\alpha_2,R_1$ and $R_2$, and we let
\[R_a = a_1 \cdot R_1 + a_2 \cdot R_2 \quad\text{with contact forms}\quad \alpha_a\]

\begin{lemma} The Reeb vector-fields $R_1$ and $R_2$ commute and generate flows of period $1$. In particular, they generate a contact $\T^2$-action on $Y$.
\end{lemma} 

\begin{proof} The fact that $[R_1,R_2] = 0$ is local and so follows from the same fact on $E$.

\vspace{3pt}

To see that $R_1$ is periodic of period $1$, note first that it is periodic of period at most $1$ since it lifts to a periodic vector-field on $E$ of period $1$. To see that the period is equal to $1$, note that a simple $R_a$-orbit $\gamma$ in $E$ that projects to a point of $\C P^1 \times \C P^1$ on which $\Z_4$ acts freely will project to a period $1$ simple orbit $\gamma'$ of $R_a$ on $Y$, via the quotient map $E \to Y$. 

\vspace{3pt}

The vector field $R_2$ is the lift of the period $1$ Hamiltonian flow generated by the Hamiltonian $H$ on the orbifold $X = \C P^1 \times \C P^1/\Z_4$, described in Lemma \ref{lem:topological_prop_non-toric}(c). In particular, it is period $1$.
\end{proof}

We will also consider the following specific hypersurface in $Y$.
\begin{definition}[diagonal hypersurface]\label{def:diagonal_hypersurface}
     Let $Y_3\subset Y$ be the codim-2 hypersurface defined as follows. Let $\Delta\simeq \C P^1 $ be the diagonal in $ \C P^1 \times \C P^1$, and let  $E_3$ be its inverse image under the projection $E \to \C P^1 \times \C P^1$. Then $Y_3$ is the quotient of $E_3$ by $\Z_4$.
\end{definition}

\vspace{3pt}

We will require a number of topological properties of $Y$ for our proof. These properties will be used to guarantee that the CZ index is defined over $\Z$, as well as the vanishing of certain homological intersections, see Remark~\ref{rem:non-toric}. We record these properties below.

\begin{lemma}\label{lem:topological_prop_non-toric}
The contact manifold $(Y,\xi)$ has the following topological properties. 
\begin{itemize}
    \item[(a)]\label{itm:chern_is_torsion} The Chern class $c_1(\xi)$ is torsion in $H^2(Y;\Z)$.
    \item[(b)]\label{itm:Y3_null_homologous} $[Y_3] = 0 \in H_3(Y;\Q)$. In particular, the intersection pairing of any $A \in H_2(Y;\Q)$ with $[Y_3]$ is $0$. 
    \item[(c)]\label{itm:H2_of_Y3_zero} $H_2(Y_3;\Q) = 0$.
\end{itemize}
\end{lemma}
\begin{proof} 
To prove (a), consider the contact structure $\xi_E$ on $E$. We know that
\[c_1(\xi_E) = \pi^*c_1(\C P^1 \times \C P^1) = c \cdot c_1(E)\]
for a constant $c$. Now consider the Gysin sequence of the bundle $E \to \C P^1 \times \C P^1$. 
\begin{equation*}
       \cdots\rightarrow H^0(\C P^1\times \C P^1)\xrightarrow{c_1(E) \cup} H^2(\C P^1\times \C P^1) \xrightarrow{\pi^*} H^2(E)\xrightarrow{\pi_*} H^1(\C P^1 \times \C P^1) \rightarrow \cdots
\end{equation*}
Since $c_1(E)$ and $c_1(\C P^1 \times \C P^1)$ are proportional, $c_1(\C P^1 \times \C P^1)$ is in the image of the map $c_1(E) \cup$ and kernel of $\pi^*$. Thus $c_1(\xi_E) = \pi^*c_1(\C P^1 \times \C P^1) = 0$. Now since $\Z_4$ acts freely on $E$, we know that pullback by the quotient map $\rho:E \to Y$ yields an isomorphism
\[
\rho^*:H^2(Y;\Q) \simeq (H^2(E;\Q))^{\Z_4}
\]
where the latter is the sub-space of $\Z_4$-invariant classes. Since $\rho^*c_1(\xi) = c_1(\xi_E) = 0$, this implies that $c_1(\xi)$ is torsion.

\vspace{3pt}

To prove (b), recall the construction of  $Y_3$ from Definition~\ref{def:diagonal_hypersurface}. 
Consider the Gysin sequence
\[
H_4(\C P^1 \times \C P^1;\Q) \xrightarrow{c_1(E) \cap} H_2(\C P^1 \times \C P^1;\Q) \xrightarrow{\pi^*} H_3(E;\Q) \xrightarrow{\pi_*} H_3(\C P^1 \times \C P^1;\Q)
\]
Now observe that $[\Delta] = \on{PD}(c_1(E)) = c_1(E) \cap [\C P^1 \times \C P^1]$ and $[E_3] = \pi^*[\Delta]$. Thus by exactness of the Gysin sequence, we have
\[
[E_3] = \pi^*[\Delta] = \pi^*(c_1(E) \cap [\C P^1 \times \C P^1]) = 0
\]
This prove that $[E_3] = 0$. Since $[Y_3] = \rho_*[E_3]$ where $\rho:E \to Y$ is the quotient by $\Z_4$, this proves (b).

\vspace{3pt}

Finally, we prove (c). Note that by Poincare duality and the isomorphism of quotient cohomology with co-invariants, we have
\[H_2(Y_3;\Q) \simeq H^1(Y_3;\Q) \simeq (H^1(E_3;\Q))^{\Z_4} \subset H^1(E_3;\Q)\]
It suffices to show that $H^1(E_3;\Q) = 0$. If we consider the Gysin sequence applies to $E_3 \to \Delta \simeq \C P^1$, we have
\[
H^1(\C P^1;\Q) \xrightarrow{\pi^*} H^1(E_3;\Q) \xrightarrow{\pi_*} H^0(\C P^1;\Q) \xrightarrow{c_1(E_3) \cup} H^2(\C P^1;\Q)
\]
Since $H^1(\C P^1;\Q) = 0$ and $c_1(E_3) \cup$ is injective, this implies that $H^1(E_3;\Q) = 0$. \end{proof}

We let (by some abuse of notation) $\alpha_1$ and $\alpha_2$ denote the contact forms descending from $\alpha$ and $\alpha_2$ on $E$, and we denote the corresonding (commuting) Reeb vectors by $R_1$ and $R_2$, respectively. Both $R_1$ and $R_2$ are periodic.

\begin{lemma} $Y$ does not admit a contact integrable structure extending the circle action generated by the Reeb vector-field $R_1$ of $\alpha$.
\end{lemma}

\begin{proof} It is proven in \cite{singer1999nontoric} that $B$ is not toric, so this is immediate from Corollary \ref{cor:toric_Y_vs_YmodS1}.
\end{proof}

\subsection{Proof of Theorem~\ref{thm:non_toric_ex_intro}.}
After studying the structure and dynamics of the contact manifold $Y$ and the Reeb flows $R_a$, we are finally ready to prove that it has the strong closing property. Our first step towards a proof of Theorem~\ref{thm:non_toric_ex_intro} is to show that the assumptions of Proposition~\ref{prop:moduli_space_count} hold for $(Y,\alpha_r)$ discussed above, when $r=(r_1,r_2)$ are rationally dependent. These assumptions were stated in Setup~\ref{set:intersection_theory}. 
\begin{lemma}\label{lem:setup_for_nontoric}
The contact manifold $(Y,\alpha_r)$ for rationally dependent $r=(r_1,r_2)$ satisfies the assumptions of Setup~\ref{set:intersection_theory}.
\end{lemma}
\begin{proof}
In what follows we go over each of the assumptions, in the same order as stated in Setup~\ref{set:intersection_theory}, and explain why they hold for $(Y,\alpha_r)$. We use the notations from the previous subsections.
\begin{enumerate}
    \item $Y$ is a contact manifold of dimension 5: See Section~\ref{subsec:non-toric_ctct_forms}.
    \item The flow of $R_r$ is periodic: This is proved in Lemma~\ref{lem:rational_a_period_and_orbits}.
    \item An almost complex structure on $\xi$: Let us describe the complex structure on the contact structure of the prequantization bundle $E\rightarrow \C P^1\times \C P^1$. It will be invariant under the $Z_4$ action any thus descend to the contact structure $\xi$ on $Y$.\\
    Since $E$ is a prequantization bundle, its contact structure is canonically isomorphic to the tangent bundle of the base, $T(\C P^1\times \C P^1)$. Since the latter is a complex vector bundle, this defines an almost complex structure on the contact structure of $E$ which descends to $\xi$. We denote this almost complex structure by $J$.
\item An invariant Morse--Bott function: Recall the function from Section~\ref{subsec:non-toric_ctct_forms}:
$$H:\C P^1\times \C P^1 \rightarrow \R, \qquad H([x, y]):= \frac{\pi}{2}\left(\frac{|x|^2}{(1 + |x|^2)} + \frac{|y|^2}{(1 + |y|^2)}\right)+1.$$
The critical points of the latter function on $\C P^1\times \C P^1 $ are $[0, 0]$, $[\infty, \infty]$, $[\infty, 0]$
and $[0, \infty]$, which are the non-free points of the $\Z_4$ action. After the quotient by the $\Z_4$ action, these descend to three classes: $[0, 0]$, $[\infty, \infty]$ and $[\infty, 0]\sim [0, \infty]$. Denote by $\widetilde{H}:Y\rightarrow \R$ the function obtained by lifting $H$ to $E$ and projecting to the quotient by the $\Z_4$ action.
The perturbing function for us is the function $f:Y\rightarrow \R$ defined using $\widetilde{H}$ by (\ref{eq:Morse-Bott_function_toric}).
By Lemma~\ref{lem:function_f}, the function $f$ is Morse--Bott, invariant under the flow and its critical manifolds are circles.  
Note that  Lemma~\ref{lem:function_f} also guarantees that $e^f\alpha_r = \alpha_a$ for $a=(r_1, r_2+\epsilon)$.
\item No assumption to check.
\item No assumption to check. We take $\gamma_\pm$ to be orbits of $R_r$ of period $T$, corresponding to the minimum and maximum of $f$. Here $T=\lcm\{r_1^{-1}, r_2^{-1}\}$ is the period of $R_r$. A simple calculation shows that the maximum of $f$ (which is the minimum of $\tilde H$) lies over $[0,0]$, and the minimum of $f$ (which is the maximum of $\widetilde{H}$) lies over $[\infty,\infty]$. 
\item Point constraint: We need to choose a point $z$ in the intersection of the ascending manifold of $\gamma_-$ and the descending manifold of $\gamma_+$. The flow lines of $f$ coincide with those of $\tilde H$ up to (orientation reversing) reparametrization. Consider the $\tilde H$ flow line given by $t \mapsto [e^t x_0, e^t x_0, 1]$, where $x_0 \neq 0$ is a fixed element of $\C P^1$ and by $1$ in the last coordinate we mean $1$ in a trivialisation of the bundle $E_{\{[e^t x, e^t x]\}}$. The image of this flow line connects $\gamma_-$ and $\gamma_+$. Let us pick $z = ([x_0,x_0], 1)$.
\item Contact flag: Since $Y$ is of dimension 5, we only need to define the codimension $2$ submanifold of $Y$. Recall Definition~\ref{def:diagonal_hypersurface} of the hypersurface $Y_3 \subset Y$.
\begin{enumerate}[label=(\alph*)]
    \item As explained in Remark~\ref{rem:non-toric}, the assumption that $H_2(Y_{2j-1})=0$ can be replaced by the weaker assumptions that the CZ index is defined over $\Z$ and that the homological intersection number of any 2-torus in $Y_{2j+1}$ with $Y_{2j-1}$ is zero. These are asserted by Lemma~\ref{lem:topological_prop_non-toric}. More explicitly, item (a) of the lemma guarantee that the relative CZ index of any  orbit is well defined over $\Z$, while (b) and (c) guarantee that the intersection number of any 2-torus in $Y$ with $Y_3$ (resp., in $Y_3$ with any cylinder) is zero.
    
    \item $Y_3$ is invariant under the Reeb flow: Since $R_r=r_1R_1+r_2R_2$, it sufficient to show that $Y_3$ is invariant under both $R_1$ and $R_2$. Since $Y_3$ is a quotient of a restriction of the prequantization bundle, it is invariant under the flow $R_1$. $R_2$ is a linear combination of $R_1$ and $(X_H)_\xi$, which is the horizontal lift of the symplectic gradient of the function $H$ defined above. Since this function is symmetric under replacing the two factors of $\C P^1\times \C P^1$, its Hamiltonian flow preserves the diagonal $\Delta$. As a conserquence, the quotient of the horizontal lift preserves $Y_3$.
    \item $(Y_3, \xi|_{Y_3})$ is a contact manifold as is coincides with the $\Z_4$ quotient of the unit circle bundle $E|_{\Delta} \to \Delta$. For the same reason, $\alpha_{TY_3}$ is a contact form on $Y_3$.
    
    The diagonal $\Delta$ is invariant under the action of $i$ as $i$ acts simultaneously on each factor. This makes the contact planes $\xi|_{Y_3}$ also $J$-invariant.
    \item $\nabla_Jf$ is tangent to $Y_3$: Since $f$ is a composition of a 1-variable function on $\tilde H$, it is enough to check that the gradient of $\tilde H$ is tangent to $Y_3$. Note that the metric defined by $J$ the $d\alpha$ is the standard metric on $\C P^1 \times \C P^1$, lifted to $E$. With respect to this metric,
    $\nabla \tilde H = (x,y,0)$ at any $([x,y], \theta) \in Y$. So, for any point $([x,x], \theta) \in Y_3$, $\nabla f = (x,x,0)$ and so is tangent to $Y_3$.
    
    The function $f|_{Y_3}$ is Morse-Bott as $f$ is Morse on the diagonal. 
    \item Conditions on Hessian of $f$: $Y_3$ contains two of the critical points, namely, $[0,0]$, and $[\infty, \infty]$ corresponding to $\gamma_+$ and $\gamma_-$. So apart from $\gamma_+$ the only critical point in $Y_3$ is $\gamma_-$, which is a minimum and hence the Hessian is positive definite (and so is any restriction of it).
    \item $z = ([x_0,x_0], 1) \in Y_3$.\qedhere
    \end{enumerate}
\end{enumerate}
\end{proof}

Our next goal is to prove the vanishing of the spectral gap for the periodic flows corresponding to rationally dependent $(a_1,a_2)$.
\begin{thm}\label{thm:gap0_rational_non-toric}
Suppose  $r=(r_1,r_2)$ are rationally dependent and denote by $T=\lcm(r_1^{-1}, r_2^{-1})$ the period of $R_r$. There exists a class $\sigma\in CH(Y)$ with $\s_\sigma(Y,\alpha_r) = T$ and $\gap_\sigma(Y,\alpha_r)=0$.
\end{thm}
\begin{proof}
We start by describing the contact homology of $Y$. Consider the perturbation $\alpha_f:= e^{f}\alpha_r$ for $f$ defined using $\widetilde{H}$, and given a positive parameter $\epsilon>0$, by (\ref{eq:Morse-Bott_function_toric}). As explained in Lemma~\ref{lem:function_f}, $\alpha_f:= e^f\alpha_r = \alpha_a$ for $a=(a_1,a_2):=(r_1, r_2+\epsilon)$. Therefore, for a generic choice of  $\epsilon>0$, $\alpha_f$ is non-degenerate. As explained in \cite[\S 1.8]{p2015},  contact homology has a relative grading $\Z/2c_1(\xi)\cdot H_2(Y)$, which descends to an absolute $\Z_2$ grading.
Lemma~\ref{lem:irrational_a_period_and_orbits} guarantees that the $\Z_2$ grading of all  of the orbits of $\alpha_f = \alpha_a$ is zero, and hence the differential of the contact dga with respect to $\alpha_f$ vanishes. Therefore, we identify the classes in $CH(Y)$ with the periodic orbits on $\alpha_a$. 
Lemma~\ref{lem:setup_for_nontoric} states that the assumptions of Proposition~\ref{prop:moduli_space_count}, as stated in Setup~\ref{set:intersection_theory}, hold for $(Y,R_r)$ and the perturbing function $f$. Arguing as in the proof of Theorem~\ref{thm:U_integer_E}, Proposition~\ref{prop:moduli_space_count} together with Proposition~\ref{prop:moduli_correspondence}, imply that  $\left<U_{P_0} \gamma_+, \gamma_-\right>=1\mod 2$, where $P_0$ is the tangency constraint from Example~\ref{exa:tangency_constraints}. We remind that here $\gamma_\pm$ are periodic orbits of $R_r$ of period $T=\lcm(r_1^{-1}, r_2^{-1})$, lying over the minimum and the maximum of $f$.  Therefore, we may consider the class  $\sigma := [\gamma_+]\in CH(Y)$ and compute its spectral gap:
\begin{align*}
    \gap_\sigma(Y,R_r)&:= \s_{\sigma}(Y, \alpha_r) - \s_{U_{P_0}\sigma}(Y, \alpha_r) \leq \s_{[\gamma_+]}(Y, \alpha_r) - \s_{[\gamma_-]}(Y, \alpha_r) \\
    &= \lim_{\epsilon\rightarrow 0}\int_{\gamma_+} e^f\alpha_r - \int_{\gamma_-}e^f\alpha_r = \int_{\gamma_+}\alpha_r -\int_{\gamma_+}\alpha_r= T-T = 0,
\end{align*}
where we used the fact that $f\xrightarrow{\epsilon\rightarrow 0}0$ and $T:=\lcm(r_1^{-1},r_2^{-1})$ is the period of $R_r$, which coincides with the periods of $\gamma_\pm$ by our choice of these orbits.
\end{proof}

Using Theorem~\ref{thm:gap0_rational_non-toric} and Proposition~\ref{prop:approx_torus_action_forms} about approximations of irrational forms by rational ones, we can conclude the vanishing of the spectral gap for $R_a$ for all $a=(a_1,a_2)$.
\begin{cor}
For any $a=(a_1,a_2)$, $\gap(Y,\alpha_a)=0$. In particular, it follows from Theorem~\ref{thm:spectral_gap_to_closing_lemma_intro} that $(Y,\alpha_a)$ satisfies the strong closing property.
\end{cor}
\begin{proof}
By Proposition~\ref{prop:limit_of_gap0}, to conclude that $\gap(Y,\alpha_a)=0$ it is sufficient to show that there exists a sequence $\alpha_i$ and classes $\sigma_i\in CH(Y)$ such that $\alpha_i\leq \alpha\leq (1+\epsilon_i)\alpha_i$, $\gap_{\sigma_i}(Y,\alpha_i)=0$ and $\epsilon_i\cdot\s_{\sigma_i}(Y,\alpha_i)\rightarrow 0$. Applying Proposition~\ref{prop:approx_torus_action_forms} for $\delta = \frac{1}{i}$ we conclude that there exist $r^{(i)}= (r^{(i)}_1,r^{(i)}_2)$ rationally dependent such that $\alpha_i:=\alpha_{r^{(i)}}$ satisfies $\alpha_i\leq \alpha_a\leq\left(1+\frac{1}{i\cdot T_i}\right)\alpha_i$, where $T_i:= \lcm(r^{(i)}_1,r^{(i)}_2)$. Applying Theorem~\ref{thm:gap0_rational_non-toric} to $\alpha_i$ we obtain classes $\sigma_i\in CH(Y)$ with $\s_{\sigma_i}(Y,\alpha_i)=T_i$ and $\gap_{\sigma_i}(Y,\alpha_i)=0$. Since  $\frac{1}{i\cdot T_i}\cdot T_i = \frac{1}{i}\rightarrow 0$, we conclude that the assumptions of  Proposition~\ref{prop:limit_of_gap0} hold and $\gap(Y,\alpha_a)=0$.
\end{proof}
\vspace{3pt}

\bibliographystyle{alpha}
\bibliography{rsft.bib}
\end{document}